\documentclass[11pt,a4paper,reqno]{amsart}

\usepackage{a4wide}
\usepackage{graphicx,units}
\usepackage{latexsym}
\usepackage{epsfig,esint}
\usepackage{amssymb,amsmath,mathtools}
\usepackage{amstext}
\usepackage{amsgen}
\usepackage{amsxtra}
\usepackage{amsgen}
\usepackage{amsthm}
\usepackage{tcolorbox}
\usepackage{mathrsfs,dsfont}
\usepackage[T1]{fontenc}

\usepackage[top=3cm, bottom=3cm,left=3cm, right=3cm]{geometry}

\usepackage{color}
\usepackage[colorlinks=true, linkcolor=blue, citecolor=red]{hyperref}

\usepackage{comment}

\usepackage[pagewise]{lineno}

\newtheorem{thm}{Theorem}
\newtheorem*{thm*}{Theorem}
\newtheorem{prop}{Proposition}[section]
\newtheorem{lemma}[prop]{Lemma}
\newtheorem{cor}[prop]{Corollary}

\theoremstyle{remark}
\newtheorem{rem}[prop]{Remark}

\theoremstyle{definition}

\newtheorem{definition}[prop]{Definition}

\numberwithin{equation}{section}

\newcommand{\RR}{\mathbb{R}}

\DeclareMathOperator{\dv}{div}
\DeclareMathOperator*{\esssup}{ess \, sup}

\DeclarePairedDelimiter\floor{\lfloor}{\rfloor}

\DeclareMathOperator*{\essinf}{ess \, inf}

\renewcommand{\P}{{\bf P}}
\newcommand{\p}{{\bf p}}

\newcommand{\e}{{\bf e}}

\newcommand{\bxi}{\xi}

\renewcommand{\d}{\, \mathrm{d} }
\newcommand{\der}{\mathrm{d}}

\newcommand{\loc}{\mathrm{loc}}
\usepackage{xcolor}

\newcommand{\T}{\mathbb{T}}

\DeclareMathOperator{\supp}{\mathrm{supp}}

\begin{document}
\title[Instantaneous Regularisation for a degenerate non-local equation]{Instantaneous Regularisation via Harnack-type Methods for a Degenerate Non-local Equation}

\author[S.~M.~Schulz]{Simon M.~Schulz}
\address[S.~M.~Schulz]{Laboratoire de Math\'ematiques de Versailles, UMR 8100 CNRS, UVSQ, Universit\'e Paris-Saclay, 45 Av.~des \'Etats-Unis, 78000 Versailles Cedex, France}\email{simon.schulz@uvsq.fr}

\keywords{Degenerate non-local equations, Harnack inequality, infinite speed of propagation, Bombier--Giusti Lemma, active particle models}

\subjclass[2020]{35B65, 35K65, 35K55, 35Q92, 35A15}

\begin{abstract}
  We prove an instantaneous smoothing effect for a degenerate non-local evolution equation (originally derived as the many-particle limit of a system of active Brownian particles). The underlying mechanism is that, by deriving a Harnack-type inequality for an averaged quantity, we are able to show that the equation becomes strongly parabolic for positive times, and thus self-regularising. This paper is a significant improvement of a previous work, where additional (weak) non-degeneracy assumptions were imposed on the initial data. Herein, we completely remove these assumptions, and prove---in addition to the regularity bootstrap---quantitative estimates on the aforementioned averaged quantity by means of an approach \emph{\`a la} Bombieri--Giusti. 
\end{abstract}

\maketitle

\vspace*{-1cm}

\setcounter{tocdepth}{1}
\begin{scriptsize}
    \tableofcontents
\end{scriptsize}

\vspace*{-1.4cm}

\section{Introduction}

The purpose of this paper is to present a refined regularity theory for weak solutions $f(t,x,\theta)$ of the degenerate non-local evolution equation 
\begin{equation}\label{eq:main eqn}
    \partial_t f +   \dv \!\big(  (1- \rho) f \e(\theta)\big)
= \dv \!\big( (1-\rho)\nabla f + f \nabla \rho \big) + \partial_{\theta}^2 f, 
\end{equation}
where $\e(\theta) = (\cos \theta, \sin \theta)$, and the (non-local) averaged quantity $\rho$ is defined by 
\begin{equation}\label{eq:rho def}
    \rho(t,x) := \int_\T   f(t,x,\theta) \d \theta, 
\end{equation}
where $\T \cong [0,2\pi)$ is the torus. The notations $\dv$ and $\nabla$ represent the divergence and gradient operators with respect to the $x$ variable only. The initial value problem for \eqref{eq:main eqn} is posed with periodic boundary conditions both in the space variable $x \in \T^2$ and the angle variable $\theta \in \T$. In the paper, we denote the \emph{space-angle coordinate} by $\bxi = (x,\theta) \in \T^3$.

Equation \eqref{eq:main eqn} was derived in \cite{bbes model}, and corresponds to the many-particle limit of a system of active Brownian particles (\textit{cf.~}\cite{Cates:2013ia,Romanczuk:2012iz})---these models appear naturally when studying collective behaviour in the biological and social sciences (\textit{cf.~}\cite{collective,Yeomans:2015dt}). A global existence result for \eqref{eq:main eqn} was obtained in \cite{Martin}, by means of a non-local \emph{boundedness-by-entropy method}, where one rewrites \eqref{eq:main eqn} as an evolution on the first variation of a suitable non-local functional (which itself can be interpreted as a perturbation of a Wasserstein-type gradient flow). The weak solutions obtained therein are dual to test functions belonging to $L^6_t W^{1,6}_{x,\theta}$ (thus, in principle, they are endowed with rather little regularity), and satisfy $0 \leq \rho \leq 1$ a.e.; hence \eqref{eq:main eqn} is of \emph{degenerate parabolic} type, with degeneracy at $\rho=1$.

Recently, a rather general regularity theory was established in \cite{reg2}; following ideas presented for a simpler active model in \cite{reg1}. Therein, the authors proved that weak solutions of \eqref{eq:main eqn} are in fact smooth for all positive times, provided certain (weak) non-degeneracy assumptions are imposed on the initial data. More precisely, provided one has suitable control on the degenerate zone $\{\rho=1\}$ (where \eqref{eq:main eqn} loses strict parabolicity) at the initial time, then one can show that \eqref{eq:main eqn} becomes self-regularising. In this work, we completely remove these non-degeneracy assumptions on the initial data, thereby showing that---surprisingly, given the degenerate/non-local appearance of the PDE---all weak solutions (as defined in \cite{Martin}) instantaneously become smooth. The present paper is therefore a significant improvement of the previous contribution \cite{reg2}, and (in some sense) it concludes the regularity study of equation \eqref{eq:main eqn}; all weak solutions are classical solutions for positive times. The questions of uniqueness/non-uniqueness and convergence to stationary states are still largely open and will be the focus of future investigations.

While \eqref{eq:main eqn} appears to be fraught with degeneracies and non-local effects, it is actually very simple to (heuristically) explain its underlying self-regularising mechanism. Indeed, the main obstacle for regularity is that the \emph{self-diffusive} term $\dv((1\!-\!\rho)\nabla f)$ is degenerate when $\rho=1$. Consequently, in the vicinity of $\{\rho=1\}$, the smoothing effects provided by this term cannot compensate for the losses of regularity produced by the non-local drift and the \emph{infinite cross-diffusion} term $\dv(f \nabla \rho)$; this latter term formally comprises $f \Delta \rho$, which includes competing diffusive effects from all angles $\theta \in [0,2\pi)$. However, upon integrating \eqref{eq:main eqn} in the angle variable $\theta$, one observes that the equation for $\rho$ is strongly parabolic: 
\begin{equation}\label{eq:rho eq}
    \partial_t \rho + \dv\big( (1-\rho) \p \big) = \Delta \rho, 
\end{equation}
where the \emph{polarisation} $\p$ is a non-local ``moment'' quantity determined from $f$   by 
\begin{equation}\label{eq:polarisation def}
    \p(t,x) := \int_\T f(t,x,\theta) \e(\theta) \d \theta. 
\end{equation}
In turn, this latter equation is endowed with an \emph{instantaneous separation from $\{\rho=1\}$} property, meaning that $0 < \rho < 1$ for all positive times. A quantitative description of this \emph{infinite speed of propagation} effect is embodied by a Harnack-type inequality for $1-\rho$ (which we prove in \S \ref{sec:quantitative} of this manuscript). Once this quantitative estimate is established, one deduces that \eqref{eq:main eqn} is strongly parabolic for all positive times, and the strategy of \cite[\S 6-7]{reg2} (see also \cite[\S 4]{reg1}) can be adapted to show smoothness of the weak solution. While a lower bound on $1-\rho$ had been obtained in \cite{reg2} by an ad-hoc method, the authors had not clearly identified that it could be derived from a Harnack inequality, which   led them to impose the aforementioned unnecessary non-degeneracy conditions on the initial data. 

Herein, we prove the desired Harnack-type inequality directly, instead of invoking an abstract result from the literature---we explain our reason for doing so in the paragraph that follows. Before jumping to this, we remind the reader that our aim, ultimately, is to obtain a \emph{quantitative} form of the instantaneous separation of $\rho$ from the degenerate state $\{\rho=1\}$; \textit{cf.~}Theorem \ref{thm:reg}. This phenomenon is \emph{expected} from the strong parabolicity of \eqref{eq:rho eq}, but obtaining an explicit (uniform-in-$x$) lower bound on $1\!-\!\rho$ near the initial time is \emph{a priori} not obvious.

Although \eqref{eq:rho eq} may be interpreted as a linear drift-diffusion equation for $\rho$---and is thus endowed with a Harnack inequality---the constants appearing in its corresponding Harnack inequality depend quantitatively on the drift $\p$. In particular, when applying the estimate on cylinders approaching the initial time, one must carefully track the dependence of the Harnack constant on the rescaled drift coefficient---\emph{rescaled} because the equation \eqref{eq:rho eq} is not exactly invariant under the usual parabolic rescaling $(t,x) \mapsto (t_0 + \lambda^2 t, x_0 + \lambda x)$. Since this dependence is seldom made explicit in Harnack inequalities found in reference texts (\textit{cf.~e.g.~}\cite[Theorem 6.27]{lieberman}), we adopt a more hands-on approach and prove the required estimate (\textit{cf.~}Lemma \ref{prop:lower bound rho improved small times only}) by means of an argument in the spirit of Bombieri--Giusti \cite{BG,MoserBG}.

Once the desired uniform lower bound on $1\!-\!\rho$ is established (and hence the strong parabolicity of \eqref{eq:main eqn}) for positive times, the regularity bootstrap essentially begins with one key observation on the structure of \eqref{eq:main eqn}: provided one justifies \emph{passing from divergence to non-divergence form} in \eqref{eq:main eqn}, a cancellation in the mixed terms $\nabla f \!\cdot\! \nabla \rho$ occurs, yielding 
\begin{equation*}
    \partial_t f + \dv((1-\rho)f\e(\theta)) = (1-\rho)\Delta f + f \Delta \rho + \partial^2_\theta f. 
\end{equation*}
By formally testing the above with $\Delta f$ (and then $\partial^2_\theta f$), one is able to close the usual $H^2$-type estimate. Then, through repeated applications of difference quotients in time and Schauder-type results, we are able to successively take additional time and space derivatives in \eqref{eq:main eqn} and show that all derivatives bounded, thereby establishing smoothness. The rigourous passage from the divergence to non-divergence form had already been written in \cite{reg2}, and relies on a Galerkin scheme (see Lemma \ref{lem:H2 f} herein). We briefly highlight that the use of Galerkin's method as a tool for regularity (while classical) has recently proven to be very useful in settings where the method of difference quotients (in space) is not justifiable, \textit{e.g.~}in the context of elliptic equations in metric measure spaces (\textit{cf.~}\cite{Ivan}).

The rest of the paper is organised as follows. In \S \ref{sec:def and background} we recall the definition of weak solution as well as the global existence result obtained in \cite{Martin}. We also recall some basic regularity results that are used throughout the paper. In \S \ref{sec:main theorem} we state our main theorem. \S \ref{sec:bombieri giusti} is concerned with a classical result known as the \emph{Bombieri--Giusti Lemma}, which is the cornerstone of the approach ``without the John--Nirenberg Lemma'' (\textit{cf.~}\cite{MoserBG}) to obtaining Harnack inequalities; to make the paper self-contained, we write a statement and proof which is tailored to our setting. In \S \ref{sec:moser}--\ref{sec:log} (which are respectively concerned with Moser-type $L^\infty$ bounds and estimates on the sub-/super-level sets of the logarithm) we verify that $1\!-\!\rho$ and its reciprocal satisfy the assumptions of the Bombieri--Giusti Lemma of \S \ref{sec:bombieri giusti}. Having shown that this is the case, we obtain the sought quantitative lower bound on $1\!-\!\rho$ in \S \ref{sec:quantitative}. \S \ref{sec:smooth} is then concerned with the regularity bootstrap for \eqref{eq:main eqn}, which (although entirely self-contained herein) relies heavily on the ideas already introduced in \cite[\S 6-7]{reg2}.

\section{Definitions and Background}\label{sec:def and background}

In this part of the manuscript, we recall the definition of weak solution of \eqref{eq:main eqn} introduced in the prior works \cite{Martin,reg2}, and summarise the main existence result  as well as some background notions that will appear in later sections.

\subsection{Definitions}

We define the following function space. Note that, therein, the space containing $\partial_t f$ is precisely the dual of $L^6_\loc([0,\infty);W^{1,6}(\T^3))$. 

\begin{definition}[Function space $X$]\label{def:function space}
The function space $X$ is defined to be the family of functions $f \in L^3_\loc([0,\infty);L^3(\T^3))$ which satisfy, with $\rho := \int_\T f \d \theta$, the following: 
    \begin{itemize}
        \item[(i)] $f \geq 0$ a.e.~$(0,\infty)\times\T^3$; 
        \item[(ii)] $\partial_\theta \sqrt{f} \in L^2_\loc([0,\infty);L^2(\T^3))$ and $\partial_t f \in L^{\nicefrac{6}{5}}_\loc([0,\infty);W^{-1,\nicefrac{6}{5}}(\T^3))$; 
        \item[(iii)] $0 \leq \rho \leq 1$ a.e.~$(0,\!\infty)\times\T^2$, $\partial_t \rho \in L^2_\loc(0,\!\infty;H^{-1}(\T^2))$, $\nabla\sqrt{1\!-\!\rho} \in L^2_\loc(0,\!\infty;L^2(\T^2))$; 
        \item[(iv)] $\sqrt{1\!-\!\rho}\nabla\sqrt{f} \in L^2_\loc([0,\infty);L^2(\T^3))$. 
    \end{itemize}
\end{definition}

We recall the definition of weak solution (\textit{cf.}~\cite[Definition 2.1]{reg2}). For succinctness, we define the test functions $\mathscr{A} := \{ \varphi \in L^6_\loc([0,\!\infty);W^{1,6}(\T^3)) :  \partial_t \varphi \in L^{\nicefrac{3}{2}}_\loc([0,\!\infty);L^{\nicefrac{3}{2}}(\T^3)) \}$. 

    \begin{definition}[Weak solution]\label{def:weak sol}
      We say that $f \in X $ is a \emph{weak solution} of \eqref{eq:main eqn} if, for all $\varphi \in \mathscr{A}$, it holds: for a.e.~$t_1,t_2 \in (0,\infty)$ 
  \begin{equation}\label{eq:weak sol}
    \begin{aligned}
     \int_{t_1}^{t_2}\! & \int_{\T^3} f \partial_t \varphi \d \bxi \d t + \int_{t_1}^{t_2} \int_{\T^3} (1-\rho) f \e(\theta) \cdot \nabla \varphi \d \bxi \d t \\ 
      =&\int_{t_1}^{t_2} \!\int_{\T^3} \!\Big( \big((1-\rho)\nabla f \!+\! f \nabla \rho \big) \!\cdot\! \nabla \varphi \!+\! \partial_\theta f\partial_\theta \varphi \Big) \d \bxi \d t \!+\! \int_{\T^3} f \varphi \big|_{t_2} \d\bxi \!\!-\!\! \int_{\T^3} f \varphi \big|_{t_1} \d\bxi. 
    \end{aligned}
\end{equation}
    Furthermore, we say that $f$ is a weak solution \emph{with initial data $f_0$} if it is a weak solution, and \eqref{eq:weak sol} holds with $t_1=0$ and the final integral replaced by $\int_{\T^3} f_0(\xi) \varphi(0,\xi) \d \xi$. 
    \end{definition}

    For the convenience of the reader, we recall the notation $\xi = (x,\theta)$ in the integrals above. It is also clear that, by considering the periodic extension of $f$ to all of $\mathbb{R}^3$, the definition of weak solution given above also implies that \eqref{eq:main eqn} holds in the sense of distributions dual to the class of test functions $C^\infty_c([0,\infty)\!\times\!\mathbb{R}^3)$; \textit{cf.}~\cite[Lemma 4.2]{Martin}.

    \subsection{Background notions}

    \subsubsection{Global existence and attainment of initial data}

    The next remark summarises the main existence result obtained in \cite{Martin}. 

\begin{rem}[Global existence]\label{rem:global exist}
Recall that, given initial data $f_0$ which satisfies 
\begin{equation}\label{eq:initial data conditions}
   0 \leq f_0 \in L^q(\T^3) \quad \text{for some } q \in (1,\infty], \qquad 0 \leq \rho_0 \leq 1 \quad \text{a.e.~in } \T^2, 
\end{equation}
then there exists a  weak solution of \eqref{eq:main eqn} with initial data $f_0$ (in the precise sense of Definition \ref{def:weak sol}), which satisfies 
\begin{equation}\label{eq:convergence to initial data old}
    \lim_{t\to 0^+} \Vert f(t) - f_0 \Vert_{W^{-1,\nicefrac{6}{5}}(\T^3)} = 0; 
\end{equation}
this was proved in \cite[Theorem 1]{Martin}. Throughout the manuscript, we say that $f_0$ is \emph{admissible initial data} if it satisfies \eqref{eq:initial data conditions}, and whenever initial data $f_0$ is mentioned, the letter $\rho_0$ always denotes $\rho_0 = \int_\T f_0 \d \theta$, \textit{i.e.}~the initial data for $\rho$. 
\end{rem}

For the convenience of the reader, we explain in detail the distributional convergence to the initial data \eqref{eq:convergence to initial data old} in the remark that follows.

\begin{rem}[Attainment of initial data]\label{rem:attainment old}
Let $f$ be a weak solution with initial data $f_0$ as per Definition \ref{def:weak sol}. Then, for all $\varphi \in W^{1,6}(\T^3)$ depending only on $\xi$, it holds a.e.~$t>0$
  \begin{equation}\label{eq:to see clearly conservation of initial mass}
    \begin{aligned}
    \int_{\T^3} \! \big( f(t,\xi) - f_0(\xi) \big) \varphi(\xi)   \d\bxi =&\, \int_{0}^{t} \int_{\T^3} (1-\rho) f \e(\theta) \cdot \nabla \varphi \d \bxi \d s \\ 
    &- \!\int_{0}^{t} \!\!\int_{\T^3} \!\!\Big( \big((1-\rho)\nabla f \!+\! f \nabla \rho \big) \!\cdot\! \nabla \varphi \!+\! \partial_\theta f\partial_\theta \varphi \Big) \d \bxi \d s. 
    \end{aligned}
\end{equation}
The boundedness of the relevant terms from Definition \ref{def:weak sol} (\textit{cf.~}\cite[Eq.~(1.5)]{Martin}) yields \eqref{eq:convergence to initial data old}: 
\begin{equation}\label{eq:as per remark attainment old}
\lim_{t \to 0^+} \Vert f(t) - f_0 \Vert_{(W^{1,6}(\T^3))'} = \lim_{t \to 0^+} \sup_{ \Vert \varphi \Vert_{W^{1,6}} \leq 1} \bigg| \int_{\T^3} \big(f(t)-f_0 \big) \varphi  \d \xi \bigg| = 0 .
\end{equation}
\end{rem}

 Next is a remark about the attainment of $\rho$-initial data, which was not proved before.

\begin{rem}[Convergence to initial data for $\rho$]\label{rem:convergence initial data rho}
  Suppose that $f$ is a weak solution of \eqref{eq:main eqn} with admissible initial data $f_0$. Then, it holds 
  \begin{equation}\label{eq:conv rho to initial data}
\lim_{t\to 0^+} \Vert \rho(t) - \rho_0 \Vert_{L^2(\T^2)}=0. 
\end{equation}
Note that $\rho \in C([0,\infty);L^2(\T^2))$ is immediate from $\partial_t \rho \in L^2_\loc([0,\infty);H^{-1}(\T^2)) , \rho \in L^2_\loc([0,\infty);H^1(\T^2))$ and the Lions--Magenes Lemma \cite[Ch.3, \S 1.4
Lem.1.2]{Temam}, but this \emph{does not directly imply} \eqref{eq:conv rho to initial data}. Nevertheless, this convergence is true, \textit{cf.}~Appendix \ref{sec:appendix rho initial data}. 
\end{rem}

\subsubsection{Background regularity notions}

We recall some elementary boudedness results, which hold for any weak solution of \eqref{eq:main eqn} produced by admissible initial data \eqref{eq:initial data conditions}. 

Recall the equation \eqref{eq:rho eq} for $\rho$, where the drift term depends on the polarisation $\p$ given by \eqref{eq:polarisation def}, $\p = \int_\T f \e(\theta) \d \theta$. Since $f \geq 0$ and $0\leq \rho \leq 1$ a.e.~from Definition \ref{def:function space}--\ref{def:weak sol}, it holds 
\begin{equation}\label{eq:polarisation bound}
    | \p(t,x) | \leq \int_\T|f(t,x,\theta)| \d \theta = \rho(t,x) \leq 1 \implies \Vert \p \Vert_{L^\infty((0,\infty)\times\T^2)} \leq 1; 
\end{equation}
this boundedness holds for all \emph{angular moments} of $f$, \textit{e.g.~}$\P = \int_\T f \e(\theta)\otimes\e(\theta) \d \theta$ satisfies $\Vert \P \Vert_{L^\infty} < \infty$. Furthermore, simply by multiplying the equation \eqref{eq:main eqn} for $f$ with $\e(\theta)$ and integrating in $\theta$, we uncover the equation for $\p$: 
\begin{equation}\label{eq:p eq}
    \partial_t \p + \dv((1-\rho)\P) = \dv((1-\rho)\nabla \p + \p \otimes \nabla \rho) - \p, 
\end{equation}
which, in coordinate form with summation convention, reads as $\partial_t p_i + \partial_j((1-\rho)P_{ij}) = \partial_j((1-\rho)\partial_j p_i + p_i \partial_j \rho) - p_i$. We also recall that the drift in \eqref{eq:rho eq} satisfies 
\begin{equation}\label{eq:grad p integ first steps}
    \dv((1-\rho)\p) \in L^{\nicefrac{3}{2}}_\loc([0,\infty);L^{\nicefrac{3}{2}}(\T^2)); 
\end{equation}
indeed, $\dv((1-\rho)\p) = \!-{\nabla \rho \!\cdot \p} + (1-\rho) \dv \p$, with $\nabla \rho \cdot \p \in L^2_\loc((0,\infty);L^2(\T^2))$ from Definition \ref{def:function space} and \eqref{eq:polarisation bound}, and (for all $T>0$) 
\begin{equation}\label{eq:the horrible one}
   \begin{aligned} \int_0^T \int_{\T^2} |(1-\rho) \nabla \p |^{\nicefrac{3}{2}} \d x \d t &= \int_0^T \int_{\T^2} (1-\rho)^{\nicefrac{3}{2}} \bigg| \int_\T \nabla f \otimes \e(\theta) \d \theta \bigg|^{\nicefrac{3}{2}} \d x \d t \\ 
   &\leq \sqrt{2\pi}\int_0^T \int_{\T^3} \Big(\sqrt{1-\rho} \frac{|\nabla f |}{\sqrt{f}}\Big)^{\nicefrac{3}{2}} (\sqrt{f})^{\nicefrac{3}{2}} \d \bxi \d t \\ 
   &\leq 4\sqrt{\pi} \Vert f \Vert_{L^{3}([0,T]\times\T^3)}^{\nicefrac{3}{4}} \Vert \sqrt{1-\rho}\nabla \sqrt{f} \Vert_{L^2([0,T]\times\T^3)}^{\nicefrac{3}{2}}, 
   \end{aligned}
\end{equation}
where we used Jensen's inequality and $0\leq 1-\rho \leq 1$ to get the second line and H\"older's inequality with the exponent $\nicefrac{4}{3}$ to get the final line. 
By \eqref{eq:rho eq} and \eqref{eq:grad p integ first steps}, we get 
\begin{equation*}
    \partial_t \rho - \Delta \rho =  -\dv((1-\rho)\p) \in L^{\nicefrac{3}{2}}_\loc([0,\infty);L^{\nicefrac{3}{2}}(\T^2)), 
\end{equation*}
whence, by the Calder\'on--Zygmund Theorem for the heat equation (\textit{cf.~}\cite[Theorem 3.1]{RosHeat}), 
\begin{equation}\label{eq:initial upgrade dt rho not the real one}
    \partial_t \rho, \nabla^2 \rho \in L^{\nicefrac{3}{2}}((t,T) \! \times \! \T^2) \qquad \forall 0<t<T; 
\end{equation}
we will use \eqref{eq:initial upgrade dt rho not the real one} as a starting point for our analysis, in \S \ref{sec:moser} of the present work.

We finish this subsection with four well-known standard results which will be used repeatedly in \S \ref{sec:smooth}. The first is an interpolation lemma that follows from the Gagliardo--Nirenberg inequality, the second and third are Schauder-type results for non-divergence form parabolic equations (in Sobolev spaces, and H\"older spaces, respectively), and the fourth is a Gehring Lemma for divergence form parabolic equations.

\begin{lemma}[\S 1 Proposition 3.2, \cite{DiBenedetto}]\label{lem:dibenedetto classic}
    Let $d \in \mathbb{N}$, $\Omega \subset \mathbb{R}^d$ have piecewise smooth boundary, $I \subseteq [0,\infty)$ any bounded interval, and $p,m \geq 1$. Then, there exists a positive constant $C=C(d,p,m,\partial\Omega,\Omega)$, independent of $I$, such that, for all $v \in L^\infty(I;L^m(\Omega)) \cap L^p(I;W^{1,p}(\Omega))$, 
    \begin{equation*}
        \Vert v \Vert_{L^q(I\times\Omega)} \leq C(1+|I|)^{\frac{1}{q}} \Big( \Vert v \Vert_{L^\infty(I;L^m(\Omega))} + \Vert v \Vert_{L^p(I;W^{1,p}(\Omega))} \Big) \qquad \text{with } q = p(1+\frac{m}{d}). 
    \end{equation*}
    Consequently, we have, for any bounded interval $I\subseteq [0,\infty)$, the continuous embeddings 
    \begin{equation*}
       \begin{aligned} 
       &L^\infty(I; L^2(\T^2)) \cap L^2(I;H^1(\T^2)) \hookrightarrow L^4(I\times\T^2) , \\
       &L^\infty(I;L^2(\T^3)) \cap L^2(I;H^1(\T^3)) \hookrightarrow L^{\nicefrac{10}{3}}(I\times\T^3). 
       \end{aligned}
    \end{equation*}
\end{lemma}

The next two results are well-known. The first is \cite[Theorem 1.9]{Wang}  (\textit{cf.~}\cite[Lemma 5.14]{Dong} and \cite[Theorem 2.4]{Kim}), and the second is \cite[Theorem 1.19]{Wang} (\textit{cf.~}\cite[Theorem 8.11.1]{Krylov}). We recall the H\"older spaces and their parabolic analogues: for $\alpha \in (0,1)$ and $Q \subseteq \mathbb{R}^{1+d}$,
\begin{equation*}
    C^\alpha(Q) = \Big\{ v \in C(\overline{Q}) : \Vert v \Vert_{C^\alpha(Q)} \!<\! \infty  \Big\}, \quad \Vert v \Vert_{C^\alpha(Q)} = \Vert v \Vert_{L^\infty(Q)} + [v]_{\alpha;Q}, 
\end{equation*}
where $[v]_{\alpha;Q} = \sup_{z_1,z_2 \in Q, z_1 \neq z_2} \frac{|v(z_1)-v(z_2)|}{|z_1 - z_2 |^\alpha}$; recall that $(C^\alpha,\Vert \cdot \Vert_{C^\alpha})$ is a Banach space, and 
\begin{equation*}
    C^{\nicefrac{\alpha}{2},\alpha}(Q) = \Big\{ v \in C(\overline{Q}) : \Vert v \Vert_{C^{\nicefrac{\alpha}{2},\alpha}(Q)} < \infty  \Big\}, \quad \Vert v \Vert_{C^{\nicefrac{\alpha}{2},\alpha}(Q)} = \Vert v \Vert_{L^\infty(Q)} + [v]_{\nicefrac{\alpha}{2},\alpha;Q}, 
\end{equation*}
where $[v]_{\nicefrac{\alpha}{2},\alpha;Q} = \sup_{(t_1,x_1) \neq (t_2,x_2)} \frac{|v(t_1,x_1)-v(t_2,x_2)|}{|t_1 - t_2 |^{\nicefrac{\alpha}{2}} + |x_1 -x_2|^\alpha}$. Also define 
$$C^{1+\nicefrac{\alpha}{2},2+\alpha}(Q) = \Big\{ \partial_t v , \nabla^2 v \in C^{\nicefrac{\alpha}{2},\alpha}(Q) \Big\}, $$
with the norm $\Vert v \Vert_{C^{1+\nicefrac{\alpha}{2},2+\alpha}(Q)} = \Vert v \Vert_{L^\infty(Q)} + \Vert \nabla v \Vert_{L^\infty(Q)} + \Vert \partial_t v\Vert_{C^{\nicefrac{\alpha}{2},\alpha}(Q)} + \Vert \nabla^2 v\Vert_{C^{\nicefrac{\alpha}{2},\alpha}(Q)}.$ 

\begin{lemma}[Schauder estimate: Sobolev spaces]\label{lem:Lp schauder}
    Let $I \subseteq [0,\infty)$ be any compact interval, $p \in [2,\infty)$, $\alpha \in (0,1)$, $c > 1$, $d\in\mathbb{N}$. Assume $A \in [C^{\alpha}(I\times\T^d)]^{d \times d}$ is symmetric and \begin{equation*}
    c^{-1} |w|^2 \leq A(t,x)w \cdot w \leq c|w|^2 \qquad \forall w \in \mathbb{R}^d, \quad \forall (t,x) \in I\times\T^d.
    \end{equation*}
    Suppose that $v \in L^2(I;H^2(\T^d))$, with $\partial_t v \in L^2(I\!\times\! \T^d)$, satisfies, for $F \in L^p(I\times\T^d)$,  
    \begin{equation*}
        \partial_t v - A:\nabla^2 v = F. 
    \end{equation*}
Then, for any $I' \subset\subset I$ compactly contained in $I$, it holds $\partial_t v , \nabla^2 v \in L^p(I' \times \T^d)$, and 
\begin{equation}\label{eq:schauder Lp}
    \Vert \partial_t v \Vert_{L^p(I'\times\T^d)} + \Vert \nabla^2_\xi v \Vert_{L^p(I'\times\T^d)} \leq C \Big(  \Vert F \Vert_{L^p(I\times\T^d)} + \Vert v \Vert_{L^2(I\times\T^d)} \Big), 
\end{equation}
where $C$ depends only on $|I\setminus I'|,I,c,\Vert A \Vert_{C^{\alpha}},\alpha,d$. 
\end{lemma}

\begin{lemma}[Schauder estimate: H\"older spaces]\label{lem:Schauder Holder}
    Let $I \subseteq [0,\infty)$ be any compact interval, $\alpha \in (0,1)$, $c > 1$, and $d\in\mathbb{N}$. Assume the matrix $A \in [C^{\alpha}(I\times\T^d)]^{d \times d}$ is symmetric and  \begin{equation*}
    c^{-1} |w|^2 \leq A(t,x)w \cdot w \leq c|w|^2 \qquad \forall w \in \mathbb{R}^d, \quad \forall (t,x) \in I\times\T^d.
    \end{equation*}
    Suppose that $v \in L^2(I;H^2(\T^d))$, with $\partial_t v \in L^2(I\!\times\! \T^d)$, satisfies, for $F \in C^{\nicefrac{\alpha}{2},\alpha}(I\times\T^d)$, 
    \begin{equation*}
        \partial_t v - A:\nabla^2 v = F. 
    \end{equation*}
Then, for any $I' \subset\subset I$ compactly contained in $I$, it holds $v \in C^{1+\nicefrac{\alpha}{2},2+\alpha}(I' \times \T^d)$, and 
\begin{equation}\label{eq:est schauder holder}
    \Vert v \Vert_{C^{1+\nicefrac{\alpha}{2},2+\alpha}(I'\times\T^d)} \leq C \Big( \Vert F \Vert_{C^{\nicefrac{\alpha}{2},\alpha}(I\times\T^d)} + \Vert v \Vert_{L^2(I\times\T^d)} \Big), 
\end{equation} 
where $C$ depends only on $|I\setminus I'|,I,c,\Vert A \Vert_{C^{\alpha}},\alpha,d$. 
\end{lemma}

The final result of this subsection is a consequence of (and a non-quantitative version of) \cite[Theorem 8.2]{AuscherGehring} and is a classical Gehring-type ``improvement of integrability'' result for solutions of strongly parabolic \emph{systems} in divergence form with \emph{measurable} coefficients---we emphasise that, unlike in Lemmas \ref{lem:Lp schauder}--\ref{lem:Schauder Holder}, the entries of the matrix $A$ are \emph{not} required to be uniformly continuous, and $A$ need not be symmetric. In what follows, for $p \geq 1$, 
\begin{equation}
    p_* := \frac{(d + 2)p}{d + 2 + p}, 
\end{equation}
where $d$ is the spatial dimension of the domain. 

\begin{lemma}[Gehring-type Lemma]\label{lem:Gehring}
    Let $I\subseteq [0,\infty)$ be any compact interval, $c>1$, $N,d \in \mathbb{N}$, and suppose $g,G$ are vector fields satisfying 
    \begin{equation*}
        g \in [L^{p_*}(I\times\T^d)]^{N}, \qquad G \in [L^p(I\times\T^d)]^{N\times d} 
    \end{equation*}
    for some $p>2$. Assume the tensor $A \in [L^\infty(I\times\T^d)]^{N^2 d^2}$ satisfies the ellipticity condition 
    \begin{equation*}
         c^{-1} |w|^2 \leq \!\!\! \underbrace{A(t,x)w \cdot w}_{ \sum_{\alpha,\beta,i,j} A(t,x)^{\alpha\beta}_{ij} w_i^\alpha w_j^\beta } \!\!\! \leq c|w|^2 \qquad \forall w \in \mathbb{R}^{d \times N}, \quad \forall (t,x) \in I\times\T^d,
    \end{equation*}
    and suppose that $v \in [L^\infty(I\times\T^d) \cap L^2(I;H^1(\T^d))]^N$, with $\partial_t v \in [L^2(I;H^{-1}(\T^d))]^N$, satisfies the linear \emph{divergence form} equation 
    \begin{equation*}
        \partial_t v - \dv(A \nabla v) = g + \dv G. 
    \end{equation*}
Then, there exists $\delta>0$ depending only on $c,I,p,d,\Vert A \Vert_{L^\infty},$ such that, for any $I' \subset\subset I$ compactly contained in $I$, it holds $ \nabla v \in L^{2+\delta}(I' \times \T^d)$. 
\end{lemma}

\subsection{Notations} Throughout the paper, for $r>0$ and $d=2,3$, we have the following notation for the parabolic cylinders: 
\begin{equation}\label{eq:parabolic cylinder}
    Q_r(t) = (t-r,t+r)\times \T^d, \qquad Q_r^+(t) = (t,t+r) \times \T^d, \qquad Q_r^-(t) = (t-r,t)\times\T^d; 
\end{equation}
and we simply write $Q_r$ to denote $Q_r(0)$ and analogously for $Q_r^\pm=Q_r^\pm(0)$. We also use the notations $a \vee b = \max\{a,b\}$ and $a \wedge b = \min\{a,b\}$ for real numbers $a,b$. We write $L^p_{t,\loc} W^{k,q}_x$ (resp.~$L^p_{t,\loc} W^{k,q}_\xi)$ as shorthand for $L^p((t,T);W^{k,q}(\T^d))$  ($d=2,3$, resp.) for all $(t,T) \subset \subset (0,\infty)$, where $\subset\subset$ means \emph{compactly contained}. In the bootstrap argument of \S \ref{sec:bootstrap}, we use the notation $D^m v$ with $m\in\mathbb{N}$ to denote the tensor of all derivatives of order $m$ in space-angle (\textit{i.e.}~excluding the time variable). Throughout \S \ref{sec:moser}--\ref{sec:quantitative}, we write $u=1-\rho$.

\section{Main Theorem}\label{sec:main theorem}

Our main theorem is the following. It may be interpreted as a significant improvement of \cite[Theorem 1]{reg2}, where one no longer requires stronger assumptions on the initial data (\textit{cf.}~\cite[Definition 3.1]{reg2}) to deduce the smoothness of weak solutions. The result shows that all weak solutions are classical solutions, provided we exclude ``saturated'' initial data satisfying $\rho_0 \equiv 1$ a.e. The underlying mechanism (as previously mentioned) is that a strictly positive lower bound on $1-\rho$ makes \eqref{eq:main eqn} strongly parabolic, and thus self-regularising. 

\begin{thm}[Interior non-degeneracy and smoothness of weak solutions]\label{thm:reg}
Let $f$ be a weak solution of \eqref{eq:main eqn} with admissible initial data $f_0$ satisfying \eqref{eq:initial data conditions}. Then, there exists a constant $C > 1$ $($independent of $t,T$, the initial data, and the solution$)$ such that 
   \begin{equation}\label{eq:infinite speed of propagation in thm statement}
       \essinf\limits_{(t,T)\times\T^2}(1\!-\!\rho) \geq  C^{-1} \Vert 1\!-\!\rho_0 \Vert_{L^1(\T^2)} e^{-C( T + t^{-10} )} 
   \end{equation}
for all $0<t<T<\infty$. Moreover, if $\Vert 1-\rho_0\Vert_{L^1(\T^2)}>0$, then $f \in C^\infty((0,\infty)\!\times\!\T^3)$.  
\end{thm}

Of course, we do not claim optimality in the exponents on the right-hand side of \eqref{eq:infinite speed of propagation in thm statement}. Note that Theorem \ref{thm:reg} allows for degenerate choices of initial data with $\rho_0 = 1$ on a Borel set of positive measure (provided it is not of full measure), and we remind the reader that the terminology ``\emph{$f$ is a weak solution with initial data $f_0$}'' was explained in Definition \ref{def:weak sol}. 

\begin{rem}[Limiting behaviour as $T\to\infty$]
    Note that the lower bound provided by \eqref{eq:infinite speed of propagation in thm statement} vanishes in the limit as $T \to \infty$, thus leading to (potentially) degenerate diffusion at infinite times. This observation is compatible with the segregated (and hence non-smooth) stationary states that were numerically obtained in \cite{bbes model}; \textit{cf.~}Figures 4 and 5 therein. 
\end{rem}

\section{The Bombieri--Giusti Lemma}\label{sec:bombieri giusti}

The following lemma is fundamental when proving Harnack inequalities (it will be used subsequently to prove the sought lower bound on $1-\rho$ in \S \ref{sec:quantitative}, which then implies that \eqref{eq:main eqn} is strongly parabolic for positive times). It applies to any non-negative measurable function, not necessarily the solution of a particular equation. The result is classical and the underlying argument is well-known to specialists (\textit{cf.}~\cite{BG,MoserBG}); of course we do not claim novelty, but we include a proof tailored to our setting (with slightly different statement and proof to---but closely following---the one in \cite[\S 3.2]{BonforteDolbeaultNazaretSimonov}) to make the paper self-contained.

\begin{lemma}[Bombieri--Giusti Lemma]\label{lem:bombieri giusti}
 Let $\beta>0$, $\tilde r \in (0,1]$, and $\{\mathcal{Q}_r\}_{r \in (0,\tilde r]}$ be a family of nested subsets of $\mathbb{R}^{d+1}$ with $\mathcal{Q}_{r} \subset \mathcal{Q}_R$ for all $0<r<R\leq \tilde r$ and $|\mathcal{Q}_1|<\infty$. Assume that $\varphi \in L^1(\mathcal{Q}_1)$ is non-negative  and satisfies$:$
    \begin{enumerate}
        \item[(i)] there exists a positive constant $C_0$ independent of $\beta,r,R,\tilde r,p$ such that 
    \begin{equation}\label{eq:first est adapt bombieri giusti}
       \Vert \varphi \Vert_{L^\infty(\mathcal{Q}_r)} \leq \left(\frac{C_0 }{(R-r)^\beta}\right)^{\frac{1}{p}} \Vert \varphi \Vert_{L^p(\mathcal{Q}_R)} \quad \text{for all } 0 < r < R \leq \tilde r \text{ and all } p \in (0,1); 
    \end{equation}
        \item[(ii)] there exists a positive constant $C_1 >  e^{-1}$ such that, for all $r \in (0,\tilde r)$, 
    \begin{equation}\label{eq:assump bombieri giusti on log set}
        \big| \{ (t,y) \in \mathcal{Q}_r : \, \log \varphi(t,y) > s \} \big| \leq \frac{C_1}{s} \qquad \text{for all } s > 0. 
    \end{equation}
    \end{enumerate}
      Then, there exists a positive constant $C_2 = C_2(C_0,C_1,\beta)$, independent of $r,\tilde r$, such that 
    \begin{equation}\label{eq:bombieri giusti final bound i}
        \Vert \varphi \Vert_{L^\infty(\mathcal{Q}_r)} \leq \exp \big( C_2 (\tilde r -r)^{-2\beta} \big) \qquad \text{for all } r \in (0,\tilde r). 
    \end{equation}
\end{lemma}
We emphasise that $C_0,C_1$ are both independent of $p,r,\tilde r,R$ in the assumptions \eqref{eq:first est adapt bombieri giusti}--\eqref{eq:assump bombieri giusti on log set}; and hence $C_2$ also is. Furthermore, the sets $\mathcal{Q}_r$ may be more general subsets than the parabolic cylinders of \eqref{eq:parabolic cylinder}. We highlight that the content of the result is not only the local boundedness of $\varphi$ \emph{per se}, but the precise dependence on $r,\tilde r$ given by \eqref{eq:bombieri giusti final bound i}.

\begin{proof}
    Define the increasing function $     r \mapsto   \Phi(r) := \esssup_{\mathcal{Q}_r} \log \varphi$. This function is well-defined since $\varphi$ is non-negative; however, \emph{a priori} it may be infinite. We will prove that the assumptions of the lemma imply that, for any $r \in (0,\tilde r)$, either: 
    \begin{itemize}
        \item[(a)] $\Phi(r) \leq 2 C_1$, in which case there is nothing to prove and we set $C_2 = 2C_1$; 
        \item[(b)] or, $\Phi(r) > 2 C_1$, in which case it holds 
        \begin{equation}\label{eq:iterate bombieri giusti}
            \Phi(r) \leq \frac{3}{4} \Phi(R) + \frac{C}{(R-r)^{2\beta}} \qquad \forall R \in (r,\tilde r], 
        \end{equation}
        where the positive constant $C=C(C_0,C_1)$ is independent of $r,R,\tilde{r}$. 
    \end{itemize}
    The rest of the proof is divided into two steps.

    \smallskip 

    \noindent 1. \textit{Proof of \eqref{eq:bombieri giusti final bound i} in case (b)}. We prove the result \eqref{eq:bombieri giusti final bound i} in the case where (b) holds. Suppose that case (b) holds at a given $r \in (0,\tilde r)$. Define the increasing sequence of radii $\{r_n\}_{n \geq 0}$:
\begin{equation*}
    r_n = \tilde r - \frac{\tilde r-r}{1+n} \qquad \forall n \in \mathbb{N} \cup \{0\} \implies r_{n+1} - r_n = \frac{\tilde r-r}{(n+2)(n+1)}, 
\end{equation*}
whence $r = r_0 < r_1 < \dots < r_n \nearrow \tilde r$. Note that, if there exists $n\in\mathbb{N}$ such that $\Phi(r_n) \leq 2C_1$, then since $\Phi$ is increasing, we get $\Phi(r)=\Phi(r_0) \leq \Phi(r_n) \leq 2C_1$, and we are back to case (a), where there is nothing to prove. Thus, henceforth we assume $\Phi(r_n) > 2C_1$ for all $n\in\mathbb{N}$, whence \eqref{eq:iterate bombieri giusti} holds with $r=r_n$, $R=r_{n+1}$ ($n\in\mathbb{N}_{\geq 0}$). 
By induction on \eqref{eq:iterate bombieri giusti}, 
    \begin{equation}\label{eq:induction on the bound bombieri giusti}
    \begin{aligned}
        \Phi(r) = \Phi(r_0) &\leq (\frac{3}{4})^n \Phi(r_n) + C \sum_{j=0}^{n-1} (\frac{3}{4})^{j} \frac{1}{(r_{j+1} - r_j)^{2\beta}}, 
    \end{aligned}
    \end{equation}
    and, using the assumption \eqref{eq:first est adapt bombieri giusti} with $R=\tilde r$, the first term on the right-hand side decays: 
      \begin{equation*}
  \begin{aligned}
      (\frac{3}{4})^n  \Phi(r_n) =  (\frac{3}{4})^n \log\Vert \varphi \Vert_{L^\infty(\mathcal{Q}_{r_n})} &\leq (\frac{3}{4})^n \log\left( \left(\frac{C_0 }{(\tilde r-r_n)^\beta}\right)^{\frac{1}{p}} \Vert \varphi \Vert_{L^p(\mathcal{Q}_1)} \right) \\ 
      &=  (\frac{3}{4})^n \big( \frac{1}{p} \log C_0 +  \log\Vert \varphi \Vert_{L^p(\mathcal{Q}_1)}  \big) + \frac{\beta}{p}(\frac{3}{4})^n \log (\frac{1+n}{\tilde r-r}) , 
  \end{aligned}
    \end{equation*}
so $ (\frac{3}{4})^n  \Phi(r_n) \to 0$,  for any fixed $p \in (0,1)$; note that $\Vert \varphi \Vert_{L^p(\mathcal{Q}_1)} \leq |\mathcal{Q}_1|^{\nicefrac{1}{p}-1}\Vert \varphi \Vert_{L^1(\mathcal{Q}_1)} < \infty$ for $p \in (0,1)$ by H\"older's inequality. By letting $n\to \infty$ in \eqref{eq:induction on the bound bombieri giusti}, we recover 
       \begin{equation*}
     \Phi(r) \leq   C \sum_{j=0}^{\infty} (\frac{3}{4})^j \frac{1}{(r_{j+1} - r_j)^{2\beta}} = \underbrace{C\bigg(   \sum_{j=0}^{\infty} (\frac{3}{4})^{j} (j+2)^{2\beta}(j+1)^{2\beta} \bigg)}_{=:C'} (\tilde r-r)^{-2\beta}, 
    \end{equation*}
where $C'$ is independent of $r,\tilde{r}$; note that the infinite sum manifestly converges for all $\beta > 0$ by the d'Alembert ratio test. It then follows (recalling the non-negativity of $\varphi$) that 
\begin{equation*}
    \Vert \varphi \Vert_{L^\infty(\mathcal{Q}_r)} = \esssup_{\mathcal{Q}_r} \exp \log \varphi = e^{\Phi(r)} \leq e^{C' (\tilde r-r)^{-2\beta}}. 
\end{equation*}
We set $C_2 = C'$, which depends only on $C_0,C_1,\beta$, and the proof of \eqref{eq:bombieri giusti final bound i} is complete.

\smallskip 

\noindent 2.i. \textit{Proof of \eqref{eq:iterate bombieri giusti}$:$ $L^p$ estimate for the logarithm}. Our goal in this step is to obtain (for $\varphi$ the function in the statement of Lemma \ref{lem:bombieri giusti} with $r \in (0,\tilde r)$ and $\Phi(r)$ satisfying case (b) at the beginning of the proof) the following estimate: 
 \begin{equation}\label{eq:tricky bombieri giusti ii}
        \begin{aligned}
           \Vert \varphi \Vert^{p_R}_{L^{p_R}(\mathcal{Q}_R)} \leq C_3 e^{\frac{1}{2}p_R\Phi(R)} \qquad \forall R \in (r,\tilde r], 
        \end{aligned}
    \end{equation}
where the positive constant $C_3$ is independent of $p_R,\beta,r,\tilde r,R$, and $p_R $ is prescribed by 
\begin{equation}\label{eq:choice of p bombieri giusti}
    p_R = \frac{2}{\Phi(R)}\log \frac{\Phi(R)}{2C_1}. 
\end{equation}
Note that $p_R \in (0,1)$ since the assumption $\Phi(r) > 2C_1$ with $C_1 > e^{-1}$ implies $\Phi(R) \geq \Phi(r) > 2C_1$ since $\Phi$ is increasing, and (since $\sup_{s \in (1,\infty)} s^{-1}\log s = e^{-1}$) 
\begin{equation}\label{eq:p of r is less than one}
   0 < \frac{2}{\Phi(R)}\log \frac{\Phi(R)}{2C_1} = \frac{1}{C_1} \frac{2C_1}{\Phi(R)}\log \frac{\Phi(R)}{2C_1} < e \sup_{s \in (1,\infty)} s^{-1}\log s = 1, 
\end{equation}
Define $
    A_R := \{(t,y) \in \mathcal{Q}_R : \, \log \varphi(t,y) > \frac{1}{2}\Phi(R)\} $ and $ B_R := \{(t,y) \in \mathcal{Q}_R : \, \log \varphi(t,y) \leq \frac{1}{2}\Phi(R) \}$. For all $p > 0$, it holds $\esssup_{\mathcal{Q}_R} \varphi^p =  \esssup_{\mathcal{Q}_R} \exp(p \log \varphi )  = e^{p \Phi(R)}$, whence 
    \begin{equation}\label{eq:tricky bombieri giusti}
        \begin{aligned}
           \Vert \varphi \Vert^p_{L^p(\mathcal{Q}_R)} \!=\!  \iint_{A_R}\!\!\!\varphi^p  \!+\! \iint_{B_R} \!\!\!\varphi^p   \!\leq\! |A_R| e^{p \Phi(R)} \!+\! |B_R| e^{ \frac{p}{2}\Phi(R)} \!\leq\! \frac{2 C_1}{\Phi(R)}e^{p \Phi(R)} \!+\! |\mathcal{Q}_1| e^{\frac{p}{2}\Phi(R)}, 
        \end{aligned}
    \end{equation}
where we used \eqref{eq:assump bombieri giusti on log set} for the first term, and bluntly bounded $|B_R| \leq |\mathcal{Q}_R| \leq |\mathcal{Q}_1|$. We emphasise that the above estimate holds for all possible choices of $p > 0$ and $R \in (r,\tilde r]$; including if we choose $p$ to depend on $R$ as per \eqref{eq:choice of p bombieri giusti}. Next, we estimate the first term on the right-hand side of \eqref{eq:tricky bombieri giusti} by the final term. Setting $p=p_R$ as per \eqref{eq:choice of p bombieri giusti}, we get 
\begin{equation*}
    \frac{2 C_1}{\Phi(R)}e^{p_R \Phi(R)} = e^{\frac{p_R}{2} \Phi(R)}, 
\end{equation*}
and we deduce from \eqref{eq:tricky bombieri giusti} that, defining the positive constant $C_3=1+|\mathcal{Q}_1|$ independent of $p,d,\beta,r,R,\tilde r$, there holds \eqref{eq:tricky bombieri giusti ii} with $p_R$ prescribed by \eqref{eq:choice of p bombieri giusti}.

\smallskip 

\noindent 2.ii. \textit{Proof of \eqref{eq:iterate bombieri giusti}$:$ estimate on $\Phi(r)$}. Using \eqref{eq:first est adapt bombieri giusti} with the choice $p=p_R \in (0,1)$ and 
\eqref{eq:tricky bombieri giusti ii}, the monotonicity of $\log$ implies, since $ \Phi(r) = \frac{1}{p_R} \log (\Vert \varphi \Vert_{L^\infty(\mathcal{Q}_r)}^{p_R} )$, 
\begin{equation}\label{eq:almost there bombieri giusti}
    \begin{aligned}
        \Phi(r) \leq \frac{1}{p_R}\log \bigg( \frac{C_0}{(R-r)^\beta} C_3 e^{\frac{1}{2}p_R\Phi(R)} \bigg) &\leq \frac{1}{2}\Phi(R) + \frac{1}{p_R} \log \Big( \frac{C_0 C_3}{(R-r)^\beta} \Big) \\ 
        &= \frac{1}{2}\Phi(R) \bigg( 1 + \frac{\log \big( \frac{C_0 C_3}{(R-r)^\beta} \big)}{\log \big(\frac{\Phi(R)}{2C_1}\big)}\bigg), 
    \end{aligned}
\end{equation}
where we used the explicit representation \eqref{eq:choice of p bombieri giusti} for $p_R$ to obtain the final line. We divide the analysis into two cases. The first case is when it holds: ${\log ( \frac{C_0 C_3}{(R-r)^\beta} )}/{\log (\frac{\Phi(R)}{2C_1})} \leq \frac{1}{2}$, 
in which case we deduce from \eqref{eq:almost there bombieri giusti} that it holds $\Phi(r) \leq \frac{3}{4}\Phi(R)$, and hence \eqref{eq:iterate bombieri giusti} is satisfied. On the other hand, if instead it holds ${\log ( \frac{C_0 C_3}{(R-r)^\beta} )}/{\log (\frac{\Phi(R)}{2C_1})} > \frac{1}{2}$, then (using again $\Phi(R) \geq \Phi(r) > 2C_1$) we deduce $\Phi(r) \leq \Phi(R) < \frac{2 C_1 C_0^2 C_3^2}{(R-r)^{2\beta}}$, where we used that $\Phi$ is an increasing function of the radius to write the first inequality. In either case, \eqref{eq:iterate bombieri giusti} holds with $C(C_0,C_1)=2C_1 C_0^2 (1+|\mathcal{Q}_1|)^2$ using the definition of $C_3$. The proof is complete.     
\end{proof}

\section{Moser-type Local $L^\infty$ Estimates}\label{sec:moser}

In this section, we derive the usual Moser-type local $L^\infty$ estimates on the quantity $(1-\rho)$ for which we wish to show infinite speed of propagation; here, \emph{local} means local-in-time, \textit{i.e.~}away from $t=0$. This will consist in proving Caccioppoli-type inequalities for this quantity (and its reciprocal $(1-\rho)^{-1}$) and estimates on the measure of the super-/sub-level sets of its logarithm; in essence, we wish to show that $(1-\rho)$ satisfies the assumptions of the Bombieri--Giusti Lemma \ref{lem:bombieri giusti}. We follow the presentation in \cite[\S 3.2]{BonforteDolbeaultNazaretSimonov}, which is based on Moser's iteration, as opposed to the De Giorgi method used in \textit{e.g.~}\cite[Proposition 5.2]{reg2}. 

The main focus of the subsections that follow is the quantity $u:= 1-\rho$, for which we wish to show an instantaneous positivity result. We rewrite \eqref{eq:rho eq} in terms of $u$: 
\begin{equation}\label{eq:u eq}
    \partial_t u - \dv(\p u) = \Delta u. 
\end{equation}
Since $0 \leq \rho \leq 1$ a.e., we have $0 \leq u \leq 1$ a.e. Moreover, the regularity \eqref{eq:initial upgrade dt rho not the real one} implies 
\begin{equation}\label{eq:initial upgrade dt rho u}
\partial_t u , \nabla^2 u \in L^{\nicefrac{3}{2}}((t,T) \!\times\! \T^2) \qquad \forall 0<t<T. 
\end{equation}

\subsection{Moser lower bound }

The main estimate of this subsection is the following (Lemma \ref{lem:moser 1/u}). For clarity of presentation, the analysis is performed for weak solutions of \eqref{eq:u eq} studied on the parabolic half-cylinder $Q_r^-$ centered at $t=0$ (\textit{cf.~}\eqref{eq:parabolic cylinder}); however, by translation these estimates hold on any cylinder of the form $Q^-_r(t)=(t-r,t)\times\T^2$.

As per \eqref{eq:initial upgrade dt rho u}, we assume that the weak solutions $u$ of \eqref{eq:u eq} satisfy: 
\begin{equation}\label{eq:initial upgrade dt rho}
   0 \leq u \leq 1 \quad \text{a.e.}, \qquad \partial_t u, \nabla^2 u \in L^{\nicefrac{3}{2}}_\loc(Q), 
\end{equation}
where $Q$ is any open cylinder which is compactly contained within the domain of definition of the solution (which we do not specify explicitly, for succinctness). 

\begin{lemma}[$L^\infty$ bound $(u+\varepsilon)^{-1}$]\label{lem:moser 1/u}
    Let $u$ be a weak solution of \eqref{eq:u eq} satisfying \eqref{eq:initial upgrade dt rho}, $\varepsilon,p \in (0,1)$, $0 < r < R \leq 1$. There exists a constant $C$ independent of $p,r,R,\varepsilon$ such that 
    \begin{equation}\label{eq:est on 1/u for future times}
   \begin{aligned}
     \Vert (u+\varepsilon)^{-1} \Vert_{L^{\infty}(Q_r^-)} \leq \left(\frac{C }{(R-r)^2} \right)^{\frac{1}{p}}\Vert (u+\varepsilon)^{-1} \Vert_{L^{p}(Q_R^-)}.
   \end{aligned} 
\end{equation}
\end{lemma}
\begin{proof}

The proof is divided into three steps. 

\smallskip 

\noindent 1. \textit{Caccioppoli-type inequality}. Let $q > 0$ and, for all $\varepsilon \in (0,1)$, define $   w_\varepsilon := (u+\varepsilon)^{-\nicefrac{q}{2}}. $ Let $0<r < R \leq 1$. Our objective in this step is to show  (for $C$ independent of $q,r,R,\varepsilon$): 
\begin{equation}\label{eq:first checkpoint moser}
   \begin{aligned}
     \esssup_{s \in (-r,0)}\Vert w_\varepsilon(s) \Vert_{L^2(\T^2)}^2 \!+\! \int^{0}_{-r}  \Vert \nabla w_\varepsilon(s)\Vert_{L^2(\T^2)}^2 \d s \leq \frac{C(1+q^2)}{R\!-\!r}   \int_{-R}^0 \Vert w_\varepsilon(s ) \Vert_{L^2(\T^2)}^2 \d s.
   \end{aligned} 
\end{equation}
To this end, let $\psi \in C^\infty_c((-R,0])$   be such that 
\begin{equation}\label{eq:test function time psi}
    0 \leq \psi \leq 1 \text{ on } [-1,0], \quad \psi|_{[-1,-R]} = 0 , \quad \psi|_{[-r,0]} = 1,  \quad \Vert \partial_t \psi \Vert_{L^\infty((0,1))} \leq \frac{C}{R-r},
    \end{equation}
    for some $C$ independent of $r,R$; $\psi$ is only a function of $t$. Recalling the discussion that follows \eqref{eq:u eq} and \eqref{eq:initial upgrade dt rho}, we have $w_\varepsilon \in L^\infty((-1,1)\times\T^2) \cap L^2((-1,1);H^1(\T^2))$; furthermore, $$\phi := -q (u+\varepsilon)^{-q-1} \psi^2 = -q w_\varepsilon^{2+\frac{2}{q}} \psi^2$$ also belongs to this space, whence it is an admissible test function. We therefore insert $\phi$ as a test function into the weak formulation of \eqref{eq:u eq} and get 
\begin{equation*}
   \begin{aligned}
       \langle \partial_t u , \phi \rangle = -q\langle \partial_t (w_\varepsilon^{-\frac{2}{q}}) , w^{2+\frac{2}{q}}_\varepsilon \psi^2 \rangle = 2 \langle w_\varepsilon \partial_t w_\varepsilon , \psi^2 \rangle = \frac{\der}{\der t} \int_{\T^2} |\psi w_\varepsilon|^2 \d x - 2 \int_{\T^2} \psi \partial_t \psi w^2_\varepsilon \d x, 
   \end{aligned} 
\end{equation*}
where we used \eqref{eq:initial upgrade dt rho} to justify the Sobolev chain rule (\textit{cf.}~\cite[\S 5.10]{Evans}); we emphasise that we are \emph{not} using the Lions--Magenes Lemma here. Similarly, it holds 
\begin{equation*}
   \begin{aligned}
        - \langle \dv(\p u) , \phi \rangle &= -2(q\!+\!1)\! \int_{\T^2} \psi^2 \p (w_\varepsilon^{-\frac{2}{q}} \!-\! \varepsilon) \! \cdot\! w^{1+\frac{2}{q}}_\varepsilon \nabla w_\varepsilon \d x \\ &= -2(q\!+\!1) \! \int_{\T^2} \psi^2 w_\varepsilon \p  \! \cdot\! \nabla w_\varepsilon \d x + 2(q\!+\!1) \varepsilon \! \int_{\T^2} \psi^2 w^{1+\frac{2}{q}}_\varepsilon \p \cdot\! \nabla w_\varepsilon \d x, 
   \end{aligned}
\end{equation*}
and $\langle \Delta u , \phi \rangle = q\int_{\T^2} \psi^2 \nabla (w_\varepsilon^{-\frac{2}{q}}) \!\cdot\! \nabla (w^{2+\frac{2}{q}}_\varepsilon) \d x = -\frac{4}{q}(q+1) \int_{\T^2} |\nabla(\psi w_\varepsilon)|^2 \d x.$ Observe that 
\begin{equation*}
    \bigg|   \varepsilon\! \int_{\T^2}  \psi^2 w^{1+\frac{2}{q}}_\varepsilon   \p \cdot\! \nabla w_\varepsilon \d x \bigg| \leq   \varepsilon\! \int_{\T^2}  \underbrace{w^{\frac{2}{q}}_\varepsilon}_{\leq \varepsilon^{-1}}  \psi w_\varepsilon |\nabla (\psi w_\varepsilon) | \d x \leq   \int_{\T^2}  \psi w_\varepsilon |\nabla (\psi w_\varepsilon) | \d x,
\end{equation*}
where we also used $|\p|\leq 1$ by \eqref{eq:polarisation bound}, hence, using also Young's inequality, we obtain 
\begin{equation*}
   \begin{aligned}
      \frac{\der}{\der t} \Vert \psi  w_\varepsilon & \Vert_{L^2(\T^2)}^2 \!+\! 4(1\!+\! q^{-1}) \Vert\nabla(\psi w_\varepsilon)\Vert_{L^2(\T^2)}^2 \leq 4(q\!+\!1) \!\int \!\! \psi w_\varepsilon |\nabla (\psi w_\varepsilon)| \d x \!+\! 2\! \int\!\! \psi |\partial_t \psi| w^2_\varepsilon \d x \\ 
      &\leq 2q^2 (1+q^{-1})\Vert \psi w_\varepsilon\Vert^2_{L^2(\T^2)} + 2 (1+q^{-1})\Vert\nabla(\psi w_\varepsilon)\Vert_{L^2(\T^2)}^2 + \frac{2C}{R-r} \Vert w_\varepsilon \Vert^2_{L^2(\T^2)}. 
   \end{aligned} 
\end{equation*}
By bounding the first term on the right-hand side bluntly using $0\leq \psi \leq 1$, it follows that 
\begin{equation*}
   \begin{aligned}
      \frac{\der}{\der t} \Vert \psi w_\varepsilon \Vert_{L^2(\T^2)}^2 + (1\!+\! q^{-1}) \Vert\nabla(\psi w_\varepsilon)\Vert_{L^2(\T^2)}^2 \leq C \Big( q^2(1+q^{-1}) +\frac{1}{R-r} \Big) \Vert w_\varepsilon \Vert^2_{L^2(\T^2)}, 
   \end{aligned} 
\end{equation*}
where $C$ is independent of $q,r,R,\varepsilon$.  Integrating in time over $[-R,t]$, for any $t \in (-r,0)$, 
\begin{equation*}
   \begin{aligned}
     \Vert w_\varepsilon(t) \Vert_{L^2(\T^2)}^2 + \int_{-R}^t \int_{\T^2} |\nabla(\psi w_\varepsilon)|^2 \d x \d s \leq  \frac{C(1+q^2)}{R-r}  \! \int_{-R}^t \Vert w_\varepsilon(s) \Vert_{L^2(\T^2)}^2 \d s, 
   \end{aligned} 
\end{equation*}
where we used that $\psi(t) = 1$ for such $t$ and $\psi(-R)=0$ from \eqref{eq:test function time psi}, $1 \leq 1 +q^{-1}$, and Young's inequality to write $2q \leq 1 + q^2$. Using also the fact that $\psi$ does not depend on $x$ and $\psi \equiv 1$ in $[-r,0]$ to control the second term on the left-hand side of the previous inequality, as well as $0<R-r<1$, we deduce \eqref{eq:first checkpoint moser} as desired.

\smallskip 

\noindent 2. \textit{Recursive estimate}. Our goal in this step is to obtain (for $C$ independent of $q,r,R,\varepsilon$): 
\begin{equation}\label{eq:second checkpoint moser}
   \begin{aligned}
     \Vert (u+\varepsilon)^{-1} \Vert_{L^{2q}(Q_r^-)} \leq \Big( \frac{C   (1+q^2)}{R\!-\!r} \Big)^{\frac{1}{q}} \Vert (u+\varepsilon)^{-1} \Vert_{L^q(Q_R^-)} , 
   \end{aligned} 
\end{equation}
for all choices of $0<r<R \leq 1$ with $R-r < 1$. 

The standard interpolation inequality of Lemma \ref{lem:dibenedetto classic} implies 
\begin{equation*}
    \begin{aligned}
        \Vert w_\varepsilon \Vert_{L^4(Q_r^-)} \leq C  \Big( \Vert w_\varepsilon \Vert_{L^\infty(-r,0;L^2(\T^2))} + \Vert \nabla w_\varepsilon \Vert_{L^2((-r,0)\times\T^2)}  \Big), 
    \end{aligned}
\end{equation*}
where the positive constant $C$ is independent of $r,R,q,\varepsilon$. By inserting  into \eqref{eq:first checkpoint moser}, we get 
\begin{equation*}
   \begin{aligned}
     \Vert w_\varepsilon \Vert_{L^4(Q_r^-)} \leq C  \Big(\frac{1+q^2}{R\!-\!r} \Big)^{\frac{1}{2}} \Vert w_\varepsilon \Vert_{L^2(Q_R^-)}, 
   \end{aligned} 
\end{equation*}
which is a \emph{reverse H\"older inequality}, and we obtain \eqref{eq:second checkpoint moser} by rewriting the above in terms of $u+\varepsilon$, where the constant $C$ has changed but is still independent of $q,r,R,\varepsilon$.

\smallskip 

\noindent 3. \textit{Moser's iterations}. As per the statement, let $0 < r < R \leq 1$ be fixed. Define the radii 
\begin{equation}\label{eq:sequence of radii}
r_0 = R, \qquad r_n := R - (R-r) \sum_{k=1}^n 2^{-k}  \qquad \forall n\in\mathbb{N}; 
\end{equation}
note that the sequence is monotonically decreasing, and $r_n \to r$ as $n\to\infty$. Furthermore, 
\begin{equation}\label{eq:spacing between radii}
    r_{n} - r_{n+1} = (R-r) 2^{-n-1}   \qquad \forall n \in \mathbb{N}; 
\end{equation}
note $0< r_{n+1} < r_n \leq 1$ and $0<r_n - r_{n+1} < 1$ for all $n$. Also define the powers $q_n := p 2^n$ for all $n $, where we note $q_0 = p$ and $q_n \to \infty$; where $p \in (0,1)$ is given in the statement.

By applying the result of \eqref{eq:second checkpoint moser} with $R=r_n$, $r=r_{n+1}$, $q=q_n$, we deduce 
\begin{equation}\label{eq:the first kappa n}
   \begin{aligned}
     \Vert (u\!+\!\varepsilon)^{-1} \Vert_{L^{p 2^{n+1}}(Q_{r_{n+1}}^-)} \leq \underbrace{\Big( \frac{C (1\!+\! p^2 2^{2n}) }{r_{n}\!-\!r_{n+1}} \Big)^{\frac{1}{p 2^{n}}}}_{=:\kappa_n} \Vert (u\!+\!\varepsilon)^{-1} \Vert_{L^{p 2^n}(Q_{r_n}^-)} \quad \forall n \in \mathbb{N}. 
   \end{aligned} 
\end{equation}
Note that, for all $n$ sufficiently small such that $p 2^n \leq 1$, it holds $\kappa_n \leq ( \frac{C 2^{n+2} }{R-r} )^{\frac{1}{p 2^{n}}} = ( \frac{4C }{R-r} )^{\frac{1}{p 2^{n}}}2^{\frac{n}{p 2^n}}, $ while for all $n$ sufficiently large such that $p 2^n \geq 1$, we obtain $\kappa_n \leq ( \frac{C p^2  2^{3n+2} }{R-r} )^{\frac{1}{p 2^{n}}} = ( \frac{4C p^2 }{R-r} )^{\frac{1}{p 2^{n}}}2^{\frac{3n}{p 2^n}}, $ whence, since $p \in (0,1)$, we deduce in either case that 
\begin{equation*}
    \kappa_n \leq \Big( \frac{C}{R-r} \Big)^{\frac{1}{p 2^{n}}}2^{\frac{3n}{p 2^n}} \qquad \forall n \in \mathbb{N}, 
\end{equation*}
where $C$ has changed, but remains independent of $p,r,R,\varepsilon,n$. By iteration on \eqref{eq:the first kappa n}, 
\begin{equation}\label{eq:just before moser}
   \begin{aligned}
     \Vert (u\!+\!\varepsilon)^{-1} \Vert_{L^{p 2^{n}}(Q_{r_n}^-)} \! \leq \! \bigg( \! \underbrace{\prod_{k=0}^{n-1} \Big(\frac{C}{R \! -\! r}\Big)^{\frac{1}{p 2^k}} 2^{\frac{3k}{p 2^k}} }_{=: A_n} \! \bigg) \Vert (u \! + \! \varepsilon)^{-1} \Vert_{L^{p}(Q_R^-)} \qquad \forall n \geq 1. 
   \end{aligned} 
\end{equation}
We study the convergence of the coefficients $A_n$. We have 
\begin{equation*}
   \begin{aligned} \log A_n = \sum_{k=0}^{n-1} \log \bigg[ \Big(\frac{C }{R-r}\Big)^{\frac{1}{p 2^k}}  2^{\frac{3k}{p2^k}} \bigg]  &= \frac{1}{p}\log \Big(\frac{C  }{R-r}\Big) \sum_{k=0}^{n-1} {2^{-k}}  + \frac{3}{p} \log 2 \sum_{k=0}^{n-1} k {2^{-k}} \\ 
   &\longrightarrow \frac{2}{p} \log \Big( \frac{C  }{R-r} \Big) + \frac{1}{p} \log \bar{C}, 
   \end{aligned}
\end{equation*}
where $\bar{C} = 2^{3 \sum_{n \geq 0} n 2^{-n}}$ is independent of $p,r,R,\varepsilon$. It follows that $  A_n \nearrow A = \left(\frac{C }{(R-r)^2}  \right)^{\frac{1}{p}} , $ for some $C$ independent of $p,r,R,\varepsilon$. Passing to the limit $n\to\infty$ in \eqref{eq:just before moser} yields   \eqref{eq:est on 1/u for future times}. 
\end{proof}

\subsection{Moser upper bound}

The following result is not sharp; especially in terms of the quintic power of $R-r$ in the denominator, however it is sufficient for our purposes. 
Again, for clarity of exposition, the estimates are written for weak solutions of \eqref{eq:u eq} studied on the parabolic cylinders $Q_r$ centered at $t=0$ (\textit{cf.~}\eqref{eq:parabolic cylinder}); however, by translation they hold on any cylinder of the form $Q_r(t)=(t-r,t+r)\times\T^2$. 

\begin{lemma}[$L^\infty$ bound $u+\varepsilon$]\label{lem:moser u upper}
    Let $u$ be a weak solution of \eqref{eq:u eq} satisfying \eqref{eq:initial upgrade dt rho}, $\varepsilon ,p \in (0,1)$,  $0 < r < R \leq 1$. There exists a  constant $C$ independent of $p,r,R,\varepsilon$ such that 
    \begin{equation}\label{eq:Linfty e on cylinder by moser}
   \begin{aligned}
     \Vert u+\varepsilon \Vert_{L^{\infty}(Q_r)} \leq \left(\frac{C}{(R\!-\!r)^5} \right)^{\frac{1}{p}}\Vert u+\varepsilon \Vert_{L^{p}(Q_R)}. 
   \end{aligned} 
\end{equation}
\end{lemma}

\begin{proof}
The proof is divided into several steps. 

\smallskip 

\noindent 1. \textit{Caccioppoli-type inequality}. Define $w_\varepsilon := (u+\varepsilon)^{\nicefrac{q}{2}}$ for $q>0$ such that $q \neq 1$, and 
\begin{equation}\label{eq:delta def in proof of moser upper}
    \delta := \big| 1 - q^{-1} \big|. 
\end{equation}
Our objective in this step is to prove 
\begin{equation}\label{eq:first checkpoint moser bis}
   \begin{aligned}
     \esssup_{s \in (-r,r)}\Vert w_\varepsilon(s) \Vert_{L^2(\T^2)}^2 \!+\!  \int_{-r}^{r} \!\! \! \Vert \nabla w_\varepsilon(s)\Vert_{L^2(\T^2)}^2 \d s \!\leq\! C \Big(  \frac{1\!+\!q^2\!+\!\delta^{-1} }{R\!-\!r} \Big) \! \int_{-R}^{R} \!\!\! \Vert w_\varepsilon(s) \Vert_{L^2(\T^2)}^2 \d s. 
   \end{aligned} 
\end{equation}

Let $\psi \in C^\infty_c((-R,R))$ be such that 
\begin{equation}\label{eq:test function time psi for moser upper}
    0 \leq \psi \leq 1 , \quad \psi|_{[-r,r]}=1, \quad \psi|_{[-1,1]\setminus[-R,R]}=0, \quad \Vert \partial_t \psi \Vert_{L^\infty((-1,1))} \leq \frac{C}{R-r},
    \end{equation}
    for some constant $C$ independent of $r,R$. As before, $\psi$ is only a function of $t$, and again recalling the discussion that follows \eqref{eq:u eq} and \eqref{eq:initial upgrade dt rho}, 
    $\phi := q (u+\varepsilon)^{q-1} \psi^2 = q w_\varepsilon^{2-\frac{2}{q}} \psi^2$ 
    is an admissible test function. We  insert $\phi$ into the weak formulation of \eqref{eq:u eq} and get 
\begin{equation*}
   \begin{aligned}
       \langle \partial_t u , \phi \rangle = q\langle \partial_t (w_\varepsilon^{\frac{2}{q}}) , w^{2-\frac{2}{q}}_\varepsilon \psi^2 \rangle = 2 \langle w_\varepsilon \partial_t w_\varepsilon , \psi^2 \rangle = \frac{\der}{\der t} \int_{\T^2} |\psi w_\varepsilon|^2 \d x - 2 \int_{\T^2} \psi \partial_t \psi w^2_\varepsilon \d x, 
   \end{aligned} 
\end{equation*}
where we use \eqref{eq:initial upgrade dt rho}, as well as $ - \langle \dv(\p u) , \phi \rangle = 2(q\!-\!1)\! \int_{\T^2} \psi^2 \p (w_\varepsilon^{\frac{2}{q}} \!-\! \varepsilon) \! \cdot\! w^{1-\frac{2}{q}}_\varepsilon \nabla w_\varepsilon \d x ,$
and $\langle \Delta u , \phi \rangle = - q\int_{\T^2} \psi^2 \nabla (w_\varepsilon^{\frac{2}{q}}) \!\cdot\! \nabla (w^{2-\frac{2}{q}}_\varepsilon) \d x = 4(q^{-1}-1) \int_{\T^2} |\nabla(\psi w_\varepsilon)|^2 \d x. $ In turn, we get 
\begin{equation}\label{eq:equality}
   \begin{aligned}
      -\frac{\der}{\der t} & \int_{\T^2} |\psi w_\varepsilon|^2 \d x + 4(q^{-1}-1) \int_{\T^2} |\nabla(\psi w_\varepsilon)|^2 \d x \\ 
      =& 2(q\!-\!1) \!\int_{\T^2}\!\! \psi^2 w_\varepsilon \p  \! \cdot\! \nabla w_\varepsilon \d x \!-\! 2(q\!-\!1) \varepsilon\! \int_{\T^2}\!\! \psi^2 w^{1-\frac{2}{q}}_\varepsilon   \p \cdot\! \nabla w_\varepsilon \d x \!-\! 2 \! \int_{\T^2}\!\! \psi \partial_t \psi w^2_\varepsilon \d x; 
   \end{aligned} 
\end{equation}
the second term on the right-hand side is bounded by (using also \eqref{eq:polarisation bound} to control $\p$)
\begin{equation*}
    \bigg| \varepsilon \int_{\T^2}\! \psi^2 w^{1-\frac{2}{q}}_\varepsilon   \p \cdot\! \nabla w_\varepsilon \d x \bigg| \leq \varepsilon \int_{\T^2}\! \underbrace{w^{-\frac{2}{q}}_\varepsilon}_{\leq \varepsilon^{-1}} \psi w_\varepsilon  |\nabla (\psi w_\varepsilon)| \d x \leq \int_{\T^2}\! \psi w_\varepsilon |\nabla (\psi w_\varepsilon)| \d x. 
\end{equation*}
 The rest of the proof of \eqref{eq:first checkpoint moser bis} is divided into two substeps (depending on the value of $q$).

\smallskip 

\noindent 1.i. \textit{Parabolic estimate: case $q \in (0,1)$}. In this case $q^{-1}-1>0$, and \eqref{eq:equality} yields 
\begin{equation*}
   \begin{aligned}
      -\frac{\der}{\der t} \Vert \psi  w_\varepsilon & \Vert_{L^2(\T^2)}^2 \!+\! 4(q^{-1}\!-\!1) \Vert\nabla(\psi w_\varepsilon)\Vert_{L^2(\T^2)}^2 \leq 4(1\!-\!q) \!\int \!\! \psi w_\varepsilon |\nabla (\psi w_\varepsilon)| \d x   \!+\! 2\! \int\!\! \psi |\partial_t \psi| w^2_\varepsilon \d x \\ 
      &\leq 2q^2(q^{-1}-1)\Vert \psi w_\varepsilon\Vert^2_{L^2(\T^2)} + 2(q^{-1}-1)\Vert\nabla(\psi w_\varepsilon)\Vert_{L^2(\T^2)}^2 + \frac{2C}{R-r} \Vert w_\varepsilon \Vert^2_{L^2(\T^2)}, 
   \end{aligned} 
\end{equation*}
where we used Young's inequality in the final line. Using $0 \leq \psi \leq 1$, it follows that 
\begin{equation*}
   \begin{aligned}
     - \frac{\der}{\der t} \Vert \psi w_\varepsilon \Vert_{L^2(\T^2)}^2 \!+\! (q^{-1}\!-\!1) \Vert\nabla(\psi w_\varepsilon)\Vert_{L^2(\T^2)}^2 \leq C\Big(q^2 (q^{-1}\!-\!1)\!+\!\frac{1}{R\!-\!r} \Big) \Vert w_\varepsilon \Vert^2_{L^2(\T^2)}, 
   \end{aligned} 
\end{equation*}
where the positive constant $C$ is independent of $q,r,R,\varepsilon$. Note that since $q \in (0,1)$ we have $q^2(q^{-1}-1) \leq q^2 q^{-1} = q < 1$. By rewriting the left-hand side in terms of $\delta$ and integrating in time over the interval $[t,R]$, for any $t \in (-r,r)$, we get 
\begin{equation*}
   \begin{aligned}
     \Vert w_\varepsilon(t) \Vert_{L^2(\T^2)}^2 + \delta \int_{t}^{R} \! \int_{\T^2} \! |\nabla(\psi w_\varepsilon)|^2 \d x \d s \leq \frac{C}{R\!-\!r}  \! \int_{t}^{R} \! \Vert w_\varepsilon(s) \Vert_{L^2(\T^2)}^2 \d s, 
   \end{aligned} 
\end{equation*}
where we used that $\psi(t) = 1$ for such $t$, $\psi(R)=0$ from \eqref{eq:test function time psi for moser upper}, and $0<R-r<1$. Thus, by considering each term on the left-hand side separately, we deduce for $q \in (0,1)$: 
\begin{equation*}
   \begin{aligned}
     \Vert w_\varepsilon(t) \Vert_{L^2(\T^2)}^2 + \int_{t}^{R} \! \int_{\T^2} \! |\nabla(\psi w_\varepsilon)|^2 \d x \d s \leq C \Big( \frac{1+\delta^{-1}}{R\!-\!r}  \Big) \int_{t}^{R} \! \Vert w_\varepsilon(s) \Vert_{L^2(\T^2)}^2 \d s. 
   \end{aligned} 
\end{equation*}
One deduces \eqref{eq:first checkpoint moser bis} easily for this range of $q$. 

\smallskip 

\noindent 1.ii. \textit{Parabolic estimate: case $q>1$}. For $q>1$, applying Young's inequality to \eqref{eq:equality} gives 
\begin{equation*}
   \begin{aligned}
      \frac{\der}{\der t} & \int_{\T^2} |\psi w_\varepsilon|^2 \d x + 4(1-q^{-1}) \int_{\T^2} |\nabla(\psi w_\varepsilon)|^2 \d x \\ 
      &\leq 2q^2(1-q^{-1})\Vert \psi w_\varepsilon\Vert^2_{L^2(\T^2)} + 2(1-q^{-1})\Vert\nabla(\psi w_\varepsilon)\Vert_{L^2(\T^2)}^2 + \frac{2C}{R-r} \Vert w_\varepsilon \Vert^2_{L^2(\T^2)}, 
   \end{aligned} 
\end{equation*}
and thus, rewriting in terms of $\delta$, noting that $\delta \in (0,1)$ in this case, and $q^2(1-q^{-1})=q^2 \delta$, 
\begin{equation*}
   \begin{aligned}
    \delta  \frac{\der}{\der t} \Vert \psi w_\varepsilon \Vert_{L^2(\T^2)}^2 + \delta \Vert\nabla(\psi w_\varepsilon)\Vert_{L^2(\T^2)}^2 \leq C \big(q^2\delta +\frac{1}{R\!-\!r} \big) \Vert w_\varepsilon \Vert^2_{L^2(\T^2)}, 
   \end{aligned} 
\end{equation*}
where the constant $C$ is independent of $q,r,R,\varepsilon$. Dividing by $\delta$ and integrating in time over $[-R,t]$ for $t \in (-r,r)$, arguing as per the end of Step 1.i, we again obtain   \eqref{eq:first checkpoint moser bis}.

\smallskip 

\noindent 2. \textit{Recursive estimate}. Using $\Vert w_\varepsilon \Vert_{L^4(Q_r)} \leq C    ( \Vert w_\varepsilon \Vert_{L^\infty(-r,r;L^2(\T^2))} + \Vert \nabla w_\varepsilon \Vert_{L^2(Q_r)}  )$ from Lemma \ref{lem:dibenedetto classic} as per Step 2 of the proof of Lemma \ref{lem:moser 1/u} and arguing identically, we get 
\begin{equation}\label{eq:second checkpoint moser upper}
   \begin{aligned}
     \Vert u+\varepsilon \Vert_{L^{2q}(Q_r)} \leq \Big( \frac{C(1+ q^2 + \delta^{-1} )}{R\!-\!r} \Big)^{\frac{1}{q}} \Vert u+\varepsilon \Vert_{L^q(Q_R)}, 
   \end{aligned} 
\end{equation}
where $C$ is independent of $q,r,R,\varepsilon$. Next, as per Step 2 of the proof of \cite[Lemma 3.9]{BonforteDolbeaultNazaretSimonov}, set $$\gamma :=5/3,$$ and using the H\"older inequality, we get $
    \Vert u+ \varepsilon \Vert_{L^{\gamma q}(Q_r)} \leq |Q_r|^{(1-\frac{\gamma}{2})\frac{1}{\gamma q}} \Vert u + \varepsilon \Vert_{L^{2q}(Q_r)} \leq |Q_1|^{\frac{1}{10q}}\Vert u + \varepsilon \Vert_{L^{2q}(Q_r)}, $ which, using \eqref{eq:second checkpoint moser upper}, yields the \emph{reverse H\"older inequality}
\begin{equation}\label{eq:second checkpoint moser upper gamma}
   \begin{aligned}
     \Vert u+\varepsilon \Vert_{L^{\gamma q}(Q_r)} \leq \Big( \frac{C (1+q^2+\delta^{-1} ) }{R\!-\!r} \Big)^{\frac{1}{q}} \Vert u+\varepsilon \Vert_{L^q(Q_R)}, 
   \end{aligned} 
\end{equation}
where $C$ is independent of $q,r,R,\varepsilon$. The reason for introducing this exponent $\gamma$ is to avoid $q_n = 1$ for some $n$ for the sequence of powers in the iteration (\textit{cf.~}next step of the proof).

\smallskip 

\noindent 3. \textit{Moser's iterations}. Let $0 < r < R \leq 1$ be fixed as in the statement, and recall from \eqref{eq:sequence of radii} the sequence of radii $r_0 = R$, $r_n = R - (R-r) \sum_{k=1}^n 2^{-k} $, which satisfies $r_{n} -r_{n+1} = (R-r) 2^{-n-1}$ as per \eqref{eq:spacing between radii}. For the sequence of powers, we set 
\begin{equation*}
    n_0 := 1+ \floor{\frac{\log(\frac{1+\gamma}{2p})}{\log \gamma}}, \qquad q_0 := \frac{1+\gamma}{2 \gamma^{n_0}}, 
\end{equation*}
where $p$ is given in the statement, and $q_n := q_0 \gamma^n$ for all $n$. The definition of $n_0$ implies 
\begin{equation*}
    \frac{\log(\frac{1+\gamma}{2p})}{\log \gamma} < n_0 \leq 1 + \frac{\log(\frac{1+\gamma}{2p})}{\log \gamma} \implies    \frac{1+\gamma}{2p} <    \gamma^{n_0} \leq \gamma \frac{1+\gamma}{2p}, 
\end{equation*}
whence $q_0 < p \leq \gamma q_0$ and 
\begin{equation*}
    \gamma^{n_0} q_0 = \frac{1+\gamma}{2}=\frac{4}{3} > 1 > \frac{4}{5} = \frac{1+\gamma}{2\gamma} = \gamma^{n_0-1} q_0  , 
\end{equation*}
whence $|q_n \!-\! 1| \geq \min\{q_{n_0}\!-\!1,1\!-\!q_{n_0-1}\} = \frac{1}{5}$. Defining $\delta_n := |1-q_n^{-1}|$ as per \eqref{eq:delta def in proof of moser upper}, it holds 
\begin{equation*}
    \delta_n^{-1} = \frac{q_n}{|q_n-1|} \leq 5 q_n \leq \frac{5}{2}(1+q_n^2) \qquad \forall n. 
\end{equation*}
As such, by iterating \eqref{eq:second checkpoint moser upper gamma} along the sequence of radii $\{r_n\}$ and powers $\{q_n\}$, we get 
\begin{equation}\label{eq:kappa n round 2}
   \begin{aligned}
     \Vert u+\varepsilon \Vert_{L^{q_0 \gamma^{n+1}}(Q_{r_{n+1}})} \leq \underbrace{\Big( \frac{C (1+q_0^2 \gamma^{2n})}{r_{n}\!-\!r_{n+1}} \Big)^{\frac{1}{q_0 \gamma^{n}}} }_{=:\kappa_n}\Vert u+\varepsilon \Vert_{L^{q_0 \gamma^n}(Q_{r_n})} \qquad \forall n \in \mathbb{N}. 
   \end{aligned} 
\end{equation}
As per Step 3 of the proof of Lemma \ref{lem:moser 1/u}, we study the boundedness of $\kappa_n$: for all $n$ sufficiently small such that $q_0 \gamma^n \leq 1$, it holds $\kappa_n \leq ( \frac{C 2^{n+2}}{R-r} )^{\frac{1}{q_0 \gamma^{n}}} \leq ( \frac{4C}{R-r} )^{\frac{1}{q_0 \gamma^{n}}} 2^{\frac{n}{q_0 \gamma^n}},$
while for all $n$ sufficiently large such that $q_0 \gamma^n \geq 1$, we have $\kappa_n \leq ( \frac{C q_0^2 \gamma^{2n} 2^{n+2}}{R-r} )^{\frac{1}{q_0 \gamma^{n}}} \leq ( \frac{4Cq_0^2}{R-r} )^{\frac{1}{q_0 \gamma^{n}}} 2^{\frac{3n}{q_0 \gamma^n}} \leq ( \frac{4C}{R-r} )^{\frac{1}{q_0 \gamma^{n}}} 2^{\frac{3n}{q_0 \gamma^n}}, $ where we also used $0 < q_0 < p < 1$ and $\gamma < 2$. In either case, we obtain 
\begin{equation*}
    \kappa_n \leq \Big( \frac{C}{R-r} \Big)^{\frac{1}{q_0 \gamma^{n}}} 2^{\frac{3n}{q_0 \gamma^n}} \qquad \forall n \in\mathbb{N}, 
\end{equation*}
where $C$ is independent of $q_0,p,r,R,\varepsilon,n$. Returning to \eqref{eq:kappa n round 2} and iterating,  
\begin{equation*}
   \begin{aligned}
     \Vert u+\varepsilon \Vert_{L^{q_0 \gamma^{n+1}}(Q_{r_{n+1}})} \leq \underbrace{\bigg( \prod_{k=0}^n \Big( \frac{C }{R-r} \Big)^{\frac{1}{q_0 \gamma^{k}}} 2^{\frac{3k}{q_0 \gamma^k}} \bigg)}_{=: B_n} \Vert u+\varepsilon \Vert_{L^{q_0}(Q_{R})}, 
   \end{aligned} 
\end{equation*}
\textit{i.e.}, using H\"older's inequality and $|Q_1|^{\frac{1}{q_0}-\frac{1}{p}} \leq (1+|Q_1|)^{\frac{1}{q_0}-\frac{1}{p}}$, we get 
\begin{equation}\label{eq:still with Bn}
    \begin{aligned}
        \Vert u\!+\!\varepsilon \Vert_{L^{q_0 \gamma^{n+1}}(Q_{r_{n+1}})} \leq B_n |Q_1|^{\frac{1}{q_0}-\frac{1}{p}}  \Vert u \!+\! \varepsilon \Vert_{L^p(Q_R)} \leq B_n (1+|Q_1|)^{\frac{1}{q_0}} \Vert u \!+\! \varepsilon \Vert_{L^p(Q_R)}. 
    \end{aligned}
\end{equation}
Analogously to Step 3 of the proof of Lemma \ref{lem:moser 1/u}, we study the convergence of $\{B_n\}_n$:
\begin{equation*}
   \begin{aligned}
      \log B_n = \frac{1}{q_0} \log  \Big( \frac{C }{R-r} \Big) \sum_{k=0}^n  \gamma^{-k} + \frac{3}{q_0}\log 2 \sum_{k=0}^n k \gamma^{-k} \nearrow \frac{1}{q_0} \frac{\gamma}{\gamma-1} \log  \Big( \frac{C }{R-r} \Big) + \frac{1}{q_0}\log \bar{C}, 
   \end{aligned} 
\end{equation*}
with $\log \bar{C} = 3 \log 2 \sum_{n\geq 0} n \gamma^{-n}$, and thus using the continuity of the logarithm, it holds 
\begin{equation*}
    B_n \nearrow B = \Big( \frac{C }{(R-r)^{\frac{\gamma}{\gamma-1}}} \Big)^{\frac{1}{q_0}}, 
\end{equation*}
for some $C$ independent of $q_0,R,r$ and can be assumed greater than $1$. Returning to \eqref{eq:still with Bn}, 
\begin{equation*}
    \begin{aligned}
        \Vert u+\varepsilon \Vert_{L^{\infty}(Q_{r})} \leq \Big( \frac{C }{(R-r)^{\frac{\gamma}{\gamma-1}}} \Big)^{\frac{1}{q_0}}  \Vert u + \varepsilon \Vert_{L^p(Q_R)} \leq \Big( \frac{C }{(R-r)^{\frac{\gamma^2}{\gamma-1}}} \Big)^{\frac{1}{p}} \Vert u + \varepsilon \Vert_{L^p(Q_R)}
    \end{aligned}
\end{equation*}
using that $p \leq \gamma q_0$ and $C \geq 1 > R-r$, where $C$ is independent of $p,r,R,\varepsilon$ but is allowed to depend on $\gamma$. The result follows by bluntly bounding $(R-r)^{-\frac{\gamma^2}{\gamma-1}} = (R-r)^{-25/6} \leq (R-r)^{-5}$, using again that $R-r < 1$. 
\end{proof}

\section{Logarithmic Estimates}\label{sec:log}

The main result of this subsection is the following estimate on the measure of the super-/sub-level sets of the logarithm of weak solutions of \eqref{eq:u eq}. Our strategy closely follows the exposition in \cite[Lemma 3.11]{BonforteDolbeaultNazaretSimonov}, though our proof is quite different on account of being tailored to the space-periodic setting (as such, we do not need to localise in the $x$ variable).

\begin{lemma}[Super-/sub-level set estimate of the logarithm]\label{lem:log level sets}
Let $u$ be a weak solution of \eqref{eq:u eq} satisfying \eqref{eq:initial upgrade dt rho}, $\varepsilon,r \in (0,1)$. Define 
\begin{equation}\label{eq:aeps def}
    w_\varepsilon := -\log(u+\varepsilon), \qquad a_\varepsilon :=  \int_{\T^2}   \log\big(u(0,x)+\varepsilon\big) \d x 
\end{equation}
There exists a positive constant $C$ independent of $\varepsilon,r$ such that, for all $s > 0$, it holds 
    \begin{equation}\label{eq:sublevel log est}
       \begin{aligned}
           & \bigg| \Big\{ (t,x) \! \in Q_r^+ \!:  w_\varepsilon(t,x) \!+\! a_\varepsilon  \! > \! s   \Big\} \bigg| \leq \frac{C}{s}, \quad \bigg| \Big\{ (t,x) \! \in Q_r^-\! :  w_\varepsilon(t,x) \!+\! a_\varepsilon \! <\!- s   \Big\} \bigg| \leq \frac{C}{s}. 
       \end{aligned}
    \end{equation}
\end{lemma}

As per the previous subsections, for clarity of presentation our analysis is performed for weak solutions of \eqref{eq:u eq} studied on the parabolic half-cylinder $Q_r^\pm$ centered at $t=0$ (\textit{cf.~}\eqref{eq:parabolic cylinder}); however, by translation these estimates hold on any cylinder of the form $Q_r^\pm(t_0)$, with the quantity $a_\varepsilon$ in \eqref{eq:aeps def} modified to $\int_{\T^2} \log(u(t_0,x)+\varepsilon) \d x $. 

\begin{proof}
The proof is divided into three steps. 

\smallskip 

\noindent 1. \textit{A differential inequality for the mean}. Recall  \eqref{eq:initial upgrade dt rho}: this regularity implies that (in view of the chain rule for Sobolev functions, \textit{cf.}~\textit{e.g.}~\cite[\S 5.10]{Evans}) the formula 
\begin{equation*}
    \frac{\der}{\der t} \int_{\T^2} -\log(u+\varepsilon)   \d x = \langle \partial_t u , -(u+\varepsilon)^{-1}  \rangle 
\end{equation*}
is justified a.e.~$t$. By testing \eqref{eq:u eq} against $-(u+\varepsilon)^{-1} $, we get 
\begin{equation*}
    \begin{aligned}
        \int_{\T^2}   w_\varepsilon(t_2,x) \d x \!-\! \int_{\T^2}   w_\varepsilon(t_1,x) \d x \!+\! \int_{t_1}^{t_2} \!\! \int_{\T^2}  | \nabla w_\varepsilon |^2 \d x \d \tau = \!\int_{t_1}^{t_2}\!\!\! \int_{\T^2} \! \frac{u}{(u+\varepsilon)} \p \!\cdot\! \nabla w_\varepsilon \d x \d \tau. 
    \end{aligned}
\end{equation*}
for a.e.~$-1 < t_1<t_2 < 1$. Thus, by applying Young's inequality and using $\Vert \p \Vert_{L^\infty} \leq 1$, 
\begin{equation}\label{eq:first log est}
    \begin{aligned}
        \int_{\T^2}  w_\varepsilon(t_2,x) \d x - \int_{\T^2}   w_\varepsilon(t_1,x) \d x + \frac{1}{2} \int_{t_1}^{t_2} \int_{\T^2}  | \nabla w_\varepsilon |^2 \d x \d t \leq C_*(t_2-t_1), 
    \end{aligned}
\end{equation}
where $C_*=\frac{1}{2}|\T^2|$ is independent of $t_1,t_2,\varepsilon$. Poincar\'e--Wirtinger's inequality in $\T^2$ implies 
\begin{equation*}
    \int_{t_1}^{t_2} \int_{\T^2}   |w_\varepsilon(\tau,x) - \bar{w}_\varepsilon(\tau)|^2   \d x \d \tau \leq C_P^2 \int_{t_1}^{t_2} \int_{\T^2}    | \nabla w_\varepsilon |^2 \d x \d \tau, 
\end{equation*}
where the Poincar\'e constant $C_P>0$ is independent of $t_1,t_2,\varepsilon$, and $\bar{w}_\varepsilon(\tau) := \fint_{\T^2}   w_\varepsilon(\tau,x) \d x$ is the spatial mean. 
By dividing \eqref{eq:first log est} by $|\T^2|$, the previous inequality implies 
\begin{equation*}
    \begin{aligned}
        \frac{\bar{w}^\varepsilon(t_2) - \bar{w}^\varepsilon(t_1)}{t_2-t_1} +  \frac{(1/2C_P^2)}{t_2-t_1}\int_{t_1}^{t_2} \fint_{\T^2}   |w_\varepsilon(\tau,x) - \bar{w}_\varepsilon(\tau)|^2 \d x \d \tau \leq C_*, 
    \end{aligned}
\end{equation*}
where the positive constant $C_*$ is independent of $t_1,t_2,\varepsilon$. By letting $t_2 \to t_1$ and using the Lebesgue Differentiation Theorem, we recover the \emph{differential inequality for the mean}: 
\begin{equation}\label{eq:first log est mod}
    \begin{aligned}
        \frac{\der}{\der t}\bar{w}_\varepsilon(t) + \frac{1}{2C_P^2 } \fint_{\T^2}   |w_\varepsilon(t,x) - \bar{w}_\varepsilon(t)|^2 \d x \leq C_* \qquad \text{a.e.~}t \in (-1,1). 
    \end{aligned}
\end{equation}

\smallskip

\noindent 2. \textit{Super-level set estimates}. Our aim in this step is to adapt the strategy of Step 1 to recover a \emph{monotonically decreasing} quantity from the differential inequality \eqref{eq:first log est mod}; \textit{i.e.}~replacing the right-hand side $C_*$ with $0$ instead. To this end, we consider the function: 
\begin{equation}\label{eq:underline subsol log est i}
    \underline{w}(t,x) := w_\varepsilon(t,x) 
    - \bar{w}_\varepsilon(0) - C_* t, 
\end{equation}
where $C_*$ is the positive constant from \eqref{eq:first log est mod}. We also set 
\begin{equation}\label{eq:big W def in superlevel est}
    W(t) := \bar{w}_\varepsilon(t) - \bar{w}_\varepsilon(0) - C_* t; 
\end{equation}
note that $W(0) = 0$, $|\underline{w}(t,x) - W(t)| = |w_\varepsilon(t,x) - \bar{w}_\varepsilon(t)|$ and $\frac{\der}{\der t}W(t) = \frac{\der}{\der t}\bar{w}_\varepsilon(t) - C_*$. Substituting these computations in \eqref{eq:first log est mod}, we deduce 
\begin{equation}\label{eq:first log est mod W}
    \begin{aligned}
      \frac{\der}{\der t}W(t) \!+\!  \frac{1}{2C_P^2}\fint_{\T^2} |\underline{w}(t,x) \!-\! W(t)|^2 \d x \leq 0 \implies W(t) \leq W(0) = 0 \quad \text{a.e.~} t \in (0,1). 
    \end{aligned}
\end{equation}

Next, for all $s>0,t \in (0,1)$, define the super-level set $\Sigma_s(t)$ and $\sigma_s(t)$ by 
\begin{equation}\label{eq:define sets Sigma and sigma}
\Sigma_s(t) := \{x \in \T^2 : \underline{w}(t,x) > s \}, \qquad \sigma_s(t) := \{x \in \T^2 : w_\varepsilon(t,x) - \bar{w}_\varepsilon(0) > s \}; 
\end{equation}
our strategy will be to use estimates on $\Sigma_s(t)$ to control the measure of $\sigma_s(t)$. Indeed, for all $t \in (0,1)$, we verify directly the inclusions 
\begin{equation}\label{eq:relationship between Sigma and sigma set}
    \Sigma_s(t) \subset \sigma_s(t) \quad \forall s>0, \qquad \sigma_{2s}(t) \subset \Sigma_s(t) \quad \forall s \geq C_*. 
\end{equation}
Using \eqref{eq:first log est mod W}, direct computation (using that $W(t) \leq 0$) yields 
\begin{equation}
    \underline{w}(t,x) - W(t) > s - W(t) \geq s > 0 \quad \text{for all } x \in \Sigma_s(t). 
\end{equation}
Furthermore, using also $\frac{\der}{\der t}W(t) = - \frac{\der}{\der t}(s-W(t))$ and \eqref{eq:first log est mod W}, we get 
\begin{equation*}
   \frac{1}{2C_P^2|\T^2|} \int_{\Sigma_s(t)} \!\!\!\!\!\! (s-W(t))^2 \d x \leq \frac{\der}{\der t}(s- W(t)) \implies  \frac{\der}{\der t}(s- W(t)) \geq \frac{1}{2C_P^2|\T^2|} |\Sigma_s(t)| (s-W(t))^2. 
\end{equation*}
We integrate the final inequality in $t$ over the interval $[0,r]$ and get 
\begin{equation*}
    |\{(t,x) \in Q_r^+ : \underline{w}(t,x) > s \}| = \int_0^r |\Sigma_s(t)| \d t \leq  \frac{2C_P^2|\T^2|}{s- W(0)}  - \frac{2C_P^2|\T^2|}{s-W(r)} \leq \frac{2C_P^2|\T^2|}{s}, 
\end{equation*}
where we used $W(0)=0$, $s-W(r)>0$. By the above and \eqref{eq:relationship between Sigma and sigma set}, we get 
\begin{equation}\label{eq:suplevel est i}
   |\{(t,x) \in Q_r^+: w_\varepsilon(t,x) \!-\! \bar{w}_\varepsilon(0) > s \}| \leq |\{(t,x) : \underline{w}(\tau,x) > \frac{s}{2} \}| \leq \frac{4C_P^2|\T^2|}{s} \quad \forall s \geq C_*, 
\end{equation}
while, for $0 < s < C_*$, we estimate bluntly by the measure of the whole space 
\begin{equation}\label{eq:suplevel est ii}
   |\{(t,x) \in Q_r^+:w_\varepsilon(t,x) - \bar{w}_\varepsilon(0) > s \}| \leq |Q_1^+| \frac{s}{s} < \frac{|Q_1| C_*}{s} \quad \forall s \in (0,C_*). 
\end{equation}
By putting together the bounds of \eqref{eq:suplevel est i}--\eqref{eq:suplevel est ii}, we obtain the super-level set estimate 
\begin{equation*}
   |\{(t,x)\in Q_r^+ : w_\varepsilon(t,x) - \bar{w}_\varepsilon(0) > s \}| \leq \frac{C}{s} \qquad \forall s>0, 
\end{equation*}
where the positive constant $C=\max\{|Q_1|C_*,4C_P^2|\T^2|\}$ is independent of $\varepsilon,r$. Rewriting in terms of $a_\varepsilon$ as per \eqref{lem:log level sets} yields the first part of \eqref{eq:sublevel log est}. 

\smallskip 

\noindent 3. \textit{Sub-level set estimates}. As per the final step of \cite[Lemma 3.11]{BonforteDolbeaultNazaretSimonov}, rewriting only Step 2 of the proof with $-t$ instead of $t$ yields the second part of \eqref{eq:sublevel log est}. We write out several details for the benefit of the reader. In this latter case, it holds $W(t) \geq 0$ for all $t \in (-1,0)$ with $W$ as per \eqref{eq:big W def in superlevel est}, and the inclusions of \eqref{eq:relationship between Sigma and sigma set} are again satisfied, but with re-defining the sets in \eqref{eq:define sets Sigma and sigma} as $\Sigma_s(t) = \{x: \underline{w}(t,x) < -s\}$ and $\sigma_s(t) = \{x : w_\varepsilon(t,x) - \bar{w}_\varepsilon(0) < -s \}$, and $\underline{w}$ still defined as per \eqref{eq:underline subsol log est i}. Then, we notice that for all $x \in \Sigma_s(t)$, it holds 
\begin{equation*}
    W(t) - \underline{w}(t,x) > W(t) + s \geq s > 0, 
\end{equation*}
and thus, arguing as before, we obtain from the differential inequality of \eqref{eq:first log est mod W} that 
\begin{equation*}
    \frac{\der }{\der t} W(t) + \frac{1}{2C_P^2|\T^2|} \int_{\Sigma_s(t)} (W(t)+s)^2 \d x \leq 0 \implies \frac{1}{2C_P^2|\T^2|} |\Sigma_s(t)| \leq \frac{\der}{\der t} (W(t)+s)^{-1}, 
\end{equation*}
and thus $|\{(t,x) \in Q_r^- : \underline{w}(t,x) < -s \}| = \int^0_{-r} |\Sigma_s(t)| \d t \leq \frac{2C_P^2|\T^2|}{W(0)+s} - \frac{2C_P^2|\T^2|}{W(-r)+s} \leq \frac{2C_P^2|\T^2|}{s}$, and the required conclusion follows as per the end of Step 2 of the proof. 
\end{proof}

\section{Quantitative Instantaneous Separation from $\{\rho=1\}$}\label{sec:quantitative}

The main result of this section is the following; it is concerned with the non-degeneracy of $1-\rho$, which is at the heart of the regularity argument \S \ref{sec:smooth}. 

\begin{prop}[Quantitative instantaneous non-degeneracy]\label{prop:lower bound rho improved}
   Let $f$ be a weak solution of \eqref{eq:main eqn} with admissible initial data $f_0$, $\rho = \int_\T f \d \theta$. Then, there exists a positive constant $C>1$ $($independent of $t,T$, the initial data, and the solution$)$ such that 
   \begin{equation}\label{eq:infinite speed of prop}
       \essinf\limits_{(t,T)\times\T^2}(1\!-\!\rho) \geq  C^{-1} \Vert 1\!-\!\rho_0 \Vert_{L^1(\T^2)} e^{-C\big( T +  (1 \wedge t)^{-10} \big)},  \qquad \forall 0 < t < T < \infty.
   \end{equation}

\end{prop}

We highlight that the lower bound \eqref{eq:infinite speed of propagation in thm statement} in Theorem \ref{eq:infinite speed of propagation in thm statement} follows immediately from \eqref{eq:infinite speed of prop}, after enlarging the constant $C$ appropriately since $1 \geq e^{-Ct^{-10}}$; changing the factor $C^{-1}$ in \eqref{eq:infinite speed of prop} to $C^{-1}e^{-C}$ suffices. Hence, once Proposition \ref{prop:lower bound rho improved} is established, so is \eqref{eq:infinite speed of propagation in thm statement}.

\subsection{Harnack estimate for small times}\label{sec:quantitative harnack}

The following is a precursor to Proposition \ref{prop:lower bound rho improved} concerned only with small times. Note the cylinders in \eqref{eq:infinite speed of prop small times} are separated by a  \emph{gap}. 

\begin{lemma}[Harnack estimate for small times]\label{prop:lower bound rho improved small times only}
   Let $u$ be a weak solution of \eqref{eq:u eq} on $(0,\infty)\times\T^2$ satisfying \eqref{eq:initial upgrade dt rho}. Then, there exists $C>0$, independent of $t$ and $u$, such that 
   \begin{equation}\label{eq:infinite speed of prop small times}
       \essinf\limits_{(\frac{5}{4}t,\frac{3}{2}t)\times\T^2} u \geq e^{-Ct^{-10}}  \esssup\limits_{(\frac{1}{4}t,\frac{3}{4}t)\times\T^2} u  \qquad \forall t \in (0,1). 
   \end{equation}
\end{lemma}

\begin{proof} 
    The proof is divided into several steps. We begin by fixing $t_0 \in (0,1)$, setting $a_\varepsilon =  \int_{\T^2}   \log(u(t_0,x) +\varepsilon) \d x $ from \eqref{eq:aeps def} and defining the non-negative measurable functions 
    \begin{equation*}
        v^-_\varepsilon(t,x) := e^{a_\varepsilon} \big(u(t,x)+\varepsilon\big)^{-1} , \qquad v^+_\varepsilon(t,x) := e^{-a_\varepsilon} \big(u(t,x)+\varepsilon\big).
    \end{equation*}

    \smallskip 

    \noindent 1. \textit{Upper bound on the supremum}. By translating the relevant cylinders and multiplying by $e^{-a_\varepsilon}$ (which is justified since $e^{-a_\varepsilon} u(\frac{t_0}{2}+t,x)$ is also a solution of \eqref{eq:u eq} when $u$ is), the conclusion of Lemma \ref{lem:moser u upper}
    holds in the form:  for all $\varepsilon , p\in (0,1)$, and $0 < r < R \leq \frac{t_0}{2}$, there exists a positive constant $C$ independent of $p,r,R,\varepsilon,t_0$ such that 
    \begin{equation*}
   \begin{aligned}
     \Vert v^+_\varepsilon \Vert_{L^{\infty}(Q_r(\frac{t_0}{2}))} \leq \left(\frac{C }{(R-r)^5} \right)^{\frac{1}{p}}\Vert v^+_\varepsilon \Vert_{L^{p}(Q_R(\frac{t_0}{2}))}. 
   \end{aligned} 
\end{equation*}
It therefore follows that $v_\varepsilon^+$ satisfies the first assumption of the Bombieri--Giusti Lemma \ref{lem:bombieri giusti}. Furthermore, for all $s>0$ and all $r<t_0/2$, using the notation $w_\varepsilon(t,x) = -\log(u(t,x)+\varepsilon)$, 
\begin{equation*}
   \begin{aligned}
       \big| \{ (t,x) \in Q_r(\frac{t_0}{2}) : \log v^+_\varepsilon (t,x) > s \} \big|        &=  \big| \{ (t,x) \in Q_{r}(\frac{t_0}{2}) : w_\varepsilon(t,x) \!+\! a_\varepsilon < -s \} \big| \\ 
       &\leq  \big| \{ (t,x) \in Q^-_{t_0}(t_0) : w_\varepsilon(t,x) \!+\! a_\varepsilon < -s \} \big|    \leq \frac{C}{s} , 
   \end{aligned} 
\end{equation*}
by applying Lemma \ref{lem:log level sets} with the choice $r=t_0$ on the translated cylinder at time $t=t_0$ (note that $a_\varepsilon$ corresponds precisely to $\int_{\T^2}\log(u(t_0,x)+\varepsilon) \d x$ so that Lemma \ref{lem:log level sets} applies with the cylinder centered at $t_0$), 
where the positive constant $C$ is independent of $\varepsilon,s,r,t_0$, and where we used the identification of the cylinders $Q_r(\frac{t_0}{2}) = (\frac{t_0}{2}-r,\frac{t_0}{2}+r)\times\T^2 \subset (0,t_0) \times \mathbb{T}^2 = Q^-_{t_0}(t_0)$ for all $r < t_0/2$. The second condition of the Bombieri--Giusti Lemma \ref{lem:bombieri giusti} is therefore also verified for $v_\varepsilon^+$ for the family of nested sets $\{\mathcal{Q}_r = Q_r(\frac{t_0}{2})\}_{r<t_0/2}$, whence an application of this result (taking $\beta =5$ and $\tilde r = t_0/2$) implies $\Vert v^+_\varepsilon \Vert_{L^\infty(Q_r(t_0))} \leq e^{\bar{C}_1 (\frac{t_0}{2}-r)^{-10}}$ for all $0<r<\frac{t_0}{2}$, where the positive constant $\bar{C}_1$ is independent of $\varepsilon,r,t_0$. Thus, 
\begin{equation}\label{eq:first half of harnack}
   \begin{aligned}
      \varepsilon + \esssup_{Q_r(\frac{t_0}{2})} u \leq e^{a_\varepsilon} e^{\bar{C}_1 (\frac{t_0}{2}-r)^{-10}} \qquad \forall 0 < r < \frac{t_0}{2}. 
   \end{aligned}
\end{equation}

\smallskip 

\noindent 2. \textit{Lower bound on the infimum}. Similarly, by multiplying the conclusion of Lemma \ref{lem:moser 1/u} with $e^{a_\varepsilon}$ and centering the cylinders at $t= 3t_0/2$ by translation, we get 
    \begin{equation*}
   \begin{aligned}
    \Vert v^-_\varepsilon \Vert_{L^{\infty}(Q_r^-(\frac{3}{2}t_0))} \leq \left(\frac{C }{(R-r)^2} \right)^{\frac{1}{p}}\Vert v^-_\varepsilon \Vert_{L^{p}(Q_R^-(\frac{3}{2}t_0))} \qquad \forall 0 < r < R \leq \frac{3}{2}t_0, 
   \end{aligned} 
\end{equation*}
for all $p \in (0,1)$, with $C$ again independent of $\varepsilon,r,R,p,t_0$. Meanwhile, another application of Lemma \ref{lem:log level sets} (this time with the choice $r=t_0/2$ and having translated the cylinder to $t=t_0$) implies for all $s>0$ and all $r < t_0/2$, with the notation $w_\varepsilon(t,x) = -\log(u(t,x)+\varepsilon)$, 
\begin{equation*}
   \begin{aligned}
      \big| \{ (t,x) \in Q_r^-(\frac{3}{2}t_0) : \log v^-_\varepsilon (t,x) > s \} \big|             &=  \big| \{ (t,x) \in Q_r^-(\frac{3}{2}t_0) : w_\varepsilon(t,x) \!+\! a_\varepsilon > s \} \big| \\ 
      &\leq  \big| \{ (t,x) \in Q_{\frac{t_0}{2}}^+( t_0) : w_\varepsilon(t,x) \!+\! a_\varepsilon > s \} \big|   \leq \frac{C}{s} , 
   \end{aligned} 
\end{equation*}
where the positive constant $C$ is independent of $\varepsilon,s,r,t_0$, where we used the identification of the cylinders $Q_r^-(\frac{3}{2}t_0) = (\frac{3}{2}t_0-r,\frac{3}{2}t_0)\times\T^2 \subset (t_0,\frac{3}{2}t_0) = Q_{t_0/2}^+(t_0)$ for all $ r < t_0/2$. The second condition of the Bombieri--Giusti Lemma \ref{lem:bombieri giusti} is thus also fulfilled by $v_\varepsilon^-$ for the family of nested sets $\{\mathcal{Q}_r = Q_r^-(\frac{3}{2}t_0)\}_{r<t_0/2}$. Applying this result with the choices $\beta =2$ and $\tilde r = t_0/2$ implies $\Vert v^-_\varepsilon \Vert_{L^\infty(Q_r^-(\frac{3}{2}t_0))} \leq e^{\bar{C}_2 (\frac{t_0}{2}-r)^{-4}}$ for all $0<r<\frac{t_0}{2}$, where the positive constant $\bar{C}_2$ is independent of $\varepsilon,r,t_0$, from which we deduce 
\begin{equation}\label{eq:second half of harnack}
   \begin{aligned}
      \varepsilon + \essinf_{Q_r^-(\frac{3}{2}t_0)} u \geq e^{a_\varepsilon} e^{-\bar{C}_2 (\frac{t_0}{2}-r)^{-4}} \qquad \forall 0 < r < \frac{t_0}{2}. 
   \end{aligned}
\end{equation}

\smallskip 

\noindent 3. \textit{Harnack inequality}. Choosing $r=t_0/4$ in \eqref{eq:first half of harnack} and \eqref{eq:second half of harnack}, we get (with $C_1' = 4^{10} \bar{C}_1$, $C_2' = 4^4 \bar{C}_2$), the bound $( \varepsilon + \essinf_{Q_{t_0/4}^-(\frac{3}{2}t_0)} u ) e^{C_2' t_0^{-4}} \geq e^{a_\varepsilon} \geq ( \varepsilon + \esssup_{Q_{t_0/4}(\frac{t_0}{2})} u ) e^{-C_1' t_0^{-10}}$. Setting $C=C_1'+C_2'$, which is independent of $\varepsilon,t_0$, we deduce by letting $\varepsilon \to 0^+$ 
\begin{equation*}
    \essinf_{Q_{\frac{t_0}{4}}^-(\frac{3}{2}t_0)} u  \geq e^{- C t_0^{-10}} \esssup_{Q_{\frac{t_0}{4}}(\frac{t_0}{2})} u , 
\end{equation*}
where we also used $0<t_0<1$. Rewriting the cylinders and setting $t=t_0$ yields \eqref{eq:infinite speed of prop small times}. 
\end{proof}

We readily deduce the following corollary, where the previous lower bound (in terms of the supremum) in the Harnack inequality is replaced by the $L^1$ norm of the initial data. 

\begin{cor}\label{cor:lower bound depending on L1 norm}
    Let $f$ be a weak solution of \eqref{eq:main eqn} with admissible initial data $f_0$, $\rho = \int_\T f \d \theta$, and $u=1-\rho$. Then, there exists a positive constant $C$ $($independent of $t$, the initial data, and the solution$)$, such that 
   \begin{equation}\label{eq:infinite speed of prop small times in terms of L1}
       \essinf\limits_{(\frac{5}{4}t,\frac{3}{2}t)\times\T^2} u \geq  C^{-1} \Vert 1-\rho_0 \Vert_{L^1(\T^2)}  e^{-Ct^{-10}} \qquad \forall t \in (0,1). 
       \end{equation}
\end{cor}

\begin{proof}
    Recall that \eqref{eq:u eq} conserves its initial mass, \textit{i.e.~}directly from \eqref{eq:for rho initial data} with $\varphi = 1$, we have $ \Vert 1- \rho \Vert_{L^1(\T^2)} = \int_{\T^2} u \d x = \int_{\T^2} u_0 \d x =  \Vert 1- \rho_0 \Vert_{L^1(\T^2)}$ a.e.~$t>0$. It therefore follows, using also Lemma \ref{prop:lower bound rho improved small times only}, that for a.e.~$t \in (0,1)$ it holds $\Vert 1- \rho_0 \Vert_{L^1(\T^2)}  =  \int_{\T^2} u(\frac{t}{2},x) \d x \leq |\T|^2 \esssup\limits_{(\frac{1}{4}t,\frac{3}{4}t)\times\T^2} u \leq |\T|^2 e^{Ct^{-10}} \essinf\limits_{(\frac{5}{4}t,\frac{3}{2}t)\times\T^2} u$, and the result follows. 
\end{proof}

\subsection{Propagation of non-degeneracy}\label{sec:quantitative extend times}

In this portion of the manuscript, we prove Proposition \ref{prop:lower bound rho improved}. We record firstly the following propagation (in time) of lower-boundedness for $1-\rho$, which closely follows \cite[Lemmas 5.3--5.4]{reg2}.

\begin{lemma}[Propagation of boundedness away from $\{\rho=1\}$]\label{lem:propogation big times}
        Let $f$ be a weak solution of \eqref{eq:main eqn} with admissible initial data $f_0$, $\rho = \int_\T f \d \theta$, and $u=1-\rho$. Fix any Lebesgue point $t_0 \in (0,1)$ of $t \mapsto \essinf_{x\in \T^2} u(t,x)$. Then, there exists a strictly positive constant $C$ $($independent of $t_0,t$, the initial data, and the solution$)$ such that 
        \begin{equation*}
            \essinf_{(t_0,t)\times\T^2} u \geq C^{-1} e^{-Ct} \essinf_{x\in\T^2}u(t_0,x) \qquad \forall t > t_0. 
        \end{equation*}
\end{lemma}

This result relies on an iteration procedure which is different from the (Moser-style) ones presented in \S \ref{sec:moser}. In \cite[Lemma 5.4]{reg2}, the final estimate was closed using a technical lemma borrowed from \cite{KZ}. In order to make the paper self-contained, and because we require a more accurate quantitative estimate than that of \cite[Lemma 3.2]{KZ}, we record the following iterative lemma (the proof is postponed to Appendix \ref{sec:appendix}). 

\begin{lemma}[Quantitative iteration lemma]\label{lem:tech iteration lemma}
Let $\{A_k\}_{k\geq 0}$ be a sequence of non-negative differentiable functions satisfying the differential recurrence inequality on $[0,t]$$:$ 
      \begin{equation}\label{eq:for tech iter lemma}
        \frac{\der A_k}{\der t} + A_k \leq C_*^k (A_{k-1})^2 \qquad \forall k \geq 1, 
    \end{equation}
    where $C_* \geq 2$ is independent of $k$. Let $n_k = 2^k$, and suppose there exists $M_1,M_2>0$ such that $\sup_{k \geq 0} A_k(0)^{1/n_k} \leq M_1$ and $\Vert A_0 \Vert_{L^\infty([0,t])} \leq M_2$. Then, with $\beta = 1 + \sum_{j=1}^\infty (j+1)2^{-j}$, 
    \begin{equation}\label{eq:KimZhangPrecise}
       \sup_{k \geq 0} \Vert A_k \Vert_{L^\infty([0,t])}^{1/n_k} \leq C_*^\beta (M_1+M_2). 
    \end{equation}
\end{lemma}

We highlight that the constant $C_*$ in the previous lemma has no relation to the constant bearing the same name in the proof of Lemma \ref{lem:log level sets}. We are ready to prove Lemma \ref{lem:propogation big times}. 

\begin{proof}[Proof of Lemma \ref{lem:propogation big times}]
    Fix $t_0 \in (0,1)$ as in the statement. By Corollary \ref{cor:lower bound depending on L1 norm}, it holds  $\essinf_{x\in\T^2} u(t_0,x) > 0$ provided $\Vert 1 -\rho_0 \Vert_{L^1(\T^2)}>0$, which we assume. In other words, defining $v:=1/u$ and $v_0 = v(t_0,\cdot)$, 
    \begin{equation}\label{eq:lower bound in proof of propagation}
      \Vert v_0 \Vert_{L^\infty(\T^2)} = \big( \essinf_{x\in\T^2} u(t_0,x) \big)^{-1} < \infty. 
    \end{equation}
    We will now rewrite the equation \eqref{eq:u eq} as an evolution on $v$, and then show by an iteration procedure that $v$ is bounded for all subsequent times by a quantity depending on $\Vert v_0 \Vert_{L^\infty}$. 

    \smallskip 

    \noindent 1. \textit{Equation for $v$}. By formal manipulations, we expect the equation for $v$ to be $\partial_t v + \dv(\p v) - 2 \p \! \cdot \! \nabla v = \Delta v - 2\frac{|\nabla v|^2}{v}$, \textit{i.e.}, we expect $v$ to be a weak \emph{subsolution} of 
    \begin{equation}\label{eq:v eqn for propagation}
        \partial_t v + \dv(\p v) - 2 \p \! \cdot \! \nabla v = \Delta v.
    \end{equation}
    Note however that \emph{a priori} it is not even clear that $v \in L^2_\loc(0,\infty;H^1(\T^2))$ for $t \geq t_0$; the purpose of this step of the proof is to verify this rigorously.

    Equation \eqref{eq:u eq} is only satisfied in the weak sense dual to $L^2_\loc(0,\infty;H^1(\T^2))$. To this end, we cannot test \eqref{eq:u eq} directly with functions of $v$ itself, but instead we consider the regularised version $v_\varepsilon := 1/(u+\varepsilon) \in L^2_\loc(0,\infty;H^1(\T^2)) \cap L^\infty((0,\infty)\!\times\!\T^2)$ by Definition \ref{def:weak sol} (\textit{cf.~}also \eqref{eq:initial upgrade dt rho}). We test \eqref{eq:u eq} against the admissible test function $-v_\varepsilon^3$ (since, formally, $\frac{1}{2} \partial_t (v_\varepsilon^2) = v_\varepsilon \partial_t v_\varepsilon = -(u+\varepsilon)^3 \partial_t u = - v_\varepsilon^3 \partial_t u$), and get 
    \begin{equation*}
    \langle \partial_t u ,-v_\varepsilon^3 \rangle -  3\int_{\T^2} v_\varepsilon^2 \nabla v_\varepsilon \cdot  \p u \d x =  3\int_{\T^2} v_\varepsilon^2 \nabla v_\varepsilon \cdot \nabla u \d x, 
\end{equation*}
using \eqref{eq:initial upgrade dt rho u} and the chain/product rules for Sobolev functions belonging to $H^1(\T^2)$ for a.e.~$t $, \textit{cf.}~\textit{e.g.}~\cite[\S 5.10]{Evans}. Then, writing $u = v_\varepsilon^{-1} - \varepsilon$ and again using the chain rule for Sobolev functions to write $\nabla u = -v_\varepsilon^{-2}\nabla v_\varepsilon$ (here we also use that $v_\varepsilon^{-2} = (u+\varepsilon)^2 \leq (1+\varepsilon)^2 \in L^\infty$), 
\begin{equation*}
    \langle \partial_t u ,-v_\varepsilon^3 \rangle -  3\int_{\T^2} v_\varepsilon \nabla v_\varepsilon \cdot  \p  \d x +  3 \varepsilon \int_{\T^2} v_\varepsilon^2 \nabla v_\varepsilon \cdot  \p  \d x = -3\int_{\T^2} |\nabla v_\varepsilon|^2 \d x. 
\end{equation*}
Using that $(1+\varepsilon)^{-1} \leq v_\varepsilon \leq \varepsilon^{-1} \in L^\infty$, $v_\varepsilon \in L^2_\loc(0,\infty;H^1(\T^2))$, and $\partial_t u \in L^{\nicefrac{3}{2}}_\loc([t_0,\infty)\times\T^2)$ from \eqref{eq:initial upgrade dt rho u}, the chain rule for Sobolev functions and the Lions--Magenes Lemma \cite[Ch.3, \S 1.4
Lem.1.2]{Temam} yield $\frac{1}{2}\frac{\der}{\der t} \Vert v_\varepsilon \Vert_{L^2(\T^2)}^2 =  \langle \partial_t v_\varepsilon , v_\varepsilon \rangle = -  \langle v_\varepsilon^2 \partial_t u , v_\varepsilon \rangle =  \langle \partial_t u , - v_\varepsilon^3 \rangle $, whence (using the boundedness of $\p$ \eqref{eq:polarisation bound}, $\varepsilon v_\varepsilon^2 \leq v_\varepsilon$, and H\"older's inequality) 
\begin{equation*}
    \frac{1}{2}\frac{\der}{\der t} \Vert v_\varepsilon \Vert_{L^2(\T^2)}^2 + 3 \Vert \nabla v_\varepsilon \Vert^2_{L^2(\T^2)} \leq 6 \Vert v_\varepsilon \Vert_{L^2(\T^2)} \Vert \nabla v_\varepsilon \Vert_{L^2(\T^2)} . 
\end{equation*}
Applying Young's inequality and then Gr\"onwall's Lemma (integrating over $[t_0,t]$) implies 
\begin{equation}\label{eq:gronwall quantitative large times}
    \esssup_{s \in [t_0,t]}\Vert v_\varepsilon(s,\cdot) \Vert^2_{L^2(\T^2)} + \Vert \nabla v_\varepsilon \Vert^2_{L^2((t_0,t)\times\T^2)} \leq  C \Vert v_\varepsilon(t_0,\cdot) \Vert_{L^2(\Omega)}^2 e^{ Ct} \leq C\Vert v_0 \Vert^2_{L^\infty} e^{Ct}, 
\end{equation}
for a.e.~$t>t_0$, where we used \eqref{eq:lower bound in proof of propagation} and $v_\varepsilon \leq v$, and the constant $C$ changed in the final inequality (but is independent of $t_0,t$ and the initial data). Letting $\varepsilon \to 0^+$ and using the lower semicontinuity of the norm, we get $v \in L^2_\loc([t_0,\infty);H^1(\T^2)) \cap L^\infty_\loc([t_0,\infty);L^2(\T^2))$.

It remains to write down the equation for $v$. By the aforementioned applications of the chain rule, it holds $\partial_t v_\varepsilon = -v_\varepsilon^2 \partial_t u$ with $\partial_t u$ integrable \eqref{eq:initial upgrade dt rho} and $\nabla v_\varepsilon = -v_\varepsilon^2 \nabla u$. It follows (using the previous computations) that, in the weak sense dual to $L^2_\loc([t_0,\infty);H^1(\T^2))$, 
\begin{equation}\label{eq:eqn for v eps}
    \begin{aligned}\partial_t v_\varepsilon + \dv\big( \p v_\varepsilon(1-\varepsilon v_\varepsilon) \big) - 2 \p \!\cdot \! \nabla v_\varepsilon (1-\varepsilon v_\varepsilon) = \Delta v_\varepsilon - 2\frac{|\nabla v_\varepsilon|^2}{v_\varepsilon}. 
    \end{aligned}
\end{equation}
Letting $\varepsilon \to 0^+$, using the Dominated Convergence Theorem with $v_\varepsilon = \frac{v}{1+\varepsilon v} \to v$ a.e.~and the regularity $v \in L^2_\loc([t_0,\infty);H^1(\T^2))$, we get that $v$ is a weak subsolution of \eqref{eq:v eqn for propagation}.

    \smallskip 

    \noindent 2. \textit{Boundedness of $v$}. We argue very closely to \cite[Lemma 5.4]{reg2} using the equation \eqref{eq:eqn for v eps}. We test the equation \eqref{eq:eqn for v eps} with $v_\varepsilon^{n-1}$ for $n \geq 2$ and obtain 
  \begin{equation*}
        \begin{aligned} 
        \frac{\der}{\der t}\int_{\T^2} |v_\varepsilon^{\frac{n}{2}}|^2 \d x + 4(1+\frac{1}{n}) \int_{\T^2} |\nabla (v_\varepsilon^{\frac{n}{2}}) |^2 \d x = & 2(n+1) \int_{\T^2} (1-\varepsilon v_\varepsilon) v_\varepsilon^{\frac{n}{2}} \nabla(v_\varepsilon^{\frac{n}{2}}) \! \cdot \! \p \d x , 
       \end{aligned}
    \end{equation*}
whence applying Young's inequality (noting also that $0 < v_\varepsilon \leq \varepsilon^{-1}$ so that $|1-\varepsilon v_\varepsilon| \leq 1$), 
\begin{equation}\label{eq:before GN}
        \begin{aligned} 
        \frac{\der}{\der t} \Vert v_\varepsilon^{\frac{n}{2}} \Vert^2_{L^2(\T^2)} + 3\Vert \nabla (v_\varepsilon^{\frac{n}{2}}) \Vert^2_{L^2(\T^2)} \leq 2n^2 \Vert   v_\varepsilon^{\frac{n}{2}} \Vert_{L^2(\T^2)}^2. 
       \end{aligned}
    \end{equation}
Next, an application of the Gagliardo--Nirenberg inequality implies 
\begin{equation*}
    \begin{aligned}
        \Vert v_\varepsilon^{\frac{n}{2}} \Vert_{L^2(\T^2)} \leq C_{GN} \Big( \Vert \nabla (v_\varepsilon^{\frac{n}{2}}) \Vert_{L^2(\T^2)}^{\nicefrac{1}{2}} \Vert v_\varepsilon^{\frac{n}{2}} \Vert_{L^1(\T^2)}^{\nicefrac{1}{2}} + \Vert v_\varepsilon^{\frac{n}{2}} \Vert_{L^1(\T^2)}  \Big); 
    \end{aligned}
\end{equation*}
note that, by squaring both sides and using Young's inequality, the above also gives  $\Vert \nabla v^{\frac{n}{2}} \Vert_{L^2(\T^2)}^{2} \geq C_1 \Vert v^{\frac{n}{2}} \Vert_{L^2(\T^2)}^2 - C_2\Vert v^{\frac{n}{2}} \Vert_{L^1(\T^2)}^2  )$ for some positive constants $C_1,C_2$ depending only on $C_{GN}$. In turn, substituting into the right-hand side of \eqref{eq:before GN} and using Young's inequality again yields $ \frac{\der}{\der t} \Vert v_\varepsilon^{\frac{n}{2}} \Vert^2_{L^2(\T^2)} +  2 \Vert \nabla (v_\varepsilon^{\frac{n}{2}}) \Vert^2_{L^2(\T^2)} \leq ( 4 n^4 C_{GN}^4 + 4n^2 C_{GN}^2 ) \Vert v_\varepsilon^{\frac{n}{2}} \Vert_{L^1(\T^2)}^2  $, \textit{i.e.}, using also $\Vert v_\varepsilon^{\frac{n}{2}} \Vert^2_{L^2(\T^2)} = \Vert v_\varepsilon^{n} \Vert_{L^1(\T^2)}$, 
\begin{equation*}
        \begin{aligned} 
        \frac{\der}{\der t} \Vert v_\varepsilon^{n} \Vert_{L^1(\T^2)} +   \Vert v_\varepsilon^{n} \Vert_{L^1(\T^2)} \leq Cn^4 \Vert v_\varepsilon^{\frac{n}{2}} \Vert_{L^1(\T^2)}^2  , 
       \end{aligned}
    \end{equation*}
where the positive constant $C$ depends only on $C_{GN}$, is independent of $\varepsilon,n,t_0$, and can be assumed greater than 1. By setting $n_k := 2^k$ and $A_k(t) := \Vert v^{n_k}_\varepsilon(t+t_0,\cdot) \Vert_{L^1(\T^2)}, $ we obtain, with $C_* := 16 C$ which can be assumed to verify $C_* \geq 2$, the recurrence inequality 
    \begin{equation*}
        \frac{\der A_k}{\der t} + A_k \leq C_*^k (A_{k-1})^2, 
    \end{equation*}
  which is precisely \eqref{eq:for tech iter lemma}. 
Using the estimate \eqref{eq:gronwall quantitative large times} and 
$$
\begin{aligned} &\sup_k A_k(0)^{1/n_k} = \sup_k \Vert v_\varepsilon(t_0) \Vert_{L^{n_k}(\T^2)} = \Vert v_\varepsilon(t_0) \Vert_{L^\infty(\T^2)} \leq \Vert v_0 \Vert_{L^\infty(\T^2)} =: M_1 , \\ 
&\Vert A_0 \Vert_{L^\infty([t_0,t])} \!=\! \esssup_{s \in [t_0,t]} \Vert v_\varepsilon(s) \Vert_{L^1(\T^2)} \!\leq \! |\T^2|^{\frac{1}{2}}  \esssup_{s \in [t_0,t]}\Vert v_\varepsilon(s) \Vert_{L^2(\T^2)} \!\leq \! |\T^2|^{\frac{1}{2}} \Vert v_0 \Vert_{L^\infty} e^{Ct} \! =: \! M_2, 
\end{aligned}$$ 
where we also used H\"older's inequality in the second line (and where the constant $C$ has changed but remains independent of $\varepsilon,t_0,t,$ and the initial data), Lemma \ref{lem:tech iteration lemma} yields 
\begin{equation*}
   \begin{aligned} \Vert v_\varepsilon \Vert_{L^\infty([t_0,t]\times\T^2)} \leq \sup_k \Vert A_k \Vert^{1/n_k}_{L^\infty([0,t])} \leq C(M_1+M_2) \leq C\Vert v_0 \Vert_{L^\infty} e^{Ct}, 
   \end{aligned}
\end{equation*}
for some positive constant $C$ independent of $\varepsilon,t_0,t$, the initial data, and the solution. As a result, by letting $\varepsilon \to 0^+$ and using that $v_\varepsilon = \frac{v}{1+\varepsilon v}$ so that $v_\varepsilon \to v$ a.e., we get 
\begin{equation*}
    \essinf_{(t_0,t)} u \geq C^{-1} e^{-Ct} \Vert v_0 \Vert_{L^\infty}^{-1} \qquad \forall t > t_0, 
\end{equation*}
which, by rewriting in terms of $u(t_0,\cdot)$ as per \eqref{eq:lower bound in proof of propagation}, yields the required result.     
\end{proof}

Proposition \ref{prop:lower bound rho improved} now follows immediately from Corollary \ref{cor:lower bound depending on L1 norm} and Lemma \ref{lem:propogation big times}, up to modifying the constant $C$; we omit the details for clarity of presentation.

\section{Instantaneous Smoothing}\label{sec:smooth}

Having obtained the strictly positive lower bound \eqref{eq:infinite speed of propagation in thm statement} on $1-\rho$ (in Proposition \ref{prop:lower bound rho improved}), the proof of the instantaneous smoothing effect encapsulated in Theorem \ref{thm:reg} now follows by the same strategy as \S 6-7 of \cite{reg2}. For completeness, and to make the present work self-contained, we write the arguments in detail. We also incorporate the corrections to \cite{reg2} provided in \cite{reg2Err}, and provide a refined treatment with clearer proofs. Finally, we remark that the higher regularity bootstrap (\S \ref{sec:bootstrap}) is written differently to the corresponding section in \cite{reg2} (\S 7 therein); indeed, here we appeal to Schauder theory in H\"older continuous spaces, while in \cite[\S 7]{reg2} this bootstrap was written in terms of Sobolev spaces. 

Throughout this section, $f$ will always denote a weak solution of \eqref{eq:main eqn} with admissible initial data, $\rho = \int_\T f \d \theta$, $\p = \int_{\T} f \e(\theta) \d \theta$. Furthermore, given $0<t<T<\infty$, the letter $c_{t,T}$ will always be used to denote positive lower bound of \eqref{eq:infinite speed of propagation in thm statement}, given explicitly by 
\begin{equation}\label{eq:ctT def}
0 < c_{t,T} := C^{-1} \Vert 1\!-\!\rho_0 \Vert_{L^1(\T^2)} e^{-C( T + t^{-10} )}  , 
\end{equation}
where we recall that $C$ above is independent of $t,T$, the initial data, and the solution. 

\subsection{Local boundedness of $f$}

The goal of this section is to prove the following result. 

\begin{prop}[$f$ is locally bounded]\label{prop:f bounded moser}
    For all $0\!<\!t\!<\!T\!<\!\infty$, it holds $f \in L^\infty((t,T)\!\times\!\T^3)$. 
\end{prop}
We divide the analysis leading up to the proof of Proposition \ref{prop:f bounded moser} into several subsections.

\subsubsection{Initial upgrade of regularity for $\p$}

Our first result is concerned with the $H^1$-type estimate for $\p$. We remind the reader that \eqref{eq:p eq} holds in duality with $L^3_\loc([0,\infty);W^{1,3}(\T^2))$; this can be seen easily from the boundedness $(1-\rho)\nabla \p \in L^{\nicefrac{3}{2}}_\loc([0,\infty);L^{\nicefrac{3}{2}}(\T^2))$ from \eqref{eq:the horrible one}, which yields $\partial_t \p \in L^{\nicefrac{3}{2}}_\loc([0,\infty);W^{-1,\nicefrac{3}{2}}(\T^2))$ from reading \eqref{eq:p eq} in the sense of distributions.

\begin{lemma}[$H^1$ estimate for $\p$, $L^4$ bound on $\nabla \rho$]\label{lem:H1 p}
 For all $0<t<T<\infty$, it holds $$\nabla \p \in L^2((t,T)\!\times\!\T^2), \qquad \partial_t \p \in L^2((t,T);H^{-1}(\T^2)), \qquad \nabla \rho \in L^4((t,T)\times\T^2).$$ 
\end{lemma}

We note in passing that a version of this result was obtained in \cite[Lemmas 4.2-4.3]{reg2}, and that a minor correction should be made therein: one should assume the lower bound $\essinf_{(t,T)\times\T^2}(1-\rho) \geq c_{t,T}>0$ in the statement.

\begin{proof}
    We cannot test \eqref{eq:p eq} with $\p$, since  it is not known to be an admissible test function. Instead, we consider smoothed approximations of $\p$ and argue similarly to \cite[Lemma 4.2]{reg2}.

    \smallskip 

    \noindent 1. \textit{Approximation}. Let $\eta$ be the usual Friedrichs bump function, and $\{\eta^\varepsilon\}_{\varepsilon>0}$ the usual sequence of Friedrichs mollifiers in $\mathbb{R}^2$; we only smooth in the $x$ variable. Let $\p^\varepsilon = \eta^\varepsilon*\p$, and define $D = (1-\rho)\P$ and its approximation $D^\varepsilon = \eta^\varepsilon*D$. We insert $\eta^\varepsilon(s,x-y)$ as test function in \eqref{eq:p eq} and integrate in $y$ over $\mathbb{R}^2$---this is the same as mollifying \eqref{eq:p eq}---and get 
    \begin{equation}\label{eq:mollified p eqn}
        \partial_t \p^\varepsilon + \dv D^\varepsilon = \dv((1-\rho)\nabla\p^\varepsilon + \p^\varepsilon \otimes \nabla \rho) - \p^\varepsilon + \dv (E^\varepsilon + F^\varepsilon), 
    \end{equation}
    where $E^\varepsilon = \eta^\varepsilon*((1-\rho)\nabla \p) - (1-\rho)\eta^\varepsilon*\nabla \p$ and $F^\varepsilon = \eta^\varepsilon*(\p \otimes \nabla \rho) - (\eta^\varepsilon*\p )\otimes \nabla \rho. $

    \smallskip 

    \noindent 2. \textit{Error control}. We show that for some $C_T$ independent of $\varepsilon,t$ (but which depends on $T$), 
    \begin{equation}\label{eq:unif bounds for H1 for p with mollifications}
        \Vert D^\varepsilon \Vert_{L^\infty([0,T]\times\T^2)} + \Vert E^\varepsilon \Vert_{L^2([0,T]\times\T^2)} + \Vert F^\varepsilon \Vert_{L^2([0,T]\times\T^2)} \leq C_T. 
    \end{equation}
   To begin with, for a.e.~$s\in (t,T)$ it holds $\Vert D^\varepsilon (s) \Vert_{L^\infty(\T^2)} \leq  \Vert D(s) \Vert_{L^\infty(\T^2)} \leq  \Vert \P \Vert_{L^\infty} \leq C,$ for some positive constant $C$ independent of $\varepsilon,t,T$, using standard properties of the mollifier in the first inequality. The first part of \eqref{eq:unif bounds for H1 for p with mollifications} follows. Next, we estimate $F^\varepsilon$: directly from basic properties of the mollifier sequence, we have for a.e.~$s \in (t,T)$
\begin{equation*}
    \Vert \eta^\varepsilon*(\p \otimes \nabla \rho)(s) \Vert_{L^2(\T^2)} \leq C \Vert (\p \otimes \nabla \rho) (s) \Vert_{L^2(\T^2)} \leq C \Vert \p \Vert_{L^\infty } \Vert \nabla \rho(s) \Vert_{L^2(\Omega)}, 
\end{equation*}
for some positive constant $C$ independent of $\varepsilon,t,T$, and analogously 
\begin{equation*}
    \Vert (\eta^\varepsilon*\p )\otimes \nabla \rho (s) \Vert_{L^2(\T^2)} \!\leq\! \Vert (\eta^\varepsilon*\p )(s) \Vert_{L^\infty(\T^2)}\Vert \nabla \rho(s) \Vert_{L^2(\T^2)} \!\leq\! \Vert \p \Vert_{L^\infty} \Vert \nabla \rho(s) \Vert_{L^2(\T^2)}, 
\end{equation*}
thus, by integrating in time, there is a positive $C_T=C\Vert \nabla \rho\Vert_{L^2([0,T]\times\T^2)}$ independent of $\varepsilon,t$ such that $\Vert F^\varepsilon \Vert_{L^2([0,T]\times\T^2)} \leq C_T$, whence the final term in \eqref{eq:unif bounds for H1 for p with mollifications} is bounded.

The final error term $E^\varepsilon$ is written $E^\varepsilon(s,x) = \int_{\mathbb{R}^2} \eta^\varepsilon(x-y)\big[ \rho(s,x)-\rho(s,y) \big]\nabla\p(t,y) \d y$, and---by integrating by parts---admits the two-term decomposition 
\begin{equation*}
   \begin{aligned}
       E^\varepsilon(s,x) =\!\!  \underbrace{\int_{\mathbb{R}^2} \!\!\p(s,y) \otimes \nabla \eta^\varepsilon(x\!-\!y) \big[ \rho(s,x)\!-\!\rho(s,y) \big] \d y}_{=: E^\varepsilon_1(s,x)} + \!\underbrace{\int_{\mathbb{R}^2} \!\!\eta^\varepsilon(x\!-\!y) \p(s,y) \otimes \nabla \rho(s,y) \d y}_{=: E^\varepsilon_2(s,x)} . 
   \end{aligned} 
\end{equation*}
Using standard properties of mollifiers, $\Vert E^\varepsilon_2(s) \Vert_{L^2(\T^2)} \leq C \Vert \nabla \rho(s) \Vert_{L^2(\T^2)} \Vert \p \Vert_{L^\infty},$ thus 
\begin{equation}\label{eq:Eeps2 bound unif}
    \Vert E^\varepsilon_2 \Vert_{L^2([0,T]\times\T^2)} \leq C \Vert \nabla \rho \Vert_{L^2([0,T]\times\T^2)}  , 
\end{equation}
for some  $C$ independent of $\varepsilon,t,T$. Meanwhile, for a.e.~$x\in \RR^2$ and $y \in \supp \eta^\varepsilon(x-\cdot)$,
\begin{equation*}
   \begin{aligned} 
 | \rho(s,x)-\rho(s,y) |= \bigg|\int_0^1 \frac{\der}{\der \sigma}\rho(s,\sigma x \!+\!(1\!-\!\sigma)y)  \d \sigma\bigg| &= \bigg|\int_0^1 \nabla \rho(s,\sigma x \!+\!(1\!-\!\sigma)y) \!\cdot\! (x-y)  \d \sigma \bigg| \\ 
 &\leq \varepsilon \int_0^1 |\nabla \rho(s,\sigma x +(1-\sigma)y)| \d \sigma,  
   \end{aligned}
\end{equation*}
using $\supp \eta^\varepsilon = B(0,\varepsilon)$ to bound $|x-y|\leq \varepsilon$. The boundedness of $\p$ then implies 
\begin{equation*}
    \begin{aligned}
       | E^\varepsilon_1(s,x) | &\leq \varepsilon^{-2} \int_{B(x,\varepsilon)} \big|\nabla \eta(\frac{x-y}{\varepsilon}) \big| \bigg(  \int_0^1 |\nabla \rho(s,\sigma x +(1-\sigma)y)| \d \sigma\bigg) \d y \\ 
        &= \int_{B(0,1)} |\nabla \eta (w)| \bigg(  \int_0^1 |\nabla \rho \big(s, x - (1-\sigma)\varepsilon w \big)| \d \sigma\bigg) \d w, 
    \end{aligned}
\end{equation*}
and, by Jensen's inequality, we get $|E^\varepsilon_1(s,x)|^2 \leq C \int_{B(0,1)} \int_0^1 |\nabla \rho (s, x - (1-\sigma)\varepsilon w )|^2 \d \sigma \d w$, for some positive constant $C$ independent of $\varepsilon,t,T$. The Fubini--Tonelli Theorem then implies, using the periodicity of $\nabla \rho$ and translation invariance of Lebesgue measure, 
\begin{equation*}
    \Vert E^\varepsilon_1(s) \Vert^2_{L^2(\T^2)} \leq C \int_{B(0,1)} \int_0^1 \int_{\T^2} |\nabla \rho(s,x)|^2 \d x \d \sigma \d w \leq C \Vert \nabla \rho(s) \Vert^2_{L^2(\T^2)}. 
\end{equation*}
By integrating in time, and combining with \eqref{eq:Eeps2 bound unif}, we deduce $\Vert E^\varepsilon \Vert_{L^2([0,T]\times\T^2)} \leq C_T$. 

\smallskip 

    \noindent 3. \textit{$H^1$ bound on $\p$}. $\p^\varepsilon \in L^\infty(t,T;C^\infty(\T^2))$ is an admissible test function for \eqref{eq:mollified p eqn}. Furthermore, it is immediate from \eqref{eq:mollified p eqn} that $\partial_t \p^\varepsilon \in L^2((t,T);H^{-1}(\T^2))$ whence the Lions--Magenes Lemma applies and $\langle \partial_t \p^\varepsilon,\p^\varepsilon\rangle = \frac{1}{2}\frac{\der}{\der t} \Vert \p^\varepsilon\Vert^2_{L^2}$. By \eqref{eq:unif bounds for H1 for p with mollifications} and Young's inequality, 
    \begin{equation*}
        \begin{aligned}
             \Vert \p^\varepsilon(s)\Vert^2_{L^2(\T^2)} \!+\!  c_{t,T} \int_t^{s} \int_{\T^2} |\nabla \p^\varepsilon|^2 \d x \d s' \leq \Vert \p^\varepsilon(t)\Vert^2_{L^2(\T^2)} \!+\! \frac{C_T^2}{c_{t,T}} \!+\!  2 c_{t,T}^{-1} \Vert \nabla \rho \Vert^2_{L^2([0,s]\times\T^2)}, 
        \end{aligned}
    \end{equation*}
for all $t<s\leq T$, using also $\Vert \p^\varepsilon(t)\Vert_{L^\infty(\T^2)} \leq C \Vert \p(t)\Vert_{L^\infty(\T^2)} \leq C$. We deduce 
    \begin{equation}\label{eq:just before H1 final bound for p}
        \begin{aligned}
            \Vert \p^\varepsilon\Vert_{L^\infty((t,T);L^2(\T^2))}^2 +  \Vert \nabla \p^\varepsilon\Vert_{L^2((t,T)\times\T^2)}^2 &\leq C_{t,T}', 
        \end{aligned}
    \end{equation}
    for some $C_{t,T}'$ independent of $\varepsilon$. We let $\varepsilon \to 0$, and the required boundedness of $\nabla \p$ follows from the lower semicontinuity of the norms. Finally, $\partial_t \p \in L^2((t,T);H^{-1}(\T^2))$ from \eqref{eq:p eq}. 

    \smallskip 

    \noindent 4. \textit{$L^4$ bound on $\nabla\rho$}. 
It now follows from equation \eqref{eq:rho eq} that $\partial_t \rho - \Delta \rho = - \dv((1-\rho)\p) \in L^2((t',T')\!\times\!\T^2)$ for all $0\!<\!t'\!<\!T'$, thus the Calder\'on--Zygmund Theorem yields 
\begin{equation}\label{eq:dt rho and laplace rho in L2loc}
    \partial_t \rho , \nabla^2 \rho \in L^2((t,T) \! \times \! \T^2) \qquad \forall 0<t<T<\infty. 
\end{equation}
Thus, we may rigorously perform the standard $H^2$ estimate on \eqref{eq:rho eq} by testing against $\Delta \rho$, which yields $\frac{1}{2}\frac{\der}{\der t}\Vert \nabla \rho\Vert_{L^2(\T^2)}^2  + \frac{1}{2}\Vert \nabla^2 \rho\Vert_{L^2(\T^2)}^2 \leq \frac{1}{2} \Vert \dv((1-\rho)\p) \Vert_{L^2(\T^2)}^2$, \textit{i.e.}, 
    \begin{equation}\label{eq:grad rho in Linft L2}
        \nabla \rho \in L^\infty((t,T);L^2(\T^2)) \qquad \forall 0 < t < T<\infty. 
    \end{equation}
    It follows from \eqref{eq:grad rho in Linft L2} and \eqref{eq:dt rho and laplace rho in L2loc} that $\nabla \rho$ satisfies the assumptions of the interpolation Lemma \ref{lem:dibenedetto classic},  which yields $\nabla \rho \in L^4((t,T)\!\times\!\T^2)$ for all $0<t<T<\infty$. 
\end{proof}

Note furthermore that, by integrating \eqref{eq:main eqn} against $\e(\theta)\!\otimes\! \e(\theta)$ in $\theta$, we also get the governing equation for $\P$ the drift coefficient in \eqref{eq:p eq}: 
\begin{equation}\label{eq:P eqn with Pi}
    \partial_t \P + \dv((1-\rho) \Pi) = \dv((1-\rho)\nabla \P + \P \otimes \nabla \rho ) + \tilde{\Pi}, 
\end{equation}
where $\Pi = \int_\T f (\e(\theta)\!\otimes\! \e(\theta)\! \otimes\! \e(\theta)) \d \theta$ is a tensor which again satisfies 
\begin{equation*}
    |\Pi(t,x)| \leq 2^3 \int_{\T} |f(t,x,\theta)| \d \theta = 2^3 \rho(t,x) \leq 2^3 \implies \Pi \in L^\infty((0,\infty) \! \times \! \T^2), 
\end{equation*}
and the same holds for $\tilde{\Pi} = \int_\T f \partial^2_\theta (\e(\theta) \otimes \e(\theta)) \d \theta$. In turn, by applying the same strategy as in Lemma \ref{lem:H1 p} to the equation \eqref{eq:P eqn with Pi}, we have proved the following. 

\begin{lemma}[$H^1$ estimate for $\P$]\label{lem:H1 for P}
    For all $0<t<T<\infty$, it holds $\nabla \P \in L^2((t,T)\!\times\!\T^2)$. 
\end{lemma}

Having established Lemma \ref{lem:H1 p}, it follows that (for positive times) the equation \eqref{eq:p eq} for $\p$ satisfies the assumptions of the Gehring-type lemma (Lemma \ref{lem:Gehring}) when rewritten as 
\begin{equation*}
    \partial_t \p - \dv((1-\rho)\nabla \p) = - \p + \dv(\p \otimes \nabla \rho - (1-\rho)\P) ; 
\end{equation*}
indeed, we set $A=(1-\rho)\mathrm{Id} \in L^\infty$, $c=c_{t,T}^{-1}$ provided by \eqref{eq:ctT def} for the ellipticity constant, $g=-\p \in L^\infty$, and $G= \p \otimes \nabla \rho - (1-\rho)\P \in L^4((t,T)\!\times\!\T^2)$ so that $p=4$ in Lemma \ref{lem:Gehring}. We emphasise that this result only requires the tensor $A$ to have \emph{measurable} coefficients---no continuity is needed. The \emph{Gehring-type} Lemma \ref{lem:Gehring} therefore yields: for all $0<t<T$ 
\begin{equation}\label{eq:grad p upgrade by gehring}
    \nabla \p \in L^{2+\delta}((t,T)\!\times\!\T^2) , 
\end{equation}
for some $\delta>0$ which depends only on $t,T,c_{t,T}$. In turn, for all $0<t<T<\infty$ there exists some $\delta>0$ depending only on $t,T,c_{t,T}$ such that $\dv((1-\rho)\p) \in L^{2+\delta}((t,T)\!\times\!\T^2)$---note that we have already shown in Lemma \ref{lem:H1 p} that $\nabla \rho \in L^4_{t,\loc}L^4_x$. Consequently, Lemma \ref{lem:Lp schauder} (or, analogously, the Calder\'on--Zygmund Theorem) and \eqref{eq:rho eq} yield: for all $0<t<T$ 
\begin{equation}\label{eq:dt hess rho in L2+delta}
    \partial_t \rho, \nabla^2 \rho \in L^{2+\delta}((t,T)\!\times\!\T^2) , 
\end{equation}
for some $\delta>0$ depending only on $t,T,c_{t,T}$; note that this $\delta$ is \emph{a priori} different to the one in \eqref{eq:grad p upgrade by gehring} for the same choice of $t,T$, since the result of Lemma \ref{lem:Lp schauder} is \emph{local}. Noting that $\T^2$ is of dimension $2<2+\delta$ (no matter the value of $\delta>0$ depending on the precise parabolic cylinder), Morrey's embedding yields: for all $0<t<T$ we have 
\begin{equation*}
    \Vert \nabla \rho(s) \Vert_{L^\infty(\T^2)} \leq C_S  \Vert \nabla \rho(s) \Vert_{W^{1,2+\delta}(\T^2)} \qquad \forall s \in (t,T), 
\end{equation*}
where the Sobolev constant $C_S$ depends only on $\T^2$ and $\delta$ (which depends only on $t,T,c_{t,T}$); thus, $C_S$ depends only on $t,T,c_{t,T}$. Raising to the power $2+\delta$ and integrating in $s$ yields 
\begin{equation*}
  \Vert  \nabla \rho \Vert_{L^{2+\delta}((t,T);L^\infty(\T^2))} \leq C_S  \Vert \nabla \rho \Vert_{L^{2+\delta}((t,T);W^{1,2+\delta}(\T^2))} . 
\end{equation*}
Thus, by \eqref{eq:dt hess rho in L2+delta}: for all $0<t<T$, there exists $\delta>0$ depending only on $t,T,c_{t,T}$ such that 
\begin{equation}\label{eq:grad rho in L2+delta Linfty}
    \nabla \rho \in L^{2+\delta}((t,T);L^\infty(\T^2)).  
\end{equation}
Putting together \eqref{eq:grad p upgrade by gehring}--\eqref{eq:grad rho in L2+delta Linfty}, we have proved the following corollary of Lemma \ref{lem:H1 p}. 

\begin{cor}\label{cor:consequence of gehring}
    For all $0\!<\!t\!<\!T\!<\!\infty$, $\exists \delta>0$ depending only on $t,T,c_{t,T}$, such that 
    \begin{equation*}
        \nabla \p \in L^{2+\delta}((t,T)\!\times\!\T^2), \quad \partial_t \rho, \nabla^2 \rho \in L^{2+\delta}((t,T)\!\times\!\T^2), \quad  \nabla \rho \in L^{2+\delta}((t,T);L^\infty(\T^2)). 
    \end{equation*}
\end{cor}

\subsubsection{$H^2$-type estimate for $\p$}\label{sec:H2 for p}

The next step in our analysis involves a $H^2$-type estimate on $\p$. The underlying idea is that, \emph{provided one can justify changing from the divergence form to non-divergence form in \eqref{eq:p eq}}, a crucial cancellation of mixed derivative terms occurs: 
\begin{equation*}
    \partial_t \p + \dv((1-\rho)\P) = (1-\rho) \Delta \p + \p \Delta \rho - \p . 
\end{equation*}
By testing the above with $\Delta \p$, which \emph{a priori} is not justified since we \emph{do not} know $\Delta \p \in L^2$, we formally get from Young's inequality 
\begin{equation}\label{eq:formal computation before galerkin}
   \frac{\der}{\der t}\Vert \nabla \p \Vert^2_{L^2} \!+\! \frac{c_{t,T}}{2}\Vert \Delta \p \Vert_{L^2}^2 \!+\! \Vert \nabla \p \Vert^2_{L^2} \leq c_{t,T}^{-1}\Big(\Vert \dv((1\!-\!\rho)\P) \Vert^2_{L^2} \! + \! \Vert \Delta \rho \Vert^2_{L^2} \Big), 
\end{equation}
where we used $\nabla \P \in L^2((t,T)\!\times\!\T^2)$ (Lemma \ref{lem:H1 for P}) so that the penultimate term is bounded. Integrating in time gives $\nabla \p \in L^\infty((t,T);L^2(\T^2)) \cap L^2((t,T);H^1(\T^2))$, as desired. 

In order to justify the aforementioned passage from divergence form to non-divergence form, we approximate $\p$ using a Galerkin scheme. That is, we will construct $x$-smooth approximations $\p^n$ (which are linear combinations of eigenfunctions of $\Delta_x$ with periodic boundary conditions in $\T^2$) which converge to $\p$ as $n\to\infty$. The previous formal computation is then justified for each $\p^n$, whence we will obtain an estimate akin to \eqref{eq:formal computation before galerkin}---since this estimate will turn out to be independent of $n$, it also holds for $\p$ itself. We stress that this Galerkin scheme is used only as a \emph{tool for regularity}---not existence (already known).

At this stage of the analysis, we \emph{fix an arbitrary pair of times} $0<t<T<\infty$ such that $t$ is a Lebesgue point of $t\mapsto (\p(t,\cdot),f(t,\cdot))$, and we denote by $\delta$ the \emph{specific} $\delta>0$ depending on $t,T,c_{t,T}$ such that Corollary \ref{cor:consequence of gehring} holds, and denote by $q_\delta$ the H\"older exponent:
\begin{equation}\label{def: q delta}
    \frac{1}{2} = \frac{1}{q_\delta} + \frac{1}{2+\delta}; 
\end{equation}
where $q_\delta \in (1,\infty)$ depends only on $t,T,c_{t,T}$. In the lines that follow we introduce the required notions to write the aforementioned Galerkin scheme. Define the \emph{ansatz}
\begin{equation}\label{eq:ansatz p n}
    \p^n(s,x) := \sum_{j=1}^n \boldsymbol{\alpha}^n_{j}(s) \varphi_{j}(x) \qquad \forall (s,x) \in [t,T]\times\T^2, 
\end{equation}
where $\{\varphi_{j}\}_{j}$ is the smooth orthonormal basis of $L^2(\T^2)$ provided by the eigenfunctions of the Laplacian on $\T^2$ (with periodic boundary conditions), \textit{i.e.}, with $0 < \lambda_1 < \lambda_2 < \dots$, 
\begin{equation}\label{eq:relation for eig}
    -\Delta \varphi_j = \lambda_j^2 \varphi_j \qquad j \in \mathbb{N}, 
\end{equation}
and the coefficient vectors $\{\boldsymbol{\alpha}^n_{j}(s)\}_{j=1}^{n} \subset \mathbb{R}^2$ are to be determined, subject to initial data
\begin{equation}\label{eq:initial data galerkin}
    \boldsymbol{\alpha}^n_{j}(t) = \langle \p(t), \varphi_{j} \rangle \qquad \forall j=1,\dots,n, \quad \forall n\in\mathbb{N}, 
\end{equation}
at the \emph{initial time $s=t$}, where $\langle \cdot , \cdot \rangle$ denotes the $L^2(\T^2)$ inner product. By construction, 
\begin{equation}\label{eq:conv of initial data for galerkin ii}
    \lim_{n\to\infty} \Vert \p(t) - \p^n(t) \Vert_{L^2(\T^2)} = 0. 
\end{equation}
Furthermore, since $\p(t) \in H^1(\T^2)$ by Lemma \ref{lem:H1 p}, integration by parts and \eqref{eq:relation for eig} imply 
\begin{equation*}
   \begin{aligned}
       \langle \p(t) ,\! \varphi_j \rangle \! =\! -\lambda_j^{-2}\!\!\int_{\T^2}\!\! \p(t,x) \Delta \varphi_j(x) \! \d x \!=\! \lambda_j^{-2}\!\!\int_{\T^2}\!\! \nabla \p(t,x) \! \cdot \! \nabla \varphi_j(x) \! \d x \!=\! \lambda_j^{-1}\!\langle \p (t) ,\! \lambda_j^{-1}\!\varphi_j \rangle_{H^1}, 
   \end{aligned} 
\end{equation*}
where $\langle \cdot, \cdot \rangle_{H^1}$ denotes the inner product on $H^1(\T^2)$. The standard spectral theory of self-adjoint operators in Hilbert spaces implies $\{ \lambda_j^{-1}\varphi_j \}_j$ is an orthonormal basis of $H^1(\T^2)$: $\p^n(t) = \sum_{j=1}^n \langle \p(t) , \lambda_j^{-1} \varphi_j \rangle_{H^1} \lambda_j^{-1}\varphi_j \to \p(t)$ strongly in $H^1(\T^2)$, \textit{i.e.}, with \emph{$t$ the initial time}
\begin{equation}\label{eq:conv of initial data for galerkin}
    \lim_{n\to\infty} \Vert \p(t) - \p^n(t) \Vert_{H^1(\T^2)} = 0. 
\end{equation}
Moreover, $\p^n(s) \in C^\infty(\T^2) $ for all $n,s$, being a linear combination of eigenfunctions of $\Delta_x$.

The next result shows that there exists $\p^n$ of the form \eqref{eq:ansatz p n} satisfying the required Galerkin approximation: for simplicity, since existence of $f,\rho,\p$ is already known, we construct $\p^n$ to be the solution of a \emph{linear} equation. 
\begin{prop}[Galerkin scheme $\p$]\label{prop:galerkin p}
    For all $n\in\mathbb{N}$, $\exists \p^n \in AC([t,T];C^\infty(\T^2))$ satisfying 
    \begin{equation}\label{eq:galerkin approx 1}
    \left\lbrace \begin{aligned} 
        &\partial_t \p^n - \dv((1-\rho)\nabla \p^n + \p^n \otimes \nabla \rho) + \p^n = - \dv((1-\rho)\P), \\ 
        &\p^n(t) = \sum_{j=1}^n \langle \p(t),\varphi_j \rangle \varphi_j, 
    \end{aligned}\right. 
\end{equation}
in the weak sense dual to $\mathrm{span}\{\varphi_j\}_{j=1}^n$ for all $s \in [t,T]$, with the initial data achieved in the strong $L^2$-sense. Furthermore, $\p^n \rightharpoonup \p $ weakly in $L^2((t,T);H^2(\T^2))$, 
and 
\begin{equation}\label{eq:n indep bounds galerkin p}
\sup_n \Vert \p^n \Vert_{L^\infty((t,T);L^2(\T^2))} + \sup_n \Vert \nabla \p^n \Vert_{L^2((t,T)\!\times\!\T^2)} + \sup_n \Vert \nabla^2 \p^n \Vert_{L^2((t,T)\!\times\!\T^2)} < \infty. 
\end{equation}
\end{prop}
\begin{proof}
We divide the proof into several steps. 

\smallskip 

 \noindent 1. \textit{Existence of $\p^n$}. Exploiting the orthonormality of the eigenfunctions and testing \eqref{eq:galerkin approx 1} with $\varphi_j$, we see that \eqref{eq:galerkin approx 1} is satisfied by the ansatz \eqref{eq:ansatz p n} in duality with $\mathrm{span}\{\varphi_j\}_{j=1}^n$ if 
   \begin{equation*}
    \begin{aligned}
        \frac{\der \boldsymbol{\alpha}^n_{j}}{\der t}  \!+\! \sum_{k=1}^n  \boldsymbol{\alpha}^n_{k} \bigg( \int_{\T^2} \Big( (1\!-\!\rho) \nabla \varphi_k + \varphi_k \nabla \rho \Big) \!\cdot\! \nabla\varphi_{j} \d x \bigg) \!+\! \boldsymbol{\alpha}^n_j = \int_{\T^2}\!\!   (1-\rho)\P   \nabla \varphi_j  \d x. 
    \end{aligned}
\end{equation*}
The above reads as a linear system of ODEs, of the form $\boldsymbol{\alpha}'(s) + A(s)\boldsymbol{\alpha}(s) = B(s)$, where the tensors $A,B$ are sufficiently regular that the standard well-posedness theory applies---no assumptions on the invertibility/definiteness of $A$ are needed. Specifically, $A(s),B(s)$ are \emph{locally integrable} in the time variable $s$, thus an application of Carath\'eodory's Theorem implies that there exists $\p^n \in AC([t,T];C^\infty(\T^2))$ solving \eqref{eq:galerkin approx 1} in duality with $\mathrm{span}\{\varphi_j\}_{j=1}^n$ (of the form \eqref{eq:ansatz p n}), with the initial data achieved in $L^2(\T^2)$ at $s=t$.    

\smallskip

\noindent 2. \textit{Estimates independent of $n$}. By testing \eqref{eq:galerkin approx 1} with $\p^n$ (this is allowed since $\p^n(s) \in \mathrm{span}\{\varphi_j\}_{j=1}^n$), we get the usual $H^1$-type estimate $\frac{1}{2}\frac{\der}{\der t}\Vert \p^n \Vert^2_{L^2(\T^2)} + \frac{c_{t,T}}{2}\Vert \nabla \p^n \Vert^2_{L^2(\T^2)} + \Vert \p^n \Vert^2_{L^2(\T^2)} \leq \frac{1}{c_{t,T}}(\Vert (1-\rho)\P \Vert^2_{L^2(\T^2)} + \Vert \p^n |\nabla \rho| \Vert^2_{L^2(\T^2)})$, where we used the lower bound on $1-\rho$ provided by $c_{t,T}$ and the Lions--Magenes Lemma (directly from the equation $\partial_t \p^n \in L^2((t,T);\mathrm{span}\{\varphi_j\}_{j=1}^n)$ since $\mathrm{span}\{\varphi_j\}_{j=1}^n$ is a finite-dimensional Hilbert space). Integrating in time yields (for all $t < s \leq T$) 
\begin{equation*}
   \begin{aligned} \Vert \p^n(s) \Vert^2_{L^2(\T^2)} + c_{t,T} \Vert \nabla \p^n \Vert^2_{L^2((t,s)\times\T^2)}  \leq & \,  \Vert \p^n(t) \Vert^2_{L^2(\T^2)} + c_{t,T}^{-1}\Vert (1-\rho)\P \Vert^2_{L^2((t,T)\times\T^2)} \\ 
   &+ c_{t,T}^{-1} \int_t^s \Vert \nabla \rho(s') \Vert_{L^\infty(\T^2)}^2 \Vert \p^n(s') \Vert^2_{L^2(\T^2)} \d s', 
   \end{aligned}
\end{equation*}
where we dropped the positive term $\Vert \p^n \Vert^2_{L^2((t,T)\times\T^2)}$ on the left-hand side. Due to the strong convergence of initial data \eqref{eq:conv of initial data for galerkin ii}, we see that the  first two terms on the right-hand side are bounded independently of $n$ by some $C_{t,T}$ depending only on $t,T,c_{t,T}$. We get 
\begin{equation*}
    \Vert \p^n(s) \Vert^2_{L^2(\T^2)} + \Vert \nabla \p^n \Vert^2_{L^2((t,s)\times\T^2)}  \leq  C_{t,T} \bigg( 1+ \int_t^s \Vert \nabla \rho(s') \Vert_{L^\infty(\T^2)}^2 \Vert \p^n(s') \Vert^2_{L^2(\T^2)} \d s' \bigg), 
\end{equation*}
and an application Gr\"onwall's Lemma implies (for a different constant $C_{t,T}$) 
\begin{equation*}
    \Vert \p^n \Vert^2_{L^\infty((t,T);L^2(\T^2))} + c_{t,T} \Vert \nabla \p^n \Vert^2_{L^2((t,T)\times\T^2)}  \leq C_{t,T} \exp \bigg( C_{t,T} \int_t^T \Vert \nabla \rho(s') \Vert_{L^\infty(\T^2)}^2  \d s' \bigg), 
\end{equation*}
where the integral is bounded since $\nabla \rho \in L^{2+\delta}((t,T);L^\infty(\T^2))$ for some $\delta>0$ depending on $t,T,c_{t,T}$ by Corollary \ref{cor:consequence of gehring}, and the boundedness of the first two terms in \eqref{eq:n indep bounds galerkin p} follows. 

Next, we perform the $H^2$-type estimate by testing \eqref{eq:galerkin approx 1} with $\Delta \p^n$ (this is allowed since $\Delta \p^n \in \mathrm{span}\{\varphi_j\}_{j=1}^n$ due to \eqref{eq:relation for eig}). Since $\p^n(s) \in C^\infty(\T^2)$, we rigorously pass to the non-divergence form of \eqref{eq:galerkin approx 1}. Similarly to the computation \eqref{eq:formal computation before galerkin}, we get 
\begin{equation*}
\begin{aligned}
\frac{\der}{\der t}\Vert \nabla \p \Vert^2_{L^2(\T^2)} \!+\! \frac{c_{t,T}}{2}\Vert \Delta & \p \Vert_{L^2(\T^2)}^2 \leq c_{t,T}^{-1}\Big( \Vert \dv((1\!-\!\rho)\P) \Vert^2_{L^2(\T^2)} \!+\! 2\int_{\T^2} |\p^n|^2 |\Delta \rho|^2 \d x \Big) \\ 
&\leq c_{t,T}^{-1}\Big( \Vert \dv((1\!-\!\rho)\P) \Vert^2_{L^2(\T^2)} \!+\! 2  \Vert \p^n(s) \Vert^2_{L^{q_\delta}(\T^2)} \Vert \Delta \rho(s) \Vert_{L^{2+\delta}(\T^2)}^2\Big), 
\end{aligned}
\end{equation*}
using the Lions--Magenes Lemma in the finite-dimensional Hilbert space $\{\varphi_j\}_{j=1}^n$, where the final term is present because \emph{a priori} we do not know $\p^n \in L^\infty$, unlike $\p$; we recall that $q_\delta$ defined in \eqref{def: q delta} depends on $t,T,c_{t,T}$. Next, the Sobolev inequality implies $\Vert \p^n(s) \Vert_{L^{q_\delta}(\T^2)} \leq C_S \Vert \p^n(s) \Vert_{H^1(\T^2)}$ for some $C_S$ independent of $n$ but which depends on $t,T,c_{t,T}$ (since $q_{\delta}$ does). Integrating in $s$ and using $\nabla \P \in L^2_\loc$ from Lemma \ref{lem:H1 for P}, we get 
\begin{equation*}
   \begin{aligned}
       \Vert \nabla\p^n(s) \Vert^2_{L^2(\T^2)} \!+\! c_{t,T} \Vert \Delta \p^n \Vert^2_{L^2((t,s)\times\T^2)}  \leq \Vert \nabla\p^n(t) \Vert^2_{L^2(\T^2)} \!+\! c_{t,T}^{-1}\Vert \dv((1\!-\!\rho)\P) \Vert^2_{L^2} & \\ 
       + 4 C_S^2 c_{t,T}^{-1} \Vert \p^n \Vert^2_{L^\infty((t,T);L^2(\T^2))} \Vert \Delta \rho \Vert^2_{L^2((t,T);L^{2+\delta}(\T^2))} & \\ 
       + 4 C_S^2 c_{t,T}^{-1} \int_t^s \Vert \Delta \rho(s') \Vert^2_{L^{2+\delta}(\T^2)} \Vert \nabla \p^n (s') \Vert^2_{L^2(\T^2)} \d s' &  
   \end{aligned} 
\end{equation*}
for all $t < s \leq T$. The strong convergence to the initial data \eqref{eq:conv of initial data for galerkin}, the uniform bound on $\Vert \p^n \Vert_{L^\infty((t,T);L^2(\T^2))}$, and Corollary \ref{cor:consequence of gehring} implies that there exists $C_{t,T}$ independent of $n$ such that the first three terms on the previous right-hand side are bounded by $C_{t,T}$. We deduce 
\begin{equation*}
    \Vert \nabla\p^n(s) \Vert^2_{L^2(\T^2)} \!+ c_{t,T} \Vert \Delta \p^n \Vert^2_{L^2((t,s)\times\T^2)} \leq C_{t,T} \Big(1 \!+\! \int_t^s \!\! \Vert \Delta \rho(s') \Vert^2_{L^{2+\delta}(\T^2)} \Vert \nabla \p^n (s') \Vert^2_{L^2(\T^2)} \d s' \Big), 
\end{equation*}
for some $C_{t,T}$ independent of $n$, whence an application of Gr\"onwall's Lemma implies 
\begin{equation*}
    \Vert \nabla\p^n \Vert^2_{L^\infty((t,T);L^2(\T^2))} \!+\! c_{t,T} \Vert \Delta \p^n \Vert^2_{L^2((t,T)\times\T^2)} \leq C_{t,T} \exp \Big( C_{t,T} \int_t^T \Vert \Delta \rho(s') \Vert^2_{L^{2+\delta}(\T^2)} \d s' \Big), 
\end{equation*}
for a different $C_{t,T}$; the final integral is bounded due to Corollary \ref{cor:consequence of gehring}. Thus \eqref{eq:n indep bounds galerkin p} follows.

\smallskip 

\noindent 3. \textit{Convergence to $\p$}. The uniform bounds \eqref{eq:n indep bounds galerkin p} imply that $\{\p^n\}_n$ is a uniformly bounded sequence in $L^2((t,T);H^2(\T^2))$. As such, by Alaoglu's Theorem, it admits a weakly convergent limit $\p^\infty \in L^2((t,T);H^2(\T^2))$ for a subsequence which we do not relabel. By passing to the limit in the \emph{linear} equation \eqref{eq:galerkin approx 1}, we see that $\p^\infty$ satisfies 
    \begin{equation}\label{eq:galerkin approx 1 limit}
    \left\lbrace \begin{aligned} 
        &\partial_t \p^\infty - \dv((1-\rho)\nabla \p^\infty + \p^\infty \otimes \nabla \rho) + \p^\infty = - \dv((1-\rho)\P), \\ 
        &\p^\infty(t) = \p(t), 
    \end{aligned}\right. 
\end{equation}
where we used that $\p = \sum_{j=1}^\infty \langle \p(t),\varphi_j \rangle \varphi_j$ in $L^2(\T^2)$ from \eqref{eq:conv of initial data for galerkin}. The equation \eqref{eq:galerkin approx 1 limit} holds in duality with $L^2((t,T);\mathrm{span}\{\varphi_j\}_{j=1}^\infty)$, and hence, by density, in duality with $L^2((t,T);H^1(\T^2))$; we also deduce directly from \eqref{eq:galerkin approx 1 limit} that $\partial_t \p^\infty \in L^2((t,T);H^{-1}(\T^2))$. Next, we define the difference $\mathbf{w} := \p^\infty - \p$. On account of $\p,\p^\infty \in L^2((t,T);H^1(\T^2))$, from Lemma \ref{lem:H1 p}, we have that $\mathbf{w}$ satisfies 
\begin{equation*}
    \partial_t \mathbf{w} - \dv((1-\rho)\nabla \mathbf{w} + \mathbf{w} \otimes \nabla \rho) + \mathbf{w} = 0, 
\end{equation*}
with zero initial condition, and $\mathbf{w}$ is itself an admissible test function for this equation with $\partial_t \mathbf{w}\in L^2((t,T);H^{-1}(\T^2))$. The Lions--Magenes Lemma applies, and testing with $\mathbf{w}$ gives 
\begin{equation*}
    \begin{aligned}
        \frac{1}{2}\frac{\der}{\der t}\Vert \mathbf{w} \Vert^2_{L^2(\T^2)} \!+\! \frac{c_{t,T}}{2}\Vert \nabla \mathbf{w} \Vert^2_{L^2(\T^2)} \!+\! \Vert \mathbf{w} \Vert^2_{L^2(\T^2)} \leq \frac{1}{2c_{t,T}} \Vert \mathbf{w} \!\otimes\! \nabla \rho \Vert^2_{L^2}\leq \frac{1}{2c_{t,T}} \Vert \nabla \rho \Vert_{L^\infty}^2 \Vert \mathbf{w} \Vert^2_{L^2}. 
    \end{aligned}
\end{equation*}
Writing $Y(s) := \Vert \mathbf{w}(s)\Vert^2_{L^2(\T^2)}$, this implies $Y'(s) \leq c_{t,T}^{-1} \Vert \nabla \rho(s) \Vert^2_{L^\infty(\T^2)} Y(s)$, and hence
\begin{equation*}
    Y(s) \leq {Y(t)}\exp\bigg( c_{t,T}^{-1} \int_t^T \Vert \nabla \rho(s') \Vert^2_{L^\infty(\T^2)} \d s' \bigg) = 0 \qquad \forall s \in [t,T], 
\end{equation*}
since $Y(t)=0$, where the integral is finite since $\nabla \rho \in L^{2+\delta}((t,T);L^\infty(\T^2))$ for some $\delta>0$ depending on $t,T,c_{t,T}$ (Corollary \ref{cor:consequence of gehring}). Hence $\p = \p^\infty \in L^2((t,T);H^2(\T^2))$. 
\end{proof}

Step 2 of the proof of Proposition \ref{prop:galerkin p} also gives $\nabla \p \in L^\infty((t,T);L^2(\T^2))$. Using the interpolation Lemma \ref{lem:dibenedetto classic}, and recalling from the lines above \eqref{def: q delta} that $t,T$ were fixed arbitrarily, we have proved the following. 

\begin{cor}\label{cor:nabla p in L4}
    For all $0<t<T<\infty$, it holds $\p \in L^2((t,T);H^2(\T^2))$ and $\nabla \p \in L^\infty((t,T);L^2(\T^2)) \cap L^4((t,T)\!\times\!\T^2)$. 
    \end{cor}

In turn, it follows from Lemma \ref{lem:H1 p} and Corollary \ref{cor:nabla p in L4} that $\dv((1-\rho)\p) \in L^4((t,T)\!\times\!\T^2)$, whence the Calder\'on--Zygmund Theorem and \eqref{eq:rho eq} imply $\partial_t \rho , \nabla^2 \rho \in L^4((t,T)\!\times\!\T^2)$. We deduce that $\rho \in W^{1,4}((t,T)\!\times\!\T^2)$, whence Morrey's Embedding implies $\rho \in C^{\nicefrac{1}{4}}([t,T]\!\times\T^2)$. Also, an application of the Gagliardo--Nirenberg inequality gives 
\begin{equation}\label{eq:GN grad rho L8}
    \Vert \nabla\rho(s) \Vert_{L^8(\T^2)} \leq C \Vert   \nabla^2\rho(s) \Vert_{L^4(\T^2)}^{\nicefrac{1}{2}} \Vert \rho(s) \Vert_{L^\infty(\T^2)}^{\nicefrac{1}{2}} + C \Vert \rho(s) \Vert_{L^1(\T^2)}, 
\end{equation}
for some constant $C$ depending only on $\T^2$. Integrating in time, we get $\Vert \nabla\rho \Vert_{L^8((t,T)\times\T^2)}^8 \leq C \Vert   \nabla^2\rho \Vert_{L^4((t,T)\times\T^2)}^{4} \Vert \rho \Vert_{L^\infty}^{4} + C(T-t). $ We have shown the following. 

\begin{cor}\label{cor:continuity rho and hess in L4 nabla rho L8}
    For all $0<t<T<\infty$, it holds 
    \begin{equation*}
        \partial_t \rho,\nabla^2 \rho \in L^4((t,T)\!\times\!\T^2), \quad \nabla \rho \in L^8((t,T)\!\times\!\T^2), \quad \rho \in C^{\nicefrac{1}{4}}([t,T]\!\times\!\T^2). 
    \end{equation*}
\end{cor}

\subsubsection{Boundedness of $f$ by Moser iterations}

The goal of this subsection is to prove Proposition \ref{prop:f bounded moser}, which we shall show by means of Moser iterations akin to those of \S \ref{sec:moser}. In order to use this strategy, we must first verify that the equation \eqref{eq:main eqn} can be written as a problem in the standard parabolic space $L^2_{t,\loc} H^1_\xi$ with $\partial_t f \in L^2_{t,\loc} H^{-1}_\xi$. 

\begin{lemma}[$f$ belongs to the parabolic space]\label{lem:f parabolic space}
For all $0<t<T<\infty$, it holds $f \in L^\infty((t,T);L^2(\T^3)) \cap L^2((t,T);H^1(\T^3)) \cap L^{\nicefrac{10}{3}}((t,T)\!\times\!\T^3)$ and $\partial_t f \in L^2((t,T);H^{-1}(\T^3))$. 
\end{lemma}
\begin{proof}
    Our objective is, as usual, to test the equation \eqref{eq:main eqn} against $f$, which gives the result formally. However, \emph{a priori} $f$ is not regular enough to be a test function. To this end, similarly to Lemma \ref{lem:H1 p}, we introduce the $\xi$-mollifier sequence $\{\eta^\varepsilon\}_\varepsilon$ in $\mathbb{R}^3$ constructed from the Friedrichs bump function. Define $D=(1-\rho) f \e(\theta)$ and $D^\varepsilon = \eta^\varepsilon*D$. Inserting $\eta^\varepsilon(s,\xi-y)$ into the weak formulation of \eqref{eq:main eqn} and integrating in $y$, we get 
    
    \begin{equation}\label{eq:mollified f eqn}
        \partial_t f^\varepsilon + \dv D^\varepsilon = \dv((1-\rho)\nabla f^\varepsilon + f^\varepsilon \nabla \rho) +\partial_\theta^2 f^\varepsilon + \dv (E^\varepsilon + F^\varepsilon), 
    \end{equation}
    where $E^\varepsilon = \eta^\varepsilon*((1-\rho)\nabla f) - (1-\rho)\eta^\varepsilon*\nabla f$ and $F^\varepsilon = \eta^\varepsilon*(f \nabla \rho) - (\eta^\varepsilon* f ) \nabla \rho. $

    \smallskip 

    \noindent 1. \textit{Control of error terms}. Using that $f \in L^3_{t,\xi}$ from Definition \ref{def:function space}, standard properties of the mollifier imply that $\Vert D^\varepsilon \Vert_{L^3([t,T]\times\T^3)} \leq C \Vert D \Vert_{L^3([t,T]\times\T^3)} \leq C\Vert f \Vert_{L^3([t,T]\times\T^3)}$, where the constant $C$ is independent of $t,T,\varepsilon$. Similarly, for the term $F^\varepsilon$, we have $\Vert \eta^\varepsilon*(f\nabla \rho)(s) \Vert_{L^2(\T^3)} \leq C \Vert f\nabla \rho(s) \Vert_{L^2(\T^3)} \leq C\Vert f (s)\Vert_{L^3(\T^2)} \Vert \nabla \rho(s) \Vert_{L^6(\T^2)}$, whence squaring and using Young's inequality with the exponent $\nicefrac{3}{2}$ yields $\Vert \eta^\varepsilon*(f\nabla \rho) \Vert_{L^2([t,T]\times\T^3)}^2 \leq C(\Vert f \Vert_{L^3([t,T]\times\T^2)}^3 + \Vert \nabla \rho \Vert^6_{L^6([t,T]\times\T^2)})$, where the constant $C$ is independent of $t,T,\varepsilon$; the final term is bounded in view of Corollary \ref{cor:continuity rho and hess in L4 nabla rho L8}. The second term in $F^\varepsilon$ is controlled analgously, whence we deduce $ \Vert D^\varepsilon \Vert_{L^2([t,T]\times\T^2)} + \Vert F^\varepsilon \Vert_{L^2([t,T]\times\T^2)} \leq C_{t,T}$ for some $C_{t,T}$ depending only on $t,T,c_{t,T}$. It remains to estimate $E^\varepsilon$, which, like in Lemma \ref{lem:H1 p}, we split into two parts. Standard computations (\textit{cf.~}Step 2 of the proof of Lemma \ref{lem:H1 p}) give 
    \begin{equation*}
   \begin{aligned}
       E^\varepsilon(s,\xi) =\!\!  \underbrace{\int_{\mathbb{R}^3} \!\!f(s,y) \nabla \eta^\varepsilon(\xi\!-\!y) \big[ \rho(s,\xi)\!-\!\rho(s,y) \big] \d y}_{=: E^\varepsilon_1(s,\xi)} + \!\underbrace{\int_{\mathbb{R}^3} \!\!\eta^\varepsilon(\xi\!-\!y) f(s,y) \nabla \rho(s,y) \d y}_{=: E^\varepsilon_2(s,\xi)} . 
   \end{aligned} 
\end{equation*}
We have $\Vert E_2^\varepsilon(s) \Vert_{L^2(\T^3)} \leq C \Vert f \nabla \rho(s) \Vert_{L^2(\T^2)}$ using standard properties of the mollifier (with $C$ independent of $t,T,\varepsilon$), which yields $\Vert E_2^\varepsilon \Vert_{L^2([t,T]\times\T^3)} \leq C_{t,T}$ using the same approach as for $F^\varepsilon$. For $E^\varepsilon_1$, an identical strategy to that in Step 2 of the proof of Lemma \ref{lem:H1 p} yields 
\begin{equation*}
    \begin{aligned}
       | E^\varepsilon_1(s,\xi) | &\leq \int_{B(0,1)} f(s,\xi-\varepsilon w) |\nabla \eta (w)| \bigg(  \int_0^1 |\nabla \rho \big(s, \xi - (1-\sigma)\varepsilon w \big)| \d \sigma\bigg) \d w. 
    \end{aligned}
\end{equation*}
By squaring, Jensen's inequality, integrating in $\xi$ and using Fubini's Theorem, we get 
\begin{equation*}
    \begin{aligned}
       \Vert E^\varepsilon_1(s) \Vert^2_{L^2(\T^3)} &\leq C \int_{B(0,1)} \int_0^1 \bigg( \int_{\T^3} |f(s,\xi-\varepsilon w)|^2  |\nabla \rho \big(s, \xi - (1-\sigma)\varepsilon w \big)|^2 \d \xi  \bigg) \d \sigma  \d w , 
    \end{aligned}
\end{equation*}
where the constant $C$ depends only on $\Vert \nabla \eta \Vert_{L^\infty}$. Using H\"older's inequality (with exponent $\nicefrac{3}{2}$) and the translation invariance of Lebesgue measure, we deduce 
\begin{equation*}
    \begin{aligned}
       \Vert E^\varepsilon_1(s) \Vert^2_{L^2(\T^3)} &\leq C \Vert f(s)^2 \Vert_{L^{\nicefrac{3}{2}}(\T^3)} \Vert |\nabla \rho(s) |^2\Vert_{L^3(\T^2)} = C\Vert f(s) \Vert_{L^3(\T^3)}^2 \Vert \nabla \rho(s) \Vert_{L^6(\T^2)}^2 . 
    \end{aligned}
\end{equation*}
Young's inequality yields $\Vert E^\varepsilon_1 \Vert_{L^2([t,T]\times\T^3)}^2 \leq C(\Vert f \Vert_{L^3}^3 + \Vert \nabla \rho \Vert^6_{L^6})$. We have shown 
\begin{equation}\label{eq:unif eps bound f in parabolic space}
\Vert D^\varepsilon \Vert_{L^2([t,T]\times\T^2)}^2 + \Vert E^\varepsilon \Vert_{L^2([t,T]\times\T^2)}^2 + \Vert F^\varepsilon \Vert_{L^2([t,T]\times\T^2)}^2 \leq C_{t,T}
\end{equation}
for some $C_{t,T}$ depending only on $t,T,c_{t,T}$, and independent of $\varepsilon$. 

    \smallskip 

    \noindent 2. \textit{$H^1$ bound}. Note that it is clear from the previous bounds and the equation \eqref{eq:mollified f eqn} that $\partial_t f^\varepsilon \in L^2((t,T);H^{-1}(\T^3))$, and hence the use of the Lions--Magenes Lemma is justified. Testing \eqref{eq:mollified f eqn} with $f^\varepsilon$, we get, using \eqref{eq:unif eps bound f in parabolic space}, the lower bound on $1\!-\!\rho$, and Young's inequality, 
    \begin{equation*}
        \frac{1}{2}\frac{\der}{\der t}\Vert f^\varepsilon \Vert^2_{L^2(\T^3)} + \frac{c_{t,T}}{4} \Vert \nabla_\xi f^\varepsilon\Vert^2_{L^2(\T^3)} \leq \frac{C_{t,T}}{2c_{t,T}} + c_{t,T}^{-1} \Vert f^\varepsilon\Vert^2_{L^3(\T^3)} \Vert \nabla \rho \Vert_{L^6(\T^2)}^2. 
    \end{equation*}
    By integrating in time, using that $\Vert f^\varepsilon(t) \Vert_{L^2(\T^3)} \leq C \Vert f(t) \Vert_{L^2(\T^3)} \leq C' \Vert f(t) \Vert_{L^3(\T^3)}$, which is finite by Definition \ref{def:function space}, we find that $\Vert f^\varepsilon \Vert^2_{L^\infty((t,T);L^2(\T^3))} + \Vert \nabla_\xi f^\varepsilon\Vert^2_{L^2((t,T)\times\T^3)}$ is bounded independently of $\varepsilon$. The result then follows from the lower semicontinuity of the norms, by letting $\varepsilon \to 0^+$; the equation \eqref{eq:main eqn} then yields $\partial_t f \in L^2((t,T);H^{-1}(\T^3))$. 
\end{proof}

We are ready to give the proof of Proposition \ref{prop:f bounded moser}; the argument is different from the one presented in \cite[Proposition 6.7]{reg2}, as we use Moser's iterations---not those of De Giorgi. Note that the integrability of the derivatives of $f$ is not sufficient (at this stage) to obtain $f$ bounded directly by Morrey's embedding, which highlights the necessity of this argument.

\begin{proof}[Proof of Proposition \ref{prop:f bounded moser}]
 Let $R \leq 1,s \in [t,T]$ be such that $Q_R(s)\subset [t,T]\!\times\!\T^3$. We show there exist $C_{t,T}$ (depending only on $t,T,c_{t,T}$) and $\beta>0$ (independent $t,T,s,R$), such that 
  \begin{equation}\label{eq:moser upper for f}
      \Vert f \Vert_{L^\infty(Q_r(s))} \leq \frac{C_{t,T}}{(R-r)^\beta} \qquad \forall 0<r<R, 
  \end{equation}
  from which the conclusion of the proposition follows immediately by a covering argument; recall that we are only interested in the \emph{qualitative} statement $f \in L^\infty_{t,\loc}L^\infty_\xi$. Without loss of generality (arguing by translation), we may assume $s=0$ for our cylinders: herein, 
  \begin{equation}\label{eq:the translated cylinders stuff}
      \essinf_{Q_R}(1\!-\!\rho) \geq c_{t,T} , \qquad \Vert \Delta \rho \Vert_{L^4(Q_R)} \leq \Vert \Delta \rho \Vert_{L^4((t,T)\times\T^2)}, 
  \end{equation}
  and the term involving $\Delta \rho$ is finite by Corollary \ref{cor:continuity rho and hess in L4 nabla rho L8}. 

  \smallskip 

  \noindent 1. \textit{Caccioppoli inequality}. Let $M>0,q \geq 2, 0<r<R\leq 1$. In this step, we prove 
\begin{equation}\label{eq:first checkpoint moser bis pour f}
   \begin{aligned}
      \Vert (f \wedge M)^{\frac{q}{2}} \Vert_{L^\infty((-r,r_;L^2(\T^3))}^2 + \Vert \nabla_\xi (f \wedge M)^{\frac{q}{2}} \Vert_{L^2(Q_r)}^2   \leq \frac{\kappa q^2}{R\!-\!r} \Vert f^{\frac{q}{2}} \Vert_{L^{\nicefrac{8}{3}}(Q_R)}^2 , 
   \end{aligned} 
\end{equation}
where $\kappa := C c_{t,T}^{-2} (1\!+\!\Vert \Delta \rho \Vert^2_{L^4((t,T)\times\T^3)} ) $ and the constant $C$ is independent of $q,M,c_{t,T},R,r$ and all other norms of $f$ and its averaged quantities (\textit{i.e.~}$\rho,\p$, \emph{etc.}). The final term on the right-hand side of \eqref{eq:first checkpoint moser bis pour f} is \emph{a priori} only finite for $2\leq q \leq \nicefrac{5}{2}$ by Lemma \ref{lem:f parabolic space}; nevertheless, the inequality remains true for all $q \geq 2$, with the possibility of this term being infinite.

Let $\psi \in C^\infty_c((-R,R))$ be a localising function (depending only on time) such that 
\begin{equation*}
    0 \leq \psi \leq 1 , \quad \psi|_{[-r,r]}=1, \quad \psi|_{[-1,1]\setminus[-R,R]}=0, \quad \Vert \partial_t \psi \Vert_{L^\infty((-1,1))} \leq \frac{C}{R-r},
    \end{equation*}
    for some constant $C$ independent of $r,R$. We wish to test \eqref{eq:main eqn} with $q (f\wedge M)^{q-1} \psi^2 \in L^2_{t,\loc}H^1_\xi$ by Lemma \ref{lem:f parabolic space}, since we \emph{formally expect} $\frac{\der}{\der t} \Vert (f \wedge M)^{\frac{q}{2}} \Vert^2_{L^2_\xi}=\langle \partial_t f , q (f \wedge M)^{q-1} \rangle$. However, since we only know $\partial_t f \in L^2_{t,\loc}H^{-1}_\xi$, there is no obvious non-linear chain rule for the pairing $\langle \partial_t f , q (f\wedge M)^{q-1}  \rangle$; the Lions--Magenes Lemma only applies when $q=2$. 
    
    A common strategy is to use \emph{Steklov averages} to bypass this: for all $h \in (0,1)$, define 
   \begin{equation}\label{eq:steklov def}
       f^h(s,\xi) := \frac{1}{h}\int_s^{s+h} f(s',\xi) \d s' \implies \partial_t f^h(s,\xi) = \frac{f(s+h,\xi) - f(s,\xi)}{h} \qquad \text{a.e.} 
   \end{equation}
  We then obtain the evolution on $f^h$ from \eqref{eq:main eqn} by applying $\frac{1}{h}\int_s^{s+h}$ to the equation, \textit{i.e.}, 
   \begin{equation}\label{eq:steklov averaged eq}
       \partial_t f^h + \dv( \big[(1-\rho) f \big]^h \e(\theta)) = \dv( \big[ (1-\rho) \nabla f + f \nabla \rho \big]^h ) + \partial^2_\theta f^h. 
   \end{equation}
   Now we test \eqref{eq:steklov averaged eq} with $q (f^h \wedge M)^{q-1} \psi^2  $. By \eqref{eq:steklov def}, $\partial_t f^h \in L^3_\loc([0,\infty);L^3(\T^3))$, whence the usual chain rule for Sobolev functions applies (\textit{cf.}~\cite[\S 5.10]{Evans}), and we get 
   \begin{equation*}
       \langle \partial_t f^h , q (f^h \wedge M)^{q-1} \psi^2  \rangle = \frac{\der}{\der t} \int_{\T^3} \Phi_{M,q}(f^h) \psi^2 \d \xi  - 2 \int_{\T^3} \psi \partial_t \psi \Phi_{M,q}(f^h) \d \xi , 
   \end{equation*}
   where $\Phi_{M,q}(s) = \int_0^s q(r \!\wedge\! M)^{q-1} \d r$. Observe that (for $q \geq 2$) this function satisfies $\frac{\der}{\der s}(s \wedge M)^q = q (s \wedge M)^{q-1} \mathds{1}_{\{0 \leq s \leq M\}} \leq q(s \wedge M)^{q-1} = \Phi'_{M,q}(s) \leq q s^{q-1} = \frac{\der}{\der s} s^q$, whence 
    \begin{equation}\label{eq:Phi for steklov}
      (s \wedge M)^q \leq \Phi_{M,q}(s) \leq s^q \qquad \forall s \geq 0. 
   \end{equation}
 Integrating in time over $[-R,s]$ for $s \in (-r,r)$, for which $\psi(s)=1$, $\psi(-R)=0$, we get 
   \begin{equation}\label{eq:steklov long}
       \begin{aligned}
           \int_{\T^3}\!&  \Phi_{M,q}(f^h(s)) \!+\! q(q\!-\!1) \!  \int_{-R}^s \!\int_{\T^3} \!\!\psi^2 (f^h\!\!\wedge\!M)^{q-2} \Big( \big[ (1\!-\!\rho) \nabla f \!+\! f \nabla \rho \big]^h \!\cdot\! \nabla f^h \!\!+\!    |\partial_\theta f^h|^2 \Big) \mathds{1}_{\{f^h < M\}} \\ 
           &\!= \!2\! \int_{-R}^s \! \int_{\T^3}\!\!  \psi \partial_t \psi \Phi_{M,q}(f^h) \!+\!  q(q\!-\!1) \! \int_{-R}^s \! \int_{\T^3} \!\! \psi^2 \big[(1-\rho) f \big]^h \e(\theta) \!\cdot\! (f^h \wedge M)^{q-2} \nabla f^h \mathds{1}_{\{f^h < M\}}. 
       \end{aligned}
   \end{equation}
   Standard results on Steklov averages (\textit{e.g.~}\cite[\S 1.3 Lemma 3.2]{DiBenedetto}) imply $\nabla_\xi f^h \to \nabla_\xi f$ and $[(1-\rho)\nabla f + f \nabla \rho]^h \to (1-\rho)\nabla f + f \nabla \rho$ strongly $L^2$ (by Lemma \ref{lem:f parabolic space}, Corollary \ref{cor:continuity rho and hess in L4 nabla rho L8}), hence 
   \begin{equation}\label{eq:steklov strong L1}
       \Big( [ (1\!-\!\rho) \nabla f \!+\! f \nabla \rho ]^h \!\cdot\! \nabla f^h \!+    |\partial_\theta f^h|^2 \Big) \to \Big( [ (1\!-\!\rho) \nabla f \!+\! f \nabla \rho ] \cdot\! \nabla f \!+    |\partial_\theta f|^2 \Big) \quad \text{strongly in } L^1. 
   \end{equation}
   Also, $f^h \to f$ strongly in $L^3$ and a.e.~(for a subsequence). Hence, $(f^h \!\wedge\! M)^{q-2} \mathds{1}_{\{f^h < M\}} \to (f\! \wedge\! M)^{q-2} \mathds{1}_{\{f < M\}} $ a.e.~and, since the term is bounded by $M^{q-2}$, the Dominated Convergence Theorem implies $\iint \psi^2 (f^h \wedge M)^{q-2} \mathds{1}_{\{f^h < M\}} \varphi \d \xi \d t \to \iint \psi^2 (f \wedge M)^{q-2} \mathds{1}_{\{f < M\}} \varphi \d \xi \d t$ for all $\varphi \in L^1$, \textit{i.e.}, 
   \begin{equation}\label{eq:steklov weak star Linfty}
      \psi^2 (f^h \wedge M)^{q-2} \mathds{1}_{\{f^h < M\}} \overset{*}{\rightharpoonup} \psi^2 (f \wedge M)^{q-2} \mathds{1}_{\{f < M\}} \quad \text{weakly-* in } L^\infty. 
   \end{equation}
   Combining \eqref{eq:steklov strong L1} and \eqref{eq:steklov weak star Linfty} as a weak-strong product implies 
   \begin{equation*}
       \begin{aligned}
          \lim_{h \to 0^+} \int_{-R}^s \!\int_{\T^3} \!\!\psi^2 (f^h\wedge & M)^{q-2} \Big( \big[ (1\!-\!\rho) \nabla f \!+\! f \nabla \rho \big]^h \!\cdot\! \nabla f^h \!\!+\!    |\partial_\theta f^h|^2 \Big) \mathds{1}_{\{f^h < M\}} \\ 
          &= \int_{-R}^s \!\int_{\T^3} \!\!\psi^2 (f\!\!\wedge\!M)^{q-2} \Big( \big[ (1\!-\!\rho) \nabla f \!+\! f \nabla \rho \big] \!\cdot\! \nabla f \!\!+\!    |\partial_\theta f|^2 \Big) \mathds{1}_{\{f < M\}} , 
       \end{aligned}
   \end{equation*}
   and an analogous argument yields the expected convergence of the final term in \eqref{eq:steklov long}. Meanwhile, directly from the definition of $\Phi_{M,q}(s) = \int_0^s q(r \wedge M)^{q-1} \d r$, it holds 
   \begin{equation*}
      \begin{aligned}
          \Phi_{M,q}(f^h) &= (f^h)^q \mathds{1}_{\{f^h  < M\}} + \big[ M^q + qM^{q-1}(f^h -M) \big] \mathds{1}_{\{f^h  \geq M\}} \\ 
          &= \Big[ (f^h)^q \!-\! \big[ M^q \!+\! qM^{q-1}(f^h\!-\!M) \big]   \Big]\mathds{1}_{\{f^h < M\}} \!+\! \big[ M^q \!+\! qM^{q-1}(f^h -M) \big] , 
      \end{aligned}
   \end{equation*}
and we have that the first term on the right-hand side is uniformly bounded in $h$ by $2(1+q)M^q $ and converges a.e.~to $[f^q \!-\! [ M^q \!+\! qM^{q-1}(f\!-\!M) ] ]\mathds{1}_{\{f < M\}}$ (for a subsequence) while the second term converges strongly to $M^q + qM^{q-1}(f-M)$ in $L^3(\T^3)$ by standard results on Steklov averages. Using also the Dominated Convergence Theorem, we deduce 
\begin{equation*}
    \lim_{h\to 0^+} 2 \int_{-R}^s  \int_{\T^3}   \psi \partial_t \psi \Phi_{M,q}(f^h)  = 2 \int_{-R}^s  \int_{\T^3}  \psi \partial_t \psi \Phi_{M,q}(f). 
\end{equation*}
Meanwhile, Fatou's Lemma implies $\int_{\T^3} \Phi_{M,q}(f(s)) \leq \liminf_{h\to 0^+} \int_{\T^3} \Phi_{M,q}(f^h(s))$. Thus, by letting $h \to 0^+$,  using also the lower bound \eqref{eq:the translated cylinders stuff}, we deduce from \eqref{eq:steklov long} that (a.e.~$s$) 
      \begin{equation*}
       \begin{aligned}
           \int_{\T^3}   \Phi_{M,q}(f(s)) + & q(q\!-\!1)c_{t,T} \int_{-R}^s \!\int_{\T^3} \psi^2 |(f\!\wedge\!M)^{\frac{q}{2}-1} \nabla_\xi f|^2 \mathds{1}_{\{f < M\}} \\ 
           \leq & \,  q(q\!-\!1)\bigg|\int_{-R}^s \!\int_{\T^3}\!\! \psi^2 (f\!\wedge\! M)^{q-1}  \nabla f \mathds{1}_{\{f<M\}} \!\cdot\! \! \nabla \rho \bigg| \\ 
           &+  q(q\!-\!1) \int_{-R}^s \! \int_{\T^3} \!\! \psi^2 \Big| (f \!\wedge\! M)^{q-1} \nabla f \mathds{1}_{\{f < M\}} \Big| + \frac{C}{R-r} \int_{-R}^s \! \int_{\T^3} \! \Phi_{M,q}(f). 
       \end{aligned}
   \end{equation*}
Using the lower bound \eqref{eq:Phi for steklov}, $0 \leq \psi \leq 1$, and integration by parts to get $\Delta \rho$, we deduce 
  \begin{equation*}
       \begin{aligned}
           &\int_{\T^3}  |(f(s)\wedge M)^{\frac{q}{2}}|^2 +  4(1\!-\!q^{-1})c_{t,T} \int_{-R}^s \!\int_{\T^3}   |\nabla_\xi [\psi (f\!\wedge\!M)^{\frac{q}{2}}] |^2  \\ 
           &\leq  (q\!-\!1)\!\!\! \iint_{Q(R)}\!\! \!\! \!\! | (f\!\wedge\! M)^{\frac{q}{2}} |^2 |\Delta \rho| \!+\!  2(q\!-\!1) \!\!\!\int_{-R}^s \! \int_{\T^3} \!\! (f \!\wedge\! M)^{\frac{q}{2}} | \nabla [\psi (f \!\wedge\! M)^{\frac{q}{2}}] | \!+\! \frac{C}{R-r}\!\! \iint_{Q(R)}\!\! \!\!\!\! \Phi_{M,q}(f) , 
       \end{aligned}
   \end{equation*}
for a.e.~$s \in (-r,r)$. Using \eqref{eq:the translated cylinders stuff}, Young's inequality, and the upper bound \eqref{eq:Phi for steklov}, it holds 
  \begin{equation*}
       \begin{aligned}
           &\Vert (f\wedge  M)^{\frac{q}{2}}\Vert_{L^\infty((-r,r);L^2(\T^3))}^2 +  3(1\!-\!q^{-1})c_{t,T} \Vert \nabla_\xi [(f\!\wedge\!M)^{\frac{q}{2}}] \Vert_{L^2(Q_r)}^2  \\ 
           &\leq \! (q\!-\!1) \Vert | (f\!\wedge\! M)^{\frac{q}{2}} |^2 \Vert_{L^{\frac{4}{3}}(Q_R)}  \Vert \Delta \rho\Vert_{L^4(Q_R)}  \!+\!  \frac{q(q\!-\!1)}{c_{t,T}} \Vert (f \!\wedge\! M)^{\frac{q}{2}} \Vert_{L^2(Q_R)}^2 \!+\!\frac{C}{R\!-\!r} \Vert f^{\frac{q}{2}} \Vert^2_{L^2(Q_R)}; 
       \end{aligned}
   \end{equation*}
note the final term is \emph{a priori} only finite for $2\leq q \leq \nicefrac{5}{2}$ by Lemma \ref{lem:f parabolic space}; nevertheless, the inequality remains true for all $q \geq 2$, with the possibility of this term being infinite. Using $1-q^{-1} \geq \nicefrac{1}{2}$ and H\"older's inequality to control the $L^2$ norm by $L^{\nicefrac{8}{3}}$ we deduce \eqref{eq:first checkpoint moser bis pour f}.

\smallskip 

\noindent 2. \textit{Moser iteration}. Lemma \ref{lem:dibenedetto classic} and \eqref{eq:first checkpoint moser bis pour f} give $ \Vert  (f \wedge M)^{\frac{q}{2}} \Vert_{L^{\nicefrac{10}{3}}(Q_r)}^2 \!\leq\! \frac{C\kappa q^2}{R\!-\!r}  \Vert f^{\frac{q}{2}} \Vert_{L^{\nicefrac{8}{3}}(Q_R)}^2 $, and thus, by letting $M\to \infty$ and the Monotone Convergence Theorem, it holds 
\begin{equation}\label{eq:first checkpoint moser bis pour f recursive}
   \begin{aligned}
      \Vert  f \Vert_{L^{\nicefrac{5q}{3}}(Q_r)} \leq \Big( \frac{C\kappa q^2}{R\!-\!r}  \Big)^{\nicefrac{1}{q}} \Vert f \Vert_{L^{\nicefrac{4q}{3}}(Q_R)}, 
   \end{aligned} 
\end{equation}
where we recall $C,\kappa$ are independent of $q,r,R$; the above holds for all choices of $0<r<R\leq 1$ and all $q \geq 2$. Next, define the powers $q_n := 2 (\frac{5}{4})^n$ for $n \geq 0$ and the radii 
\begin{equation*}
   r_0 = R, \qquad  r_n = R-(R-r) \gamma^{-1} \sum_{j=1}^n (\frac{5}{4})^{-j}
\end{equation*}
where $\gamma = \sum_{n\geq 1}(5/4)^{-n}$, so that $R=r_0 > r_1 > \dots > r_n \searrow r$. Note that $r_n - r_{n+1} = (R-r)\gamma^{-1} (5/4)^{-n-1}$. Iterating \eqref{eq:first checkpoint moser bis pour f recursive} along this sequence (setting $q=q_n$, $R=r_n$, $r=r_{n+1}$), it holds---by setting $L := 5C\kappa \gamma$ and $\alpha := 5/4$---the estimate 
\begin{equation}\label{eq:f moser before n to infty}
    \Vert  f\Vert_{L^{\frac{8 }{3}\alpha^{n+1}}(Q_{r_{n+1}})} \leq \Big( \frac{L  \alpha^{3n}}{R\!-\!r}  \Big)^{\frac{1}{2\alpha^n}} \Vert f \Vert_{L^{\frac{8}{3}\alpha^n}(Q_{r_n})} \leq  \bigg( \prod_{j=0}^{n}\Big( \frac{L  \alpha^{3j}}{R\!-\!r}  \Big)^{\frac{1}{2\alpha^j}} \bigg) \Vert f \Vert_{L^{\frac{8}{3}}(Q_{R})}. 
\end{equation}
We study the convergence of the product. We have 
\begin{equation*}
    \begin{aligned}
        \log\bigg( \prod_{j=0}^{n}\Big( \frac{L  \alpha^{3j}}{R\!-\!r}  \Big)^{\frac{1}{2\alpha^j}} \bigg) =& \,  \frac{\alpha}{2}\log\Big(\frac{L}{R\!-\!r}\Big) \sum_{j=1}^{n+1} \alpha^{-j} \!+\! \frac{3}{2} \log \alpha \sum_{j=0}^n j \alpha^{-j} \nearrow  \frac{\alpha \gamma}{2} \log\Big(\frac{L}{R\!-\!r}\Big) \!+\! \bar{C}
    \end{aligned}
\end{equation*}
where $\bar{C}$ depends only on $\alpha$. Returning to \eqref{eq:f moser before n to infty} and letting $n\to\infty$, it holds 
\begin{equation*}
     \Vert  f\Vert_{L^{\infty}(Q_{r})} \leq C\bigg(\frac{L}{R-r}\bigg)^{\frac{5\gamma}{8}} \Vert f \Vert_{L^{\nicefrac{8}{3}}(Q_R)} , 
\end{equation*}
where $C$ depends only on $\alpha$. Using $\Vert f \Vert_{L^{8/3}(Q_R)} \leq C_{t,T} \Vert f \Vert_{L^{10/3}((t,T)\times\T^3)}$ and Lemma \ref{lem:f parabolic space}, and recalling the definitions of $L,\kappa$ from \eqref{eq:first checkpoint moser bis pour f}, we get \eqref{eq:moser upper for f}. 
\end{proof}

\subsection{Control of $L^p$ norms of $\partial_t f, \nabla^2 f$}

Having shown that $f \in L^\infty_{t,\loc}L^\infty_\xi$ in Proposition \ref{prop:f bounded moser}, our next aim is to show the following. 

\begin{prop}[Higher integrability $\partial_t f, \nabla^2_\xi f$]\label{prop:higher integ dt f and hess f}
    Let $0<t<T<\infty$. Then, $\partial_t f , \nabla^2_\xi f \in L^p((t,T)\!\times\!\T^3)$ for all $1\leq p < \infty$. 
\end{prop}

In order to prove this proposition, we first need to perform a $H^2$-type estimate on $f$---in the same spirit as Proposition \ref{prop:galerkin p} for $\p$. This is the content of the next result. 

\begin{lemma}[$H^2$ estimate for $f$]\label{lem:H2 f}
    Let $0<t<T<\infty$. Then, $\partial_t f, \nabla^2_\xi f \in L^2((t,T)\!\times\!\T^3)$ and $\nabla_\xi f \in L^\infty((t,T);L^2(\T^3)) \cap L^4((t,T)\!\times\!\T^3)$. 
\end{lemma}

\begin{proof}
    As in \S \ref{sec:H2 for p}, the underlying notion is that, provided the passage from divergence form \eqref{eq:main eqn} to non-divergence form is justified, then \eqref{eq:main eqn} reads 
    \begin{equation}\label{eq:non div form f}
        \partial_t f + \dv((1-\rho)f\e(\theta)) = (1-\rho)\Delta f + f \Delta \rho + \partial^2_\theta f, 
    \end{equation}
    whence testing with $\Delta f$ (and $\partial^2_\theta f$) gives the conclusion. To do this rigorously, we employ Galerkin's method. Without loss of generality $t$ is a Lebesgue point of $t \mapsto (f(t),\nabla_\xi f(t))$. 

    \smallskip 

    \noindent 1. \textit{Galerkin approximations}. We recall the notations of \S \ref{sec:H2 for p}, with $\{\varphi_j\}_j$ the orthonormal basis of $L^2(\T^2)$ of eigenfunctions of the Laplacian, and define the \emph{ansatz}
\begin{equation*}
    f^n(s,x,\theta) := \sum^n_{j=1} \alpha^n_j(s,\theta) \varphi_j(x), 
\end{equation*}
where the scalar coefficients $\{\alpha^n_j\}_{j=1}^{n}$ are determined from the linear system of PDEs 
\begin{equation}\label{eq:coeff f galerkin}
    \partial_t \alpha^n_j - \partial^2_\theta \alpha^n_j + \sum_{k=1}^n \alpha^n_k\bigg( \int_{\T^2} \Big( (1\!-\!\rho) \nabla \varphi_k  + \varphi_k \nabla \rho \Big)\!\cdot \!\nabla \varphi_j \d x \bigg) \!=\! \int_{\T^2}  (1-\rho) f \e(\theta) \!\cdot \!\nabla \varphi_j \d x, 
\end{equation}
with initial data $\alpha^n_j(t,\theta) = \langle f(t,\cdot,\theta) , \varphi_j \rangle, $ chosen such that $f^n$ solves the \emph{linear} equation
\begin{equation}\label{eq:galerkin f}
    \left\lbrace\begin{aligned}
        &\partial_t f^n + \dv((1-\rho)f \e(\theta)) = \dv((1-\rho)\nabla f^n + f^n \nabla \rho ) + \partial^2_\theta f^n, \\ 
        &f^n(t,\cdot,\theta) = \sum_{j=1}^n \langle f(t,\cdot,\theta) , \varphi_j \rangle \varphi_j 
    \end{aligned}\right.
\end{equation}
in the weak sense dual to $L^2((t,T);H^1(\T);\mathrm{span}\{\varphi_j\}_{j=1}^n)$. The standard theory for strongly parabolic linear systems (recall \eqref{eq:ctT def}) implies the existence $\alpha^n_j \in L^2((t,T);H^1(\T))$ with $\partial_t \alpha^n_j \in L^2((t,T);H^{-1}(\T))$ solving \eqref{eq:coeff f galerkin}, and thus of $f^n$ solving \eqref{eq:galerkin f} in the aforementioned weak sense. Additionally, we have (as $n\to\infty$)
\begin{equation*}
    \Vert f^n(t,\cdot,\theta) \Vert^2_{L^2(\T^2)} = \sum_{j=1}^n \langle f(t,\cdot,\theta) , \varphi_j \rangle^2 \longrightarrow \Vert f(t,\cdot,\theta) \Vert^2_{L^2(\T^2)} \quad \text{a.e.~} \theta. 
\end{equation*}
Integrating in $\theta$ gives $\Vert f^n(t) \Vert^2_{L^2(\T^3)} = \int_\T \sum_{j=1}^n \langle f(t,\cdot,\theta) , \varphi_j \rangle^2 \d \theta,$ so  the Monotone Convergence Theorem (the integrands form an increasing non-negative sequence) implies 
\begin{equation*}
   \lim_{n\to\infty}\Vert f^n(t) \Vert^2_{L^2(\T^3)} = \int_\T \lim_{n\to\infty} \sum_{j=1}^n \langle f(t,\cdot,\theta) , \varphi_j \rangle^2 \d \theta = \int_\T \Vert f(t,\cdot,\theta) \Vert^2_{L^2(\T^2)} \d \theta. 
\end{equation*}
Analogously, by similar manipulations as \eqref{eq:conv of initial data for galerkin}, this convergence holds for $\Vert \nabla_\xi f(t)\Vert_{L^2}$:
\begin{equation}\label{eq:conv initial data galerkin f}
    \lim_{n\to\infty}\Vert f^n(t) \Vert^2_{L^2(\T^3)} = \Vert f(t) \Vert^2_{L^2(\T^3)}, \quad \lim_{n\to\infty}\Vert \nabla_{\xi} f^n(t) \Vert^2_{L^2(\T^3)} = \Vert \nabla_{\xi} f(t,\cdot) \Vert^2_{L^2(\T^3)}. 
\end{equation}

\smallskip 

\noindent 2. \textit{$H^1$ estimate}. We begin with the standard $H^1$ estimate, by testing \eqref{eq:galerkin f} with $f^n$ itself; this is allowed since $f^n $ is written as a linear combination of the $\{\varphi_j\}$. Using the Lions--Magenes Lemma over the finite-dimensional Hilbert space $\mathrm{span}\{\varphi_j\}_{j=1}^n$, and \eqref{eq:ctT def}, 
\begin{equation*}
    \begin{aligned}
        \frac{1}{2}\frac{\der}{\der t}\Vert f^n \Vert_{L^2(\T^3)}^2 + \frac{c_{t,T}}{2} \Vert & \nabla_{\xi} f^n \Vert^2_{L^2(\T^3)} \leq \frac{1}{2c_{t,T}}\Vert (1\!-\!\rho) f \e(\theta) \Vert^2_{L^2(\T^3)} - \int_{\T^3} f^n \nabla f^n \!\cdot\! \nabla \rho \d \xi \\ 
        &\leq \frac{1}{2c_{t,T}}\Vert f \Vert^2_{L^2(\T^3)} + \frac{1}{c_{t,T}} \int_{\T^3} |f^n|^2 | \nabla \rho |^2 \d \xi + \frac{c_{t,T}}{4} \Vert \nabla_\xi f^n \Vert^2_{L^2(\T^3)} , 
    \end{aligned}
\end{equation*}
whence $ \frac{1}{2}\frac{\der}{\der t}\Vert f^n \Vert_{L^2(\T^3)}^2 + \frac{c_{t,T}}{4} \Vert\nabla_{\xi} f^n \Vert^2_{L^2(\T^3)} \leq \frac{1}{2c_{t,T}}\Vert f \Vert^2_{L^2(\T^3)} + \frac{1}{c_{t,T}}\Vert f^n\Vert_{L^2(\T^3)}^2 \Vert \nabla \rho \Vert_{L^\infty(\T^3)}^2$, and integrating in time yields (for all $t<s\leq T$) 
\begin{equation*}
    \begin{aligned}
     \Vert f^n (s) \Vert^2_{L^2(\T^3)} + {c_{t,T}}\Vert \nabla_\xi f^n \Vert^2_{L^2((t,s)\times\T^3)} \leq & \, \Vert f^n(t) \Vert^2_{L^2(\T^3)} \!+\! c_{t,T}^{-1} \Vert f \Vert^2_{L^2((t,T)\times\T^3)}   \\ 
     &+ 2c_{t,T}^{-1} \int_t^s \! \Vert \nabla \rho (s')\Vert_{L^\infty(\T^3)}^2  \Vert f^n(s') \Vert_{L^{2}(\T^3)}^2 \d s' \\ 
     \leq & \, C_{t,T} \bigg( 1 + \int_t^s \! \Vert \nabla \rho (s')\Vert_{L^\infty(\T^3)}^2  \Vert f^n(s') \Vert_{L^{2}(\T^3)}^2 \d s' \bigg) , 
    \end{aligned}
\end{equation*}
where we used the convergence of initial data \eqref{eq:conv initial data galerkin f} to establish the boundedness-in-$n$ of the right-hand side, and $C_{t,T}$ depends only on $t,T,c_{t,T}$. By Gr\"onwall's Lemma, we deduce 
\begin{equation*}
    \begin{aligned}
        \Vert f^n \Vert^2_{L^\infty((t,T);L^2(\T^3))} \leq C_{t,T} \exp \bigg( C_{t,T} \int_t^T \Vert \nabla \rho(s) \Vert^2_{L^\infty(\T^3)} \d s \bigg) , 
    \end{aligned}
\end{equation*}
where we notice that the final term is bounded by Corollary \ref{cor:consequence of gehring}. Returning to the previous estimate, 
we deduce that there exists $C_{t,T}$ depending only on $t,T,c_{t,T}$ (not $n$) such that 
\begin{equation}\label{eq:H1 bound galerkin f}
  \Vert f^n \Vert_{L^\infty((t,T);L^2(\T^3))}^2 + \Vert \nabla_\xi f^n \Vert^2_{L^2((t,T)\times\T^3)} \leq C_{t,T}. 
\end{equation}

\smallskip 

\noindent 3. \textit{$H^2$ estimate}. We pass from the divergence form to non-divergence form of \eqref{eq:galerkin f}; this is allowed since each $f^n$ is smooth in $x$ (being a linear combination of eigenfunctions of the Laplacian). Testing with $\Delta f^n$, which is allowed since $\Delta f^n(t,\cdot,\theta) \in \mathrm{span}\{\varphi_j\}_{j=1}^n$ a.e.~$t,\theta$, 
\begin{equation*}
    \begin{aligned}
        \frac{1}{2}\frac{\der}{\der t}\Vert \nabla f^n \Vert^2_{L^2(\T^3)} \!+\! \frac{c_{t,T}}{2} \Vert \Delta f^n \Vert^2_{L^2(\T^3)} \!+ & \Vert \nabla \partial_\theta f^n \Vert^2_{L^2(\T^3)} \\ 
        &\leq \frac{1}{2c_{t,T}} \Vert \dv((1\!-\!\rho) f \e(\theta) ) \Vert^2_{L^2(\T^3)} - \int_{\T^3}f^n \Delta f^n \Delta \rho \d \xi , 
    \end{aligned}
\end{equation*}
so that, by Young's inequality and integrating in time, for all $t<s\leq T$, 
\begin{equation*}
    \begin{aligned}
         \Vert \nabla & f^n(s) \Vert^2_{L^2(\T^3)}  \!+\! \frac{c_{t,T}}{2} \Vert \Delta f^n \Vert^2_{L^2((t,s)\times\T^3)} \!+ \! \Vert \nabla \partial_\theta f^n \Vert^2_{L^2((t,s)\times\T^3)} \\ 
        &\leq \Vert \nabla f^n(t) \Vert^2_{L^2(\T^3)} + c_{t,T}^{-1} \Vert \dv((1\!-\!\rho) f \e(\theta)) \Vert^2_{L^2((t,T)\times\T^3)} + 2 c_{t,T}^{-1}\int_t^s \Vert f^n \Delta \rho \Vert_{L^2(\T^3)}^2 \d s' . 
    \end{aligned}
\end{equation*}
From the convergence of the initial data  \eqref{eq:conv initial data galerkin f}, Lemma \ref{lem:f parabolic space}, and Proposition \ref{prop:f bounded moser}, we deduce that there exists $C_{t,T}$ depending only on $t,T$ such that 
\begin{equation}\label{eq:long part of H2 for f}
    \begin{aligned}
         \Vert \nabla  f^n(s) \Vert^2_{L^2(\T^3)}  \!+\! \frac{c_{t,T}}{2} \Vert \Delta f^n & \Vert^2_{L^2((t,s)\times\T^3)} + \! \Vert \nabla \partial_\theta f^n \Vert^2_{L^2((t,T)\times\T^3)} \\ 
        &\leq C_{t,T} \bigg( 1 \!+\! \int_t^s\!\! \int_\T \!\! \Vert (f^n \Delta \rho) (s',\theta') \Vert_{L^2(\T^2)}^2 \d \theta' \d s' \bigg) , 
    \end{aligned}
\end{equation}
note that we have written the $\theta$ integral explicitly in the final term. Using H\"older's inequality, this final term may be estimated as (since $\rho$ is independent of $\theta$)
\begin{equation*}
    \begin{aligned}
         \int_t^s\!\! \int_\T \!\! \Vert (f^n \Delta \rho) (s',\theta') \Vert_{L^2(\T^2)}^2 \d \theta' \d s'  &\leq \int_t^s\!\! \int_\T \!\! \Vert f^n (s',\theta') \Vert^2_{L^4(\T^2)} \Vert \Delta \rho (s',\theta') \Vert_{L^4(\T^2)}^2 \d \theta' \d s' \\ 
         &\leq C\int_t^s\!\! \Vert \Delta \rho (s') \Vert_{L^4(\T^2)}^2 \bigg( \int_\T \!\! \Vert f^n (s',\theta') \Vert^2_{H^1(\T^2)}  \d \theta' \bigg) \d s'
    \end{aligned}
\end{equation*}
where we used the Sobolev inequality in $\T^2$ (\textit{i.e.~}with respect to only $x$), and the constant $C$ depends only on $\T^2$ and the exponent $4$. Returning to \eqref{eq:long part of H2 for f}, we get 
\begin{equation*}
    \begin{aligned}
         \Vert &\nabla  f^n(s) \Vert^2_{L^2(\T^3)}  +\! \frac{c_{t,T}}{2}  \Vert \Delta f^n \Vert^2_{L^2((t,s)\times\T^3)} + \! \Vert \nabla \partial_\theta f^n \Vert^2_{L^2((t,s)\times\T^3)} \\ 
        &\leq C_{t,T} \bigg( 1 \!+\! \int_t^s\!\! \Vert \Delta \rho (s') \Vert_{L^4(\T^2)}^2 \Big( \Vert f^n (s') \Vert^2_{L^2(\T^3)} + \Vert \nabla f^n (s') \Vert^2_{L^2(\T^3)} \Big) \d s' \bigg)\\ 
        &\leq C_{t,T} \bigg( 1 \!+\! \Vert f^n \Vert^2_{L^\infty((t,T);L^2(\T^3))} \Vert \Delta \rho \Vert_{L^4((t,T)\times\T^3)}^2 \!+\! \int_t^s\!\! \Vert \Delta \rho (s') \Vert_{L^4(\T^2)}^2 \Vert \nabla f^n (s') \Vert^2_{L^2(\T^3)}  \d s' \bigg), 
    \end{aligned}
\end{equation*}
\textit{i.e.}, $\Vert \nabla  f^n(s) \Vert^2_{L^2} \leq C_{t,T}(1\!+\!\int_t^s \Vert \Delta \rho (s') \Vert_{L^4}^2 \Vert \nabla f^n (s') \Vert^2_{L^2)}  \d s' )$, and Gr\"onwall's Lemma implies 
\begin{equation*}
    \Vert \nabla f^n \Vert^2_{L^\infty((t,T);L^2(\T^3))} \leq C_{t,T} \exp\bigg( \int_t^T \Vert \Delta \rho(s) \Vert^2_{L^4(\T^2)} \d s \bigg), 
\end{equation*}
and the final integral is bounded on account of $\Delta \rho \in L^4((t,T)\!\times\!\T^3)$ by Corollary \ref{cor:continuity rho and hess in L4 nabla rho L8}. Returning to the previous estimate, we deduce that, for some $C_{t,T}$ independent of $n$, 
\begin{equation}\label{eq:H2 glaerkin f n}
    \begin{aligned}
         \Vert \nabla  f^n \Vert^2_{L^\infty((t,T);L^2(\T^3))}  +\! \Vert \Delta f^n \Vert^2_{L^2((t,T)\times\T^3)} + \! \Vert \nabla \partial_\theta f^n \Vert^2_{L^2((t,T)\times\T^3)} \leq C_{t,T}. 
    \end{aligned}
\end{equation}

\smallskip 

\noindent 4. \textit{Convergence and conclusion}. An analogous strategy to Step 3, but testing this time with $\partial^2_\theta f^n$ (again this belongs to the span of the eigenfunctions for a.e.~$t,\theta$) shows the uniform-in-$n$ boundedness of $\Vert \partial^2_\theta f^n \Vert_{L^2}$. Putting together with the bounds \eqref{eq:H1 bound galerkin f} and \eqref{eq:H2 glaerkin f n}, we get that $f^n \rightharpoonup {f}^\infty$ weakly in $L^2((t,T);H^2(\T^3))$, and an argument analogous to Step 3 of the proof of Proposition \ref{prop:galerkin p} shows that ${f}^\infty=f$. Hence estimate \eqref{eq:H2 glaerkin f n} also holds for $f$.

Finally, using Proposition \ref{prop:f bounded moser} and Gagliardo--Nirenberg inequality (for $C$ depending only on $\T^3$), $\Vert \nabla_{\bxi} f(s) \Vert_{L^4(\T^3)} \leq C ( \Vert \nabla_{\bxi}^2 f(s) \Vert_{L^2(\T^3)}^{\frac{1}{2}} \Vert f(s) \Vert_{L^\infty(\T^3)}^{\frac{1}{2}} + \Vert f(s) \Vert_{L^1(\T^3)})$, and thus $\Vert \nabla_{\bxi} f \Vert_{L^4((t,T)\times\T^3)}^4 \leq C ( \Vert \nabla_{\bxi}^2 f \Vert_{L^2((t,T)\times\T^3)}^{2} \Vert f \Vert_{L^\infty((t,T)\times\T^3)}^{2} + (T-t) )$. 
\end{proof}

We prove Proposition \ref{prop:higher integ dt f and hess f}, by an inductive argument similar to \cite[Proposition 7.1]{reg2}. 

\begin{proof}[Proof of Proposition \ref{prop:higher integ dt f and hess f}]
  The proof is divided into two steps. 
  
  \smallskip 
  
  \noindent 1. \textit{Initial upgrade of integrability}. By Lemma \ref{lem:H2 f}, the $\xi$-derivatives of $f$ are endowed with sufficient integrability to rigorously pass to the non-divergence form of the equation \eqref{eq:non div form f}, which holds weakly in duality with $L^2((t,T)\times\T^3)$. By defining the matrix 
  \begin{equation}\label{eq:A def for non div f}
      A := \left( \begin{matrix}
          (1-\rho)\mathrm{Id}_2 & 0 \\ 
          0 & 1
      \end{matrix} \right) , 
  \end{equation}
  which is symmetric, $\nicefrac{1}{4}$-H\"older continuous by Corollary \ref{cor:continuity rho and hess in L4 nabla rho L8}, and strongly elliptic by the lower bound on $1-\rho$ provided by \eqref{eq:ctT def}, we may recast \eqref{eq:non div form f} in the form 
  \begin{equation}\label{eq:recast dt f fess f in all Lp}
      \partial_t f - A:\nabla^2_\xi f = f\Delta \rho + f\e(\theta)\cdot\nabla \rho - (1-\rho)\nabla f \cdot \e(\theta), 
  \end{equation}
and thus $\partial_t f - A:\nabla^2_\xi f \in L^4((t,T)\!\times\!\T^3)$  by Proposition \ref{prop:f bounded moser}, Corollary \ref{cor:continuity rho and hess in L4 nabla rho L8}, and Lemma \ref{lem:H2 f}. In turn, the Schauder result of Lemma \ref{lem:Lp schauder} applies, and we deduce $\partial_t f , \nabla^2_\xi f \in L^4((t,T)\!\times\!\T^3)$---technically, the estimate holds for any time interval $(t',T')\subset\subset (t,T)$, but by resizing the original $(t,T)$, we obtain the result on any such interval. By arguing as per Step 4 of the proof of Lemma \ref{lem:H2 f}, the Gagliardo--Nirenberg inequality implies $\Vert \nabla_{\bxi} f(s) \Vert_{L^8(\T^3)}^8 \leq C ( \Vert \nabla_{\bxi}^2 f(s) \Vert_{L^4(\T^3)}^{4} \Vert f(s) \Vert_{L^\infty(\T^3)}^{4} + {\Vert f(s) \Vert^4_{L^1(\T^3)}} )$, hence $\nabla_\xi f \in L^8((t,T)\!\times\!\T^3)$. 

  \smallskip 

  \noindent 2. \textit{Inductive argument}. We will show that, for all $m \geq 1$ and all $(t,T)\subset \subset (0,\infty)$, it holds $\partial_t f , \nabla^2_\xi f \in L^{2^m}((t,T)\!\times\T^3)$ and $\nabla_\xi f \in L^{2^{m+1}}((t,T)\!\times\!\T^3)$. We have already shown the result for $m=1,2$. Suppose the result is true for $m\in\mathbb{N}$: we study the case $m+1$. Then, by \eqref{eq:rho eq}, it holds $\partial_t \rho - \Delta \rho = \rho \dv \p - \nabla \rho \cdot \p \in L^{2^{m+1}}((t,T)\!\times\!\T^2)$ by the induction hypothesis. Thus, the Calder\'on--Zygmund Theorem implies that (up to resizing the original time interval), it holds $\partial_t \rho , \nabla^2 \rho \in L^{2^{m+1}}((t,T)\!\times\!\T^2)$. Whence, returning to \eqref{eq:recast dt f fess f in all Lp} and using also the boundedness of $f$ from Proposition \ref{prop:f bounded moser}, it holds $\partial_t f - A:\nabla^2_\xi f \in L^{2^{m+1}}((t,T)\!\times\T^3)$. The Schauder Lemma \ref{lem:Lp schauder} implies (up to resizing the original time interval) that $\partial_t f , \nabla^2_\xi f \in L^{2^{m+1}}((t,T)\!\times\T^3)$, and arguing by the Gagliardo--Nirenberg inequality as per the end of Step 1 (\textit{cf.}~Step 4 of Lemma \ref{lem:H2 f}) shows $ \Vert \nabla_{\bxi} f \Vert^{2^{m+2}}_{L^{2^{m+2}}((t,T)\times\T^3)} \leq C ( \Vert \nabla_{\bxi}^2 f \Vert_{L^{2^{m+1}}((t,T)\times\T^3)}^{2^{m+1}} \Vert f \Vert_{L^\infty((t,T)\times\T^3)}^{2^{m+1}} + (T-t)). $ By induction, the result holds for all $m$, and thus the conclusion of the lemma follows by interpolation in Lebesgue spaces. 
\end{proof}

\subsection{Higher regularity bootstrap}\label{sec:bootstrap}

We argue differently to the strategy of \S 7 of \cite{reg2}---here, we shall rely on Schauder estimates in H\"older continuous spaces instead of repeated uses of the Schauder type result in $L^p$ (Lemma \ref{lem:Lp schauder}). Ultimately, the goal will be to argue using H\"older continuity of the right-hand side of \eqref{eq:recast dt f fess f in all Lp} and a bootstrap argument. For this, we need to show that $\Delta \rho$ and $\nabla f$ are continuous, and we will do this by obtaining sufficient integrability of their derivatives and Morrey's embedding.

In what follows, we use the shorthand notation $\dot{f}=\partial_t f, \dot{\rho}=\partial_t \rho, \dot{\p}=\partial_t \p$. Furthermore, we extend $f(t)$ smoothly in $(-\infty,0]$ so that it is defined for a.e.~$t \in \mathbb{R}$. For $h>0$, we define the \emph{difference quotient} $D_h f(t,\xi) := \frac{f(t+h,\xi)-f(t,\xi)}{h}$ for a.e.~$t,\xi$, and similarly $D_h \rho(t,x) := \frac{\rho(t+h,x)-\rho(t,x)}{h}$ and $D_h \p(t,x) := \frac{\p(t+h,x)-\p(t,x)}{h}$; observe that $D_h \rho = \int_\T D_h f \d \theta$ and $D_h \p = \int_\T D_h f \e(\theta) \d \theta$ by linearity. We write $\tau_h \rho(t,x) = \rho(t+h,x)$, and observe that for all $h \in (0,1)$ and $0<t<T<\infty$, it holds $1 \geq 1-\tau_h \rho \geq c_{t,T+1}$ by \eqref{eq:ctT def}.

\begin{lemma}[$\dot{\rho}$ equation]\label{lem:rho dot H1}
    For all $0<t<T<\infty$, it holds $\dot{\rho}\in L^\infty((t,T);L^2(\T^2))$, $\nabla \dot{\rho} \in L^2((t,T)\!\times\!\T^2)$, $\partial_t \dot{\rho} \in L^2((t,T);H^{-1}(\T^2))$, and 
    \begin{equation}\label{eq:dot rho eq}
        \partial_t \dot{\rho} + \dv((1-\rho)\dot{\p} - \p \dot\rho ) = \Delta \dot{\rho} \qquad \text{weakly in } L^2((t,T);H^1(\T^2)). 
    \end{equation}
\end{lemma}

\begin{proof}
    Firstly, observe that $\dot{\rho},\dot{\p} \in L^p((t,T)\!\times\!\T^2)$ for all $p \in [1,\infty)$ by Proposition \ref{prop:higher integ dt f and hess f}, so that \eqref{eq:dot rho eq} has meaning in the sense of distributions---we need to show that it follows from \eqref{eq:rho eq}, and we do so by means of a difference quotient argument. From \eqref{eq:rho eq} and linearity:
    \begin{equation*}
        \partial_t D_h \rho + \dv( (1-\tau_h\rho)D_h \p - \p D_h \rho ) = \Delta D_h \rho \qquad \text{weakly in }L^2(0,\infty;H^1(\T^2));
    \end{equation*}
    recall $\tau_h \rho(t) = \rho(t\!+\!h)$. We test this with $D_h \rho$, whence the Lions--Magenes Lemma gives 
    \begin{equation*}
       \begin{aligned} \frac{1}{2}\frac{\der}{\der t} \Vert D_h \rho \Vert^2_{L^2(\T^2)} + \Vert \nabla D_h \rho \Vert^2_{L^2(\T^2)} = \int_{\T^2} (1-\tau_h \rho) \nabla D_h \rho \!\cdot \!D_h \p \d x - \int_{\T^2} D_h \rho \p \!\cdot\! \nabla D_h \rho \d x .
       \end{aligned}
    \end{equation*}
    Young's inequality, the boundedness of $\p$, and integrating in time imply  $\Vert D_h \rho(s) \Vert^2_{L^2(\T^2)} + \Vert \nabla D_h \rho \Vert^2_{L^2((t,s)\times\T^2)} \leq  \Vert D_h \rho(t) \Vert^2_{L^2(\T^2)}  + 2\Vert D_h \p \Vert^2_{L^2((t,T)\times\T^2)} + 2\Vert D_h \rho \Vert^2_{L^2((t,T)\times\T^2)}$ for a.e.~$t<s\leq T$. Standard properties of difference quotients (\textit{cf.~e.g.~}\cite[\S 5.8.2]{Evans}) yield 
    \begin{equation*}
       \begin{aligned} \Vert D_h \rho \Vert^2_{L^\infty((t,T);L^2(\T^2))} \!+\! \Vert \nabla D_h \rho & \Vert^2_{L^2((t,T)\times\T^2)} \\ 
       &\leq C \Big( \Vert \dot\rho(t) \Vert^2_{L^2(\T^2)}  \!+ \!\Vert \dot \p \Vert^2_{L^2((t,T)\times\T^2)}\! +\! \Vert \dot \rho \Vert^2_{L^2((t,T)\times\T^2)} \Big) , 
       \end{aligned}
    \end{equation*}
    where the constant $C$ is independent of $h$. Standard results on difference quotients imply that (by letting $h \to 0^+$) it holds $\dot{\rho} \in L^\infty((t,T)\!\times\!\T^2), \nabla \dot{\rho} \in L^2((t,T)\!\times\!\T^2)$ as claimed. In turn, having established this regularity, we may rigorously pass a full time derivative $\partial_t $ onto \eqref{eq:rho eq}, and we deduce \eqref{eq:dot rho eq}. The boundedness $\partial_t \dot{\rho} \in L^2((t,T);H^{-1}(\T^2))$ follows from this equation and the boundedness of the other terms. 
\end{proof}

We need to upgrade the regularity of $\dot{\rho}$ further, as doing so will enable us to deduce continuity of $\Delta \rho$ via \eqref{eq:rho eq}, and thus show that the right-hand side of \eqref{eq:recast dt f fess f in all Lp} is H\"older continuous, whence we can initiate the bootstrap argument in H\"older spaces. In order to do this, we argue using the equation for $\dot{f}$, which we derive in the next lemma. 

\begin{lemma}[$\dot{f}$ equation]\label{lem:H1 dot f}
    For all $0<t<T<\infty$, it holds $\dot{f} \in L^\infty((t,T);L^2(\T^3)), \nabla_\xi \dot{f} \in L^2((t,T)\!\times\!\T^3)$, $\partial_t \dot{f} \in L^2((t,T);H^{-1}(\T^3))$, and moreover, weakly in $L^2((t,T);H^1(\T^3))$: 
    \begin{equation}\label{eq:dot f eq}
        \partial_t \dot{f} + \dv(((1-\rho)\dot{f} - f \dot{\rho}  )\e(\theta)) = \dv((1-\rho)\nabla\dot{f} + \dot{f} \nabla \rho - \dot{\rho} \nabla f + f\nabla \dot{\rho}) + \partial^2_\theta \dot{f} . 
    \end{equation}
\end{lemma}

\begin{proof}
    As per the proof of Lemma \ref{lem:rho dot H1}, we obtain directly from \eqref{eq:main eqn} and linearity: 
    \begin{equation*}
 \begin{aligned} \partial_t D_h f +   \dv\big[ & \big(D_h f (1-\tau_h \rho) - f D_h \rho \big)\e(\theta) \big] \\ 
 &= \dv\big[ (1-\tau_h \rho) \nabla D_h f - D_h \rho \nabla f + \tau_h f \nabla D_h \rho + D_h f \nabla \rho \big] + \partial^2_\theta D_h f;
 \end{aligned}
\end{equation*}
recall $\tau_h f = f(t\!+\!h)$. Testing with $D_h f$, we get (using the Lions--Magenes Lemma) 
\begin{equation*}
   \begin{aligned} \frac{1}{2}&\frac{\der}{\der t} \Vert D_h f \Vert^2_{L^2(\T^3)} + c_{t,T+1} \Vert \nabla_\xi D_h f \Vert^2_{L^2(\T^3)} \\ 
   &\leq \!\Big( \!\Vert D_h f \Vert_{L^2} \! + \! \Vert f D_h \rho \Vert_{L^2} \! + \! \Vert D_h \rho \nabla f \Vert_{L^2} \!+\! \Vert \tau_h f   \nabla D_h \rho \Vert_{L^2} \!+\! \Vert D_h f \nabla \rho \Vert_{L^2}\!\Big)\! \Vert \nabla D_h f \Vert_{L^2}, 
   \end{aligned}
\end{equation*}
whence Young's inequality and standard properties of difference quotients yield 
\begin{equation*}
   \begin{aligned} \frac{1}{2}&\frac{\der}{\der t} \Vert D_h f \Vert^2_{L^2(\T^3)} + \frac{c_{t,T+1}}{2} \Vert \nabla_\xi D_h f \Vert^2_{L^2(\T^3)} \\ 
   &\leq \! C c_{t,T+1}^{-1} \Big( \!\Vert \dot{f} \Vert_{L^2}^2 \! + \! \Vert f \Vert^2_{L^\infty} \Vert \dot{\rho} \Vert_{L^2}^2 \! + \! \Vert \dot \rho \Vert_{L^4}^2 \Vert \nabla f \Vert_{L^4}^2 \!+\! \Vert \tau_h f \Vert^2_{L^\infty} \Vert  \nabla \dot{\rho} \Vert_{L^2}^2 \!+\! \Vert \dot f \Vert_{L^4} \Vert \nabla \rho \Vert_{L^4}^2 \!\Big)
   \end{aligned}
\end{equation*}
where $C$ is independent of $s,h$. Using the boundedness of the terms on the right-hand side (Propositions \ref{prop:f bounded moser}, \ref{prop:higher integ dt f and hess f}, Lemma \ref{lem:rho dot H1}) $\Vert D_h f \Vert^2_{L^\infty((t,T);L^2(\T^3))} +\Vert \nabla_\xi D_h f \Vert^2_{L^2((t,T)\times\T^3)} \leq C_{t,T}$ for some constant depending only on $t,T,c_{t,T+1}$. As such, $\dot{f} \in L^\infty((t,T);L^2(\T^3)) \cap L^2((t,T);H^1(\T^3))$. This regularity is sufficient to pass a time derivative (in the sense of distributions) in \eqref{eq:main eqn}, and we deduce \eqref{eq:dot f eq}
\end{proof}

Lemma \ref{lem:H1 dot f} implies that $\dot{\p}$ is endowed with the same regularity as $\dot{f}$, and 
\begin{equation}\label{eq:dot p eq}
    \partial_t \dot\p +\dv((1-\rho)\dot\P - \P \dot{\rho}) = \dv((1-\rho)\nabla\dot\p + \dot\p\otimes\nabla\rho -\dot{\rho}\nabla\p +\p\otimes\nabla\dot{\rho}) - \dot\p. 
\end{equation}
Returning to \eqref{eq:dot rho eq}, we get $\partial_t \dot{\rho} - \Delta \dot{\rho} = -\dv((1-\rho)\dot{\p} - \p \dot{\rho}) \in L^2_{t,\loc}L^2_x$, and the Calder\'on--Zygmund Theorem implies $\partial_t \dot{\rho},\nabla^2\dot{\rho} \in L^2((t,T)\!\times\!\T^2)$ for all $0<t<T<\infty$, from which the usual $H^2$ estimate gives $\nabla \dot{\rho} \in L^\infty((t,T);L^2(\T^2))$ and by interpolation (Lemma \ref{lem:dibenedetto classic}) implies $\nabla \dot{\rho} \in L^4((t,T)\!\times\!\T^2)$. We have shown the following. 

\begin{cor}\label{cor:dot rho hess}
    For all $0\!<\!t\!<T\!<\!\infty$, $\partial_t \dot{\rho},\nabla^2\dot{\rho} \in L^2((t,T)\!\times\!\T^2)$, $ \nabla \dot{\rho} \in L^4((t,T)\!\times\!\T^2)$. 
\end{cor}

We need to upgrade the integrability of $\partial_t \dot \rho$ a little more: indeed, in order to use Morrey's embedding and deduce that $\dot\rho$ is H\"older continuous, we need $\dot\rho \in W^{1,p}_{t,\loc}W^{1,p}_x$ for $p>3$, with $3$ being the dimension of $(t,x)$. To this end, we perform the $H^2$ estimate on $\dot\p$ from \eqref{eq:dot p eq}.

\begin{lemma}[$H^2$ estimate for $\dot\p$]\label{lem:H2 dot p}
    For all $0<t<T<\infty$, it holds $$\nabla\dot\p \in L^\infty((t,T);L^2(\T^2)) \cap L^2((t,T);H^1(\T^2)), \qquad \nabla\dot\p \in L^4((t,T)\!\times\!\T^2).$$ 
\end{lemma}
\begin{proof}
    We argue analogously to Proposition \ref{prop:galerkin p}. Following this strategy to the letter, we obtain the existence of $ \dot{\p}^n \in AC([t,T];\mathrm{span}\{\varphi_j\}_{j=1}^n)$ ($\varphi_j$ eigenfunctions of $\Delta_x$) solving 
    \begin{equation}\label{eq:galerkin approx dot p}
    \left\lbrace \begin{aligned} 
        &\partial_t \dot{\p}^n \!-\! \dv((1\!-\!\rho)\nabla \dot{\p}^n \!+\! \dot\p^n \otimes \nabla \rho) \!+\! \dot\p^n =  \dv( \P \dot\rho \!-\!(1\!-\!\rho)\dot{\P} \!+\! \p \otimes \nabla \dot\rho  \!-\!   \dot\rho \nabla \p ), \\ 
        &\dot\p^n(t) = \sum_{j=1}^n \langle \dot\p(t),\varphi_j \rangle \varphi_j. 
    \end{aligned}\right. 
\end{equation}
Testing the above with $\dot{\p}^n$ and Young's inequality we get (a.e.~$t < s \leq T$)  
$$\begin{aligned}\Vert \dot\p^n(s) \Vert_{L^2 } + \Vert \nabla \dot\p^n \Vert_{L^2((t,s)\times\T^2)} \leq C_{t,T}\Big( & \Vert \dot\p (t) \Vert_{L^2}^2 +  \Vert \dot\p^n |\nabla \rho| \Vert^2_{L^2} + \Vert \dot{\rho} \Vert^2_{L^2}  \\ 
&+ \Vert \dot\P \Vert^2_{L^2} + \Vert \p |\nabla \dot\rho | \Vert^2_{L^2} + \Vert \dot\rho |\nabla \p | \Vert^2_{L^2} \Big) \end{aligned}$$
for $C_{t,T}$ depending only on $t,T,c_{t,T}$; we use the lower bound on $1\!-\!\rho$ from \eqref{eq:ctT def}. Hence, 
\begin{equation*}
    \begin{aligned}\Vert \dot\p^n(s) \Vert_{L^2(\T^2)} + \Vert \nabla \dot\p^n \Vert_{L^2((t,s)\times\T^2)} \leq C_{t,T}\Big(1 + \int_t^s \Vert \nabla \rho(s') \Vert_{L^\infty(\T^2)}^2 \Vert \dot\p^n(s') \Vert^2_{L^2(\T^2)} \d s' \Big),  \end{aligned}
\end{equation*}
and we conclude that $\Vert \dot\p^n \Vert_{L^\infty((t,T);L^2(\T^2))} + \Vert \nabla \dot\p^n \Vert_{L^2((t,T)\times\T^2)} \leq C_{t,T}$ using Gr\"onwall's Lemma as per the proof of Proposition \ref{prop:galerkin p}. Furthermore, given that $\dot\p^n(t) \in C^\infty(\T^2)$ for each $n$ (being a linear combination of eigenfunctions of $\Delta_x$), we rigorously pass to the non-divergence form of \eqref{eq:galerkin approx dot p}. Testing with $\Delta \dot\p^n$, we obtain the uniform bound $\sup_n (\Vert \nabla \dot\p^n \Vert_{L^\infty((t,T);L^2(\T^2))} + \Vert \Delta \dot\p^n \Vert_{L^2((t,T)\times\T^2)} ) \leq C_{t,T}$; we skip the details since they are analogous (and simpler) to those in Proposition \ref{prop:galerkin p}. An argument akin to Step 3 in Proposition \ref{prop:galerkin p} shows $\dot\p^n \rightharpoonup \dot\p$ weakly in $L^2((t,T);H^2(\T^2))$, and the result follows. 
\end{proof}

Using the result of Lemma \ref{lem:H2 dot p} and \eqref{eq:dot rho eq}, we get $\partial_t\dot\rho - \Delta\dot\rho \in L^4((t,T)\!\times\!\T^2)$, whence the Calder\'on--Zygmund Theorem implies $\partial_t \dot\rho , \nabla^2 \dot\rho \in L^4_{t,\loc}L^4_x$ and the following corollary. 

\begin{cor}\label{cor:continuity of laplace rho}
    For all $0<t<T$, $\partial_t \dot\rho , \nabla^2 \dot\rho \in L^4((t,T)\!\times\!\T^2)$, and $\dot{\rho} , \Delta \rho \in C^{\nicefrac{1}{4}}([t,T]\times\T^2)$ as well as $\nabla \dot{\rho} \in L^8((t,T)\!\times\!\T^2)$. 
\end{cor}
\begin{proof}
    From $\partial_t \dot\rho , \nabla^2 \dot\rho \in L^4_{t,\loc}L^4_x$, we get $\dot\rho \in W^{1,4}([t,T]\!\times\!\T^2)$ for all $0<t<T<\infty$, whence Morrey's embedding implies $\dot\rho \in C^{\nicefrac{1}{4}}([t,T]\times\T^3)$. Expanding $\dv((1-\rho)\p) = -\p \cdot \nabla \rho + (1-\rho) \dv \p$, we get $\dv((1-\rho)\p) \in W^{1,4}_{t,\loc}W^{1,4}_x \hookrightarrow C^{\nicefrac{1}{4}}_{t,\loc,x}$ using Proposition \ref{prop:higher integ dt f and hess f}, Corollary \ref{cor:dot rho hess}, and Lemma \ref{lem:H2 dot p}. By \eqref{eq:rho eq}, we deduce  $\Delta \rho = \partial_t \rho + \dv((1-\rho)\p) \in C^{\nicefrac{1}{4}}([t,T]\!\times\!\T^2)$. Finally, since $C^{\nicefrac{1}{4}}([t,T]\!\times\! \T^2) \hookrightarrow L^\infty([t,T]\!\times\! \T^2)$, Gagliardo--Nirenberg's inequality implies 
    \begin{equation*}
        \Vert \nabla\dot{\rho}(s) \Vert_{L^8(\T^2)} \leq C \Big(   \Vert \nabla^2 \dot{\rho} (s) \Vert^{\nicefrac{1}{2}}_{L^4(\T^2)} \Vert \dot{\rho}(s) \Vert^{\nicefrac{1}{2}}_{L^\infty(\T^2)} + \Vert \dot{\rho}(s) \Vert_{L^1(\T^2)} \Big) , 
    \end{equation*}
    where $C$ only depends on $\T^2$, and we then deduce $\nabla \dot{\rho} \in L^8_{t,\loc}L^8_x$.
\end{proof}

In order to initiate the bootstrap in H\"older spaces, it remains to show that $\nabla f$ is H\"older continuous; once this is shown, the entire right-hand side of \eqref{eq:recast dt f fess f in all Lp} is H\"older continuous. 

\begin{lemma}[H\"older continuity of $\nabla f$]\label{lem:nabla f continuous}
    For all $0<t<T<\infty$, $\nabla_\xi f \in C^{\nicefrac{1}{5}}([t,T]\!\times\!\T^3)$. 
\end{lemma}

\begin{proof}
    The results of Lemma \ref{lem:H1 dot f} and Corollary \ref{cor:continuity of laplace rho} imply, by returning to \eqref{eq:dot f eq}, that 
    \begin{equation}\label{eq:dot f eqn rewritten}
        \partial_t \dot{f} - \dv((1-\rho)\nabla \dot{f} + \dot{f} \nabla \rho) - \partial^2_\theta f = F, 
    \end{equation}
    with $F \in L^2_{t,\loc}L^2_\xi$. By replicating the Galerkin approach of Lemma \ref{lem:H2 f}, we perform the $H^2$ estimate for $\dot{f}$; we omit the details for concision, which follow the formal computation associated with testing \eqref{eq:dot f eqn rewritten} with $\Delta \dot{f}$ after passing to non-divergence form. We get $\partial_t \dot{f} , \nabla^2_\xi \dot{f} \in L^2_{t,\loc}L^2_\xi$, as well as $\nabla_\xi \dot{f} \in L^\infty_{t,\loc}L^2_\xi \cap L^2_{t,\loc}H^1_\xi$. By Lemma \ref{lem:dibenedetto classic}, $\nabla_\xi \dot{f} \in L^{\nicefrac{10}{3}}_{t,\loc}L^{\nicefrac{10}{3}}_\xi$. We now rigorously pass to the non-divergence form of \eqref{eq:dot f eq} and observe that $\partial_t \dot{f} - A:\nabla_\xi^2 \dot{f} \in L^{\nicefrac{10}{3}}_{t,\loc}L^{\nicefrac{10}{3}}_\xi$ where $A$ is as per \eqref{eq:A def for non div f}. Applying the Schauder result in Sobolev spaces (Lemma \ref{lem:Lp schauder}) implies $\partial_t \dot{f},\nabla^2_\xi \dot{f} \in L^{\nicefrac{10}{3}}_{t,\loc}L^{\nicefrac{10}{3}}_\xi$. By Gagliardo--Nirenberg's inequality, we get 
    \begin{equation*}
        \begin{aligned}
            \Vert \nabla_\xi \dot{f}(s) \Vert_{L^{\nicefrac{50}{9}}(\T^3)} \leq C \Big( \Vert \nabla^2_\xi \dot{f}(s) \Vert^{\nicefrac{3}{5}}_{L^{\nicefrac{10}{3}}(\T^3)} \Vert \nabla_\xi \dot{f}(s) \Vert^{\nicefrac{2}{5}}_{L^2(\T^3)} + \Vert \nabla_\xi \dot{f}(s) \Vert_{L^2(\T^3)} \Big), 
        \end{aligned}
    \end{equation*}
    with $C$ depending only on $\T^3$. Raising to the power $50/9$ and integrating in time, we get $\nabla_\xi \dot{f} \in L^{\nicefrac{50}{9}}((t,T)\times\T^3)$ for all $0<t<T$. Since $50/9>5$, we deduce also by Proposition \ref{prop:higher integ dt f and hess f} that $\nabla_\xi f \in W^{1,5}_{t,\loc}W^{1,5}_\xi$, whence by Morrey's embedding $\nabla_\xi f \in C^{\nicefrac{1}{5}}([t,T]\!\times\!\T^3)$. 
\end{proof}

We are ready to give the proof of Theorem \ref{thm:reg}, which concludes this work. 

\begin{proof}[Proof of Theorem \ref{thm:reg}]
    The lower bound \eqref{eq:infinite speed of propagation in thm statement} was proved in Proposition \ref{prop:lower bound rho improved}: it suffices to show the interior smoothness of weak solutions. The proof is divided into several steps: our underlying strategy is to repeatedly apply additional $\xi$-derivatives to \eqref{eq:main eqn} and $x$-derivatives to \eqref{eq:rho eq} and deduce smoothness using the Schauder result in H\"older spaces (Lemma \ref{lem:Schauder Holder}).

\smallskip 

\noindent 1. \textit{Improved regularity of $\rho$}. Our aim in this step is to prove the following: for all $0<t<T$, there exists $\beta>0$ depending on $t,T,c_{t,T}$ such that 
   \begin{equation}\label{eq:Drho is Holder}
   D \rho \in C^{1+\nicefrac{\beta}{2},2+\beta}([t,T]\times\T^2), \quad \text{\textit{i.e.}~}\quad\partial_t D\rho, D^3 \rho \in C^{\nicefrac{\beta}{2},\beta}([t,T]\times\T^2). 
   \end{equation}

Recall from Corollary \ref{cor:continuity of laplace rho} and Lemma \ref{lem:nabla f continuous} that $\Delta \rho , \nabla f \in C^{\nicefrac{1}{5}}([t,T]\!\times\!\T^3)$ for all $0<t<T$. As such, the entire right-hand side of \eqref{eq:recast dt f fess f in all Lp} is H\"older continuous, as is the matrix $A$ from \eqref{eq:A def for non div f}; in fact for all $0<t<T$ they belong to $C^{\nicefrac{\alpha}{2},\alpha}([t,T]\!\times\!\T^3)$ for some $\alpha \in (0,1)$. As such (by modifying the size of the time interval---we omit the details for concision), the Schauder result of Lemma \ref{lem:Schauder Holder} implies 
\begin{equation}\label{eq:after one schauder f}
f \in C^{1+\nicefrac{\alpha}{2},2+\alpha}([t,T]\times\T^3) \qquad \forall 0<t<T;
\end{equation}
note that this level of regularity is already enough to deduce that $f$ is a classical solution of \eqref{eq:main eqn}. We return to the equation \eqref{eq:rho eq} for $\rho$. We mollify the entire equation (in $x$ only) and get $\partial_t \rho^\varepsilon - \Delta \rho^\varepsilon = F^\varepsilon := -\dv((1-\rho)\p)*\eta_\varepsilon$,  where $\rho^\varepsilon = \rho*\eta_\varepsilon$ for $\eta_\varepsilon = \varepsilon^{-2}\eta(\cdot/\varepsilon)$ with $\eta$ the usual Friedrichs bump function in $\mathbb{R}^2$. By applying an $x_j$ derivative, we get 
    \begin{equation*}
        \partial_t \rho^\varepsilon_{x_j} - \Delta \rho^\varepsilon_{x_j} = F^\varepsilon_{x_j} , 
    \end{equation*}
   and note that $F^\varepsilon_{x_j} = F_{x_j}*\eta_\varepsilon = (\p_{x_j} \!\cdot\! \nabla \rho + \p \!\cdot\! \nabla \rho_{x_j} + \rho_{x_j} \nabla \p - (1-\rho) \dv \p_{x_j})*\eta_\varepsilon$. The regularity of $f$ from \eqref{eq:after one schauder f} implies that, for all $0<t<T$, there exists $\beta \in (0,1)$ depending on $t,T,c_{t,T}$ such that $F_{x_j} \in C^{\nicefrac{\beta}{2},\beta}([t,T]\times\T^2)$. Standard properties of the mollifier sequence then imply that $\Vert F^\varepsilon_{x_j} \Vert_{C^{\nicefrac{\beta}{2},\beta}([t,T]\times\T^2)} \leq C \Vert F_{x_j} \Vert_{C^{\nicefrac{\beta}{2},\beta}([t,T]\times\T^2)} $ where the constant $C$ is independent of $t,T,\varepsilon$. Hence, another application of Lemma \ref{lem:Schauder Holder} implies $\Vert \rho^\varepsilon_{x_j} \Vert_{C^{1+\nicefrac{\beta}{2},2+\beta}([t',T']\times\T^2)} \leq C(\Vert F_{x_j} \Vert_{C^{\nicefrac{\beta}{2},\beta}([t,T]\times\T^2)} + \Vert \nabla \rho \Vert_{L^2([t,T]\times\T^2)})$ for $[t',T']\subset\subset [t,T]$ with $C$ depending on $t,t',T,T'$. In turn, an application of the Ascoli--Arzel\`a Theorem implies the strong convergence of $\rho^\varepsilon_{x_j} \to \rho_{x_j}$ in the norm of $C^{1+\nicefrac{\tilde\beta}{2},2+\tilde\beta}([t',T']\times\T^2)$ for all $0<\tilde\beta < \beta$, and \eqref{eq:Drho is Holder} follows (for a different $\beta$ depending on $t,T,c_{t,T}$).

\smallskip 

\noindent 2. \textit{Improved regularity for $f$}. The regularity \eqref{eq:after one schauder f} implies that we may rigorously pass a $\xi_j$ derivative \emph{in the divergence form of \eqref{eq:main eqn}} to deduce that 
\begin{equation}\label{eq:one space deriv f xi j}
   \begin{aligned}
       \partial_t f_{\xi_j} + \dv((1-\rho)f_{\xi_j}\e(\theta) &- \rho_{\xi_j} f \e(\theta) + (1-\rho)f \e_{\xi_j}(\theta)) \\ 
       &= \dv((1-\rho)\nabla f_{\xi_j} + f_{\xi_j} \nabla \rho - \rho_{\xi_j} \nabla f  + f \nabla \rho_{\xi_j}) + \partial^2_\theta f_{\xi_j}
   \end{aligned}
\end{equation}
holds (in divergence form) in the sense dual to $L^1((t,T);W^{1,1}(\T^3))$. Our objective is then to show that $\nabla^2_\xi f_{\xi_j} \in L^2_{t,\loc}L^2_\xi$ so that we may pass to the non-divergence form of this equation. To do this rigorously, we perform the same Galerkin method as in Lemma \ref{lem:H2 f} to deduce $f_{\xi_j} \in L^2((t,T);H^2(\T^3))$; we do not write the details in full for clarity of presentation. The sought boundedness can be seen from the formal computation (testing \eqref{eq:one space deriv f xi j} with $\Delta f_{\xi_j}$ after having passed to non-divergence form) 
\begin{equation*}
    \begin{aligned}
        \frac{1}{2}\frac{\der}{\der t}\Vert \nabla f_{\xi_j} \Vert^2_{L^2} + c_{t,T}\Vert \Delta f_{\xi_j} \Vert^2_{L^2} & \leq   \Vert \Delta f_{\xi_j}\Vert_{L^2} \Big( \! \Vert f_{\xi_j} \Delta \rho \Vert_{L^2} \!+\! \Vert \rho_{\xi_j} \Delta f \Vert_{L^2} \!+\! \Vert f \Delta \rho_{\xi_j} \Vert_{L^2} \!+\! \Vert \nabla f_{\xi_j} \Vert_{L^2} \\ 
        &+\! \Vert f_{\xi_j} \nabla \rho \Vert_{L^2} \!+\! \Vert f \nabla \rho_{\xi_j} \Vert \!+\! \Vert \rho_{\xi_j} \nabla f \Vert_{L^2} \!+\! \Vert \nabla f \Vert_{L^2} \!+\! \Vert f \nabla \rho \Vert_{L^2}   \! \Big)    , 
    \end{aligned}
\end{equation*}
where we note from \eqref{eq:after one schauder f}--\eqref{eq:Drho is Holder} that all terms on the right-hand side are bounded in $L^\infty$. Applying Young's inequality yields $\Delta f_{\xi_j} \in L^2((t,T);L^2(\T^3))$, and repeating the same argument by testing with $\partial^2_\theta f_{\xi_j}$ yields $\nabla^2_\xi f_{\xi_j} \in L^2((t,T);L^2(\T^3))$.

   Next, we show that: for all $0<t<T$ there exists $\beta>0$ depending on $t,T$ such that 
   \begin{equation}\label{eq:schauder level 2 f}
      D_\xi f \in C^{1+\nicefrac{\beta}{2},2+\beta}([t,T]\times\T^3). 
   \end{equation}
   Having shown $\nabla^2_\xi f_{\xi_j} \in L^2_{t,\loc}L^2_\xi$, we return to \eqref{eq:one space deriv f xi j} and rigorously pass to the non-divergence form of the equation to get 
   \begin{equation*}
       \partial_t f_{\xi_j} - A:\nabla^2_\xi f_{\xi_j} = A_{\xi_j}:\nabla_\xi^2 f + G_{\xi_j}, 
   \end{equation*}
   where $G=f\Delta \rho + f\e(\theta)\cdot\nabla \rho - (1-\rho)\nabla f \cdot \e(\theta)$ is such that $D_\xi G \in C^{\beta}([t,T]\times\T^3)$ for some $\beta>0$ (depending on $t,T,c_{t,T}$) by \eqref{eq:Drho is Holder}--\eqref{eq:after one schauder f}; note that the equation holds in the sense of $L^2$ functions. Furthermore, the results \eqref{eq:after one schauder f} and \eqref{eq:Drho is Holder} imply that the entire right-hand side belongs to $C^{\beta}([t,T]\times\T^3)$ for some $\beta>0$ depending on $t,T,c_{t,T}$. In turn, we directly apply the Schauder result of Lemma \ref{lem:Schauder Holder} and deduce that $f_{\xi_j} \in C^{1+\nicefrac{\beta}{2},2+\beta}([t',T']\times\T^3)$ for the same $\beta$ on any $[t',T'] \subset\subset [t,T]$ compactly contained therein, and \eqref{eq:schauder level 2 f} follows.

    \smallskip 

    \noindent 3. \textit{Conclusion}. We iterate the previous two steps: having shown that $D^m_\xi f$ is H\"older continuous, we apply a further $\partial_{x_j}$ to \eqref{eq:rho eq} to deduce H\"older continuity of $D^{m+1} \rho$, and then returning to \eqref{eq:main eqn} and applying an additional $\partial_{\xi_j}$ we deduce $D^{m+1}_\xi f$ is H\"older continuous. Using the equation to exchange $t$-derivatives for $\xi$-derivatives, we also obtain the H\"older continuity of all time derivatives, and mixed $t,\xi$ derivatives. The proof is complete. 
\end{proof}

\appendix
	
	\section{Proofs of Technical Lemmas}\label{sec:appendix}

    \subsection{Proof of convergence to initial data for $\rho$}\label{sec:appendix rho initial data}

We prove \eqref{eq:conv rho to initial data} directly from the definition of weak solution and the regularity of $\rho$ from Definition \ref{def:function space}.

\begin{proof}[Proof of \eqref{eq:conv rho to initial data}]
   Inserting test functions $\varphi \in W^{1,6}(\T^2)$ independent of $t,\theta$ in \eqref{eq:weak sol}, we obtain 
\begin{equation}\label{eq:for rho initial data}
    \begin{aligned}
        \int_{\T^2} \big( \rho(t,x) - \rho_0(x) \big) \varphi(x) \d x = \int_0^t \int_{\T^2} (1-\rho) \p \cdot \nabla \varphi \d x \d s - \int_0^t \int_{\T^2} \nabla \rho \cdot \nabla \varphi \d x \d s ; 
    \end{aligned}
\end{equation}
and, by a standard density argument, the above holds for all $\varphi \in H^1(\T^2)$. In turn, by bounding the relevant terms and arguing as per \eqref{eq:as per remark attainment old} in Remark \ref{rem:attainment old}, we deduce 
\begin{equation}\label{eq:rho attainment old}
    \lim_{t \to 0^+} \Vert \rho(t) - \rho_0 \Vert_{H^{-1}(\T^2)} = 0. 
\end{equation}
Selecting $t$ a Lebesgue point and inserting $\varphi=\rho(t) \in H^1(\T^2)$ in \eqref{eq:for rho initial data}, we get (using H\"older's inequality) $\limsup_{t\to0 } \Vert \rho(t) \Vert^2_{L^2(\T^2)} \leq \Vert \rho_0 \Vert_{L^2(\T^2)} \limsup_{t \to0} \Vert \rho(t) \Vert_{L^2(\T^2)}$; this $\limsup$ is finite since $0\leq \rho \leq 1$ and $\T^2$ is bounded. Thus, 
\begin{equation}\label{eq:comparing norms rho t attainment}
   \limsup_{t \to 0^+} \Vert \rho(t) \Vert_{L^2(\T^2)} \leq \Vert \rho_0 \Vert_{L^2(\T^2)}. 
\end{equation}
Meanwhile, using $0\leq \rho \leq 1$ we deduce that $\rho(t) \rightharpoonup \rho_0$ weakly in $L^2(\T^2)$; here we also use \eqref{eq:rho attainment old} to identify the limit as $\rho_0$. The weak lower semicontinuity of the norm then implies 
\begin{equation*}
    \Vert \rho_0 \Vert_{L^2(\T^2)} \leq \liminf_{t \to 0^+} \Vert \rho(t) \Vert_{L^2(\T^2)} \leq \limsup_{t \to 0^+} \Vert \rho(t) \Vert_{L^2(\T^2)}  \leq  \Vert \rho_0 \Vert_{L^2(\T^2)}
\end{equation*}
by \eqref{eq:comparing norms rho t attainment}. Thus $\Vert \rho(t) \Vert_{L^2} \to \Vert \rho_0 \Vert_{L^2}$ and $\rho(t) \rightharpoonup \rho_0$ in $L^2$, whence \eqref{eq:conv rho to initial data} follows. 
\end{proof}

\subsection{Proof of Lemma \ref{lem:tech iteration lemma}}

Our proof is similar to that of \cite[Lemma 3.2]{KZ}, but simpler, as our setting is less general. 

    \begin{proof}[Proof of Lemma \ref{lem:tech iteration lemma}]
       Multiply \eqref{eq:for tech iter lemma} by the integrating factor $e^t$, and get (for all $0\leq s \leq t$)
        \begin{equation*}
            \begin{aligned}
                A_k(s) \leq e^{-s} A_k(0) + e^{-s} C_*^k \int_0^s e^\tau A_{k-1}(\tau)^2 \d \tau . 
            \end{aligned}
        \end{equation*}
       By the assumptions on the initial values, we deduce (by bluntly bounding several terms), $A_k(s) \leq M_1^{n_k} + C_*^k \Vert A_{k-1} \Vert_{L^\infty([0,t])}^2 (1-e^{-s} )$, and thus 
        \begin{equation}\label{eq:for the recurrence proof}
            \begin{aligned}
               \Vert A_k \Vert_{L^\infty([0,t])} \leq M_1^{n_k} + C_*^k \Vert A_{k-1} \Vert_{L^\infty([0,t])}^2. 
            \end{aligned}
        \end{equation}

        \smallskip 

        \noindent 1. \textit{Inductive estimate}. Our objective in this step is to prove by induction that  
\begin{equation}\label{eq:induction hypothesis for realzzz}
            \begin{aligned}
               \Vert A_k \Vert_{L^\infty([0,t])} \leq C_*^{\alpha_k} (M_1+M_2)^{n_k} 
            \end{aligned}
        \end{equation}
   for all integer $k\geq 0$,     where $\alpha_k$ satisfies the recurrence relation 
        \begin{equation}\label{eq:alphak recurrence rel}
            \alpha_k = 2\alpha_{k-1}+k+1 \qquad (\alpha_0 := 1). 
        \end{equation}
Indeed, for $k=0$, we simply have $\Vert A_0 \Vert_{L^\infty([0,t])} \leq M_2$ whence \eqref{eq:induction hypothesis for realzzz} is trivially satisfied. Then, suppose \eqref{eq:induction hypothesis for realzzz} holds for the case $k-1$: we show that it holds for $k$. Indeed, by \eqref{eq:for the recurrence proof} and using $2n_{k-1}=n_k$, it holds $\Vert A_k \Vert_{L^\infty([0,t])} \leq M_1^{n_k} + C_*^{k+2\alpha_{k-1} }(M_1+M_2)^{n_{k}} \leq (1+C_*^{k+2\alpha_{k-1} })(M_1+M_2)^{n_{k}}$, and thus 
\begin{equation*}
            \begin{aligned}
               \Vert A_k \Vert_{L^\infty([0,t])} \leq 2 C_*^{k+2\alpha_{k-1} }(M_1+M_2)^{n_{k}} \leq {C_*^{1+k+2\alpha_{k-1} }}(M_1+M_2)^{n_{k}} , 
            \end{aligned}
        \end{equation*}
using $2\leq C_*$, and \eqref{eq:induction hypothesis for realzzz} follows directly from the definition \eqref{eq:alphak recurrence rel} of $\alpha_k$. 

\smallskip 

\noindent 2. \textit{Uniform boundedness}. We have $ \Vert A_k \Vert_{L^\infty([0,t])}^{1/n_k} \leq C_*^{\alpha_{k}/n_k}(M_1+M_2)$ and thus it suffices to show that $\alpha_k/n_k$ is bounded. Dividing \eqref{eq:alphak recurrence rel} by $n_k=2^k$ and rewriting $n_k=2n_{k-1}$, 
\begin{equation*}
   \frac{\alpha_k}{n_k} = \frac{\alpha_{k-1}}{n_{k-1}} + \frac{k+1}{2^k} 
\end{equation*}
for all integer $k \geq 1$, from which we deduce (using $\alpha_0/n_0=1$) $\frac{\alpha_k}{n_k} = 1 + \sum_{j=1}^k (j+1)2^{-j}$. As a result, $ \sup_k \frac{\alpha_k}{n_k} = 1 + \sum_{j=1}^\infty (j+1)2^{-j}$, where the final sum is manifestly convergent by the d'Alembert ratio test, and the aforementioned quotient is bounded, as required. We record, by setting $\beta := 1 + \sum_{j=1}^\infty (j+1)2^{-j}$, the estimate $\sup_k \Vert A_k \Vert_{L^\infty([0,t])}^{1/n_k} \leq C_*^{\sup_k \alpha_{k}/n_k}(M_1+M_2) = C_*^\beta (M_1+M_2), $ where we used again that $C_* \geq 2$, and the proof is complete. 
    \end{proof}

\begin{small}


\noindent\textbf{Acknowledgements.} The author thanks Luca Alasio and Nikita Simonov for useful discussions, the hospitality of CRM De Giorgi (SNS Pisa) as part of the Research in Pairs Programme, and the financial support from PEPS JC JC 2026 (Insmi CNRS). 


\end{small}

\end{document}